%% file: toroidal-grids-arxiv2.tex
\documentclass[11pt,a4aper]{article}
\usepackage[paper=a4paper]{geometry}

\usepackage[english]{babel}
\usepackage{mdwlist}
\usepackage{amsrefs}
\usepackage{upgreek}
\usepackage{bbm} 
\usepackage{euscript,amssymb,amsthm,amsmath}
\usepackage{mathrsfs}
\usepackage{amsfonts}
\usepackage{fancybox}
\usepackage{graphicx}
\usepackage{color}
\usepackage{rotating}
\usepackage{graphicx,graphics,epsf}
\usepackage{mathcomp} 
\usepackage{textcomp}

\usepackage{fullpage}
\newtheorem{thm}{Theorem}[section]
\newtheorem{thm*}{Theorem*}
\newtheorem{lemma}[thm]{Lemma} 
\newtheorem{theorem}[thm]{Theorem}

\newtheorem{observation}[thm]{Observation}
\newtheorem{corollary}[thm]{Corollary}
\newtheorem{claim}[thm]{Claim}

\newtheorem{definition}[thm]{Definition}

\theoremstyle{definition}
\newtheorem{remark}[thm]{Remark}

\newcommand{\useInnerQed}{\renewcommand{\qedsymbol}{$\blacksquare$}}

\def\ca#1{{\cal#1}}
\def\sem{\setminus}

\def\len(#1){{\|#1\|}}
\def\fw{{\text{\sl fw}}}
\def\ew{{\text{\sl ew}}}
\def\ewn{{\text{\sl ewn}}}
\def\ewnd{{\text{\sl ewn}^*}}
\def\cutt{/\!\!/}

\def\dee{{\lfloor\Delta(G)/2\rfloor}}
\def\deee{{\lfloor\Delta/2\rfloor}}

\def\ignore#1{{}}

\def\sizealpha{{\len(\alpha)}}

\def\oldh{{z}}

\def\Halpha{{H\cutt\alpha}}

\def\stretch{{\text{\sl Str}}}
\def\stretchd{{\text{\sl Str}^*}}

\def\Tex{{\text{\sl Tex}}}

\def\ee{{\mathcal E}}

\def\cc{{\mathcal C}}
\def\dd{{\mathcal D}}

\def\floor#1{{\lfloor{#1}\rfloor}}

\def\ceil#1{{\lceil{#1}\rceil}}

\def\real{{\mathbb{R}}}

\def\crg{\mathop{\text{\sl cr}}}

\def\mcr{\mathop{\text{\sl mcr}}}

\def\floor#1{{\lfloor{#1}\rfloor}}

\def\ceil#1{{\lceil{#1}\rceil}}

\title{Toroidal Grid Minors and Stretch in Embedded Graphs%
\footnote{This draws upon and extends partial results presented at 
	ISAAC 2007~\cite{HS07} and SODA 2010~\cite{HC10}.}
}

\author{Markus Chimani\thanks{
Faculty of Mathematics/Computer Science, Osnabr\"uck University. Osnabr\"uck, Germany.}
\and
Petr Hlin\v{e}n\'y \thanks{
Faculty of Informatics, Masaryk University. Brno, Czech Republic. 
Supported by the Czech Science Foundation,
projects 14-03501S (until 2016) and 17-00837S.
}
             \and
             Gelasio Salazar\thanks{Instituto de Fisica,
Universidad Autonoma de San Luis Potosi. 
San Luis Potosi, Mexico. Supported by CONACYT Grant 106432.
}}

\date{\today}

\begin{document}

\maketitle

\begin{abstract}

We investigate the {\em toroidal expanse} of an embedded graph $G$,
that is, the size of the
largest toroidal grid contained in $G$ as a minor. 
In the course of this work we introduce a new
embedding density parameter, the {\em stretch} of an embedded graph~$G$,
and use it to bound the toroidal expanse from above and from below
within a constant factor depending only on the genus and the maximum degree.
We also show that these parameters are tightly related to the planar
{\em crossing number} of~$G$.
As a consequence of our bounds, we derive an efficient constant factor
approximation algorithm for the toroidal expanse and for the crossing number
of a surface-embedded graph with bounded maximum degree.
\end{abstract}

\bigskip

{\small

}
\bigskip

\maketitle

\pagebreak

\section{Introduction}\label{sec:intro}

In their development of the Graph Minors theory towards the proof of
Wagner's Conjecture~\cite{RoSeGMXX}, Robertson and Seymour
made extensive use of surface embeddings of graphs. Robertson and
Seymour introduced parameters that measure the density of an
embedding, and established results that are not only central to the Graph
Minors theory, but are also of independent interest. We recall that
the {\em face-width} $\fw(G)$ of a graph $G$ embedded in a surface $\Sigma$ is the
smallest $r$ such that $\Sigma$ contains a noncontractible closed
curve (a {\em loop}) that intersects $G$ in $r$ points.

\begin{theorem}[Robertson and Seymour~\cite{RoSeGMVII}]
\label{thm:fw-minor}
For any graph $H$ embedded on a surface $\Sigma$, there exists a
constant $c:=c_{\Sigma}(H)$ such that every graph $G$ that embeds in $\Sigma$
with face-width at least $c$ contains $H$ as a minor.
\end{theorem}

This theorem, and other related results, 
spurred great interest in understanding which structures are forced by imposing 
density conditions on graph embeddings. 
For instance, Thomassen~\cite{Th94} and Yu~\cite{Yu97} proved the
existence of spanning trees with bounded degree for graphs embedded
with large enough face-width. In the same paper, Yu showed that under
strong enough connectivity conditions, $G$ is Hamiltonian if $G$ is a triangulation.

Large enough density, in the form of edge-width, also guarantees several nice coloring
properties. We recall that the {\em edge-width} $\ew(G)$ of an embedded graph
$G$ is the length of a shortest
noncontractible cycle in $G$. Fisk and Mohar~\cite{FM94} proved that there is a
universal constant $c$ such that every graph $G$ embedded in a surface
of Euler genus $g >0$ with edge-width at least $c\log{g}$ is
$6$-colorable. Thomassen~\cite{Th93} proved that larger (namely
$2^{14g+6}$) edge-width guarantees $5$-colorability. More recently, DeVos,
Kawarabayashi, and Mohar~\cite{DKM08} proved that large enough edge-width
actually guarantees $5$-choosability. 

In a direction closer to our current interest, 
Fiedler et al.~\cite{FHRR95} proved that if $G$ is embedded with
face-width $r$, then it has $\floor{r/2}$
pairwise disjoint contractible cycles, all bounding discs containing a
particular face. 
Brunet, Mohar, and Richter~\cite{BMR96} showed that such a $G$ contains at least
$\floor{(r-1)/2}$ pairwise disjoint, pairwise homotopic,
non-separating (in $\Sigma$) cycles, 
and at least $\floor{(r-1)/8} -1$ pairwise disjoint,
pairwise homotopic, separating, noncontractible cycles. 
We remark that throughout this paper, ``homotopic'' refers to
``freely homotopic'' (that is, not to ``fixed point
homotopic'').

For the particular case in which the host surface is the torus,
Schrijver~\cite{Sc93} 
unveiled a beautiful connection with the geometry
of numbers and proved that $G$ has at least $\floor{3r/4}$ pairwise
disjoint noncontractible cycles, and proved that the factor $3/4$ is
best possible. 

\begin{figure}[tb]
\centering
\includegraphics{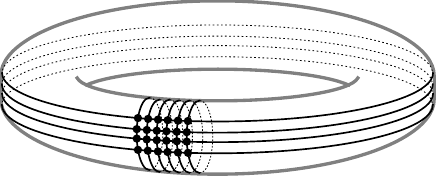}\qquad\qquad\qquad
\includegraphics{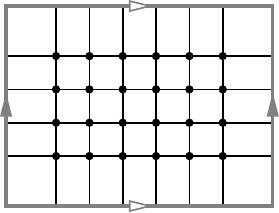}
\caption{Two visualizations of a natural toroidal embedding of the $4\times 6$ toroidal grid.
	On the right, the top and bottom edge of the rectangular frame 
	are identified, and same with the left and right edges.}
\label{fig:torgrid}
\end{figure}

The {\em toroidal $p \times q\>$-grid} is the 
Cartesian product $C_p\Box C_q$ of the cycles of sizes $p$ and $q$.
See Figure~\ref{fig:torgrid}.
Using results and techniques from~\cite{Sc93}, de Graaf and
Schrijver~\cite{dS94} showed the following:
\begin{theorem}[de Graaf and Schrijver \cite{dS94}]
\label{thm:deGS}
Let $G$ be a graph embedded in the torus with face-width $\fw(G)=r\ge 5$.
Then $G$ contains the toroidal $\floor{2r/3} \times \floor{2r/3}\>$-grid as a minor.
\end{theorem}
De Graaf and Schrijver also proved that $\floor{2r/3}$ is best possible, by exhibiting
(for each $r\ge 3$) a graph that embeds in the torus with face-width $r$ and that 
does not contain a toroidal \mbox{$(\floor{2r/3}+1) \times (\floor{2r/3}+1)\>$}-grid as
a minor. 
As they observe, their result shows that $c=\ceil{3m/2}$ is the smallest
value that applies in (Robertson-Seymour's)
Theorem~\ref{thm:fw-minor} for the case of $H=C_m\Box C_m$.

\paragraph*{
		Toroidal expanse, stretch, and crossing number. }

Along the lines of the aforementioned de Graaf-Schrijver
result, our aim is to investigate the largest 
size (meaning the number of vertices) of a toroidal grid minor 
contained in a graph $G$ embedded in an arbitrary orientable surface of
genus greater than zero.
We do not restrict ourselves to square proportions of the grid and
define this parameter as follows.

\begin{definition}[Toroidal expanse]
\label{def:texpanse}
The {\em toroidal expanse} of a graph $G$, denoted by $\Tex(G)$, is the largest value of $p\cdot q$
over all integers $p,q\geq3$ such that $G$ contains a toroidal $p \times q\>$-grid
 as a minor.
If $G$ does not contain $C_3\Box C_3$ as a minor, then let $\Tex(G)=0$.
\end{definition}\noindent
Our interest is both in the structural and the algorithmic aspects of the toroidal expanse.

The ``bound of nontriviality'' $p,q\geq3$ required by
Definition~\ref{def:texpanse} is natural in the view of toroidal
embeddability ---the degenerate cases $C_2\Box C_q$ are planar, while $C_p\Box C_q$
has orientable genus one for all $p,q\geq3$.
It is not difficult to combine results from~\cite{BMR96}
and~\cite{dS94} to show that for each positive integer $g>0$ there is
a constant $c:=c(g)$ with the following property: if $G$ embeds in the
orientable surface $\Sigma$ of genus $g$ with face-width $r$, then
$G$ contains a toroidal $(c\cdot r) \times (c\cdot r)$-grid as a minor; that is,
$\Tex(G) = \Omega(r^2)$.  

On the other hand, it is very easy to come up with a sequence of graphs $G$ embedded
in a fixed surface with face-width $r$ and arbitrarily large $\Tex(G)/ r^2$:
it is achieved by a natural toroidal embedding of $C_r \Box C_{q}$ for arbitrarily
large $q$.  This inadequacy of face-width to estimate the
toroidal expanse of an embedded graph is to be expected, due to the
one-dimensional character of this parameter. 
To this end, we define a new density parameter of embedded graphs
that captures the truly two-dimensional character of our problem;
the {\em stretch of an embedded graph} in Definition~\ref{def:stretch}.
In short, the notion of stretch is related to that of edge-width,
and the stretch equals the smallest product of lengths of two cycles that
transversely meet once on the surface.
The notion of stretch first appeared in the conference paper~\cite{HC10} 
and has also been studied from an algorithmic point of view in~\cite{CCH13}.

Using stretch as a core tool, 
we unveil our main result---a tight two-way relationship between
the toroidal expanse of a graph $G$ in an orientable surface 
and its {\em crossing number} $\crg(G)$ in the plane,
under an assumption of a sufficiently dense embedding.
We furthermore provide an approximation algorithm for both these numbers.
Our treatment of the new concepts of stretch and toroidal expanse
in the paper is completely self-contained.

A simplified summary of the main results follows.

\begin{thm}[Main Theorem]
\label{thm:main-overview}
Let $\Sigma$ be an orientable surface of fixed genus $g>0$,
and let $\Delta$ be an integer.
There exist constants $r_0,c_0,c_1,c_2>0$, depending only on $g$ and $\Delta$,
such that the following two claims hold
for any graph $G$ of maximum degree $\Delta$ embedded in $\Sigma$:
\begin{itemize}
\item[(a)] If $G$ is embedded in $\Sigma$ with face-width at least $r_0$,
then $c_0\cdot\crg(G) \leq \Tex(G) \leq c_1\cdot\crg(G)$.
\item[(b)] There is a polynomial time algorithm that outputs a drawing of
$G$ in the plane with at most $c_2\cdot\crg(G)$ crossings.
\end{itemize}
\end{thm}

The density assumption that $\fw(G)\geq r_0$ is unavoidable for (a). Indeed,
consider a very large planar grid plus an edge. Such a graph
clearly admits a toroidal embedding with face-width~$1$. 
By suitably placing the additional edge, such a graph would
have arbitrarily large crossing number, and yet no $C_3\Box C_3$ minor.
However, one could weaken this restriction a bit by considering ``nonseparating''
face-width instead, as we are going to do in the proof.
On the other hand, an embedding density assumption such as in (a) can be
completely avoided 
for the algorithm in (b) by using additional results of~\cite{CH17}.

Regarding the constants $r_0,c_0,c_1,c_2$ we note that, in our proofs,
\begin{itemize}\parskip0pt
\item $r_0$ is exponential in $g$ (of order $2^g$) and linear in $\Delta$,
\item $1/c_0$ is quadratic in $\Delta$ and exponential in $g$ (of order $8^g$),
\item $c_1$ is independent of $g,\Delta$, and
\item $c_2$ is quartic in $\Delta$ and exponential in $g$ (of order $16^g$).
\end{itemize}
Moreover, the estimate of $c_2$ can be improved to asymptotically match
$1/c_0$ if the density assumption of (a) is fulfilled also in (b).

The rest of this paper is structured as follows. In
Section~\ref{sec:prelims} we present some basic terminology and
results on graph drawings and embeddings, and introduce the key
concept of stretch of an embedded graph. In
Section~\ref{sec:mainresults} we give a commentated walkthrough on the
lemmas and theorems leading to the proof of Theorem~\ref{thm:main-overview}.
The exact values of the constants $r_0,c_0,c_1$ are given there as well.
Some of the presented statements seem to be of
independent interest, and their (often long and technical) proofs are deferred 
to Sections \ref{sec:drawing-upper}\,--\,\ref{sec:finding} of the paper. 
Section~\ref{sec:nodensity} then finishes the algorithmic task
of Theorem~\ref{thm:main-overview}(b) by using
\cite{CH17} to circumvent the density assumption which was crucial in
the previous sections, and gives the value of~$c_2$.
Final Section~\ref{sec:concluding} then outlines some possible extensions
of the main theorem and directions for future research.

\section{Preliminaries}\label{sec:prelims}

We follow standard terminology of topological graph theory, see
Mohar and Thomassen~\cite{MT01} and Stillwell~\cite{St93}.
We deal with undirected multigraphs by default;
so when speaking about a {\em graph}, we allow multiple edges and loops.
The vertex set of a graph $G$ is denoted by $V(G)$, the edge set by
$E(G)$, the number of vertices of $G$ (the {\em size}) by $|G|$,
and the maximum degree by $\Delta(G)$.

In this section we lay out several concepts and basic results relevant to this work,
and introduce the key concept of stretch of an embedded graph.

\subsection{Graph drawings and embeddings in surfaces}
\label{sub:gdes}

We recall that in a {\em drawing} of a graph $G$ in a surface $\Sigma$,
vertices are mapped to distinct points and edges are mapped to continuous curves
(arcs) such that the endpoints of an arc are the vertices of the
corresponding edge; no arc contains a point that
represents a non-incident vertex. For simplicity, we often make no
distinction between the topological objects of a drawing (points and
arcs) and their corresponding graph theoretical objects (vertices and
edges). A {\em crossing} in a drawing is an
intersection point of two edges (or a self-intersection of one edge) 
in a point other than a common endvertex. 
An {\em embedding} of a graph in a surface is a drawing with no edge crossings.

Throughout this paper, we exclusively focus on
orientable surfaces; for each $g\ge 0$ we let $\Sigma_g$ denote
the {\em orientable surface of genus $g$}. 

If we regard an embedded graph $G$ as a subset of its host surface $\Sigma$, 
then the connected components of $\Sigma \setminus G$ are the {\em
faces} of the embedding. 
For clarity, we always assume that our embeddings are {\em cellular},
which means that every face is homeomorphic to an open disc.
For a face $a$ of~$G$, the vertices and edges incident to $a$ form a walk in
the graph $G$, which we call the {\em facial walk} of~$a$.
It is folklore that under the assumption of a cellular embedding 
(and with a restriction to orientable surfaces), 
the set of facial walks of an embedded graph $G$ is fully determined by the
{\em rotation scheme} of $G$, which is the set of cyclic permutations of
edges of $G$ around the vertices of~$G$.

We recall that the vertices of the {\em
  topological dual} $G^*$ of $G$ are the faces of~$G$, and its edges
are the edge-adjacent pairs of faces of~$G$. 
There is a natural one-to-one correspondence between the edges of $G$
and the edges of $G^*$, and so, for an arbitrary $F\subseteq E(G)$,
we denote by $F^*$ the corresponding subset of edges of~$E(G^*)$. We
often use lower case Greek letters (such as $\alpha, \beta, \gamma$)
to denote dual cycles. The rationale behind this practice is the
convenience to regard a dual
cycle as a simple closed curve, often paying no attention to its
graph-theoretical properties.

Let $G$ be a graph embedded in the surface $\Sigma_g$, 
and let $C$ be a surface-non\-separating cycle of $G$. 
We denote by $G\cutt C$ the graph obtained by {\em cutting $G$ through $C$} as
follows. Let $F$ denote the set of edges not in $C$ that are incident
with a vertex in $C$. Orient $C$ arbitrarily, so that $F$ gets
naturally partitioned into the set $L$ of edges to the left of $C$ and
the set $R$ of edges to the right of $C$. 
More formally, $L$ and $R$ should be viewed as sets of half-edges,
since the same one edge of $F$ may have one of its half-edges in $L$ 
and the other in~$R$, but this does not constitute a real problem.
Now contract (topologically) the whole curve representing $C$ to a point-vertex $v$, 
to obtain a pinched surface, and then naturally split $v$ into two
vertices, one incident with the edges in $L$ and another incident with
the edges in $R$.
The resulting graph $G\cutt C$ is thus embedded on a surface $\Sigma'$ such
that $\Sigma$ results from $\Sigma'$ by adding one handle. 
Clearly $E(G\cutt C)=E(G)\sem E(C)$, and so for every subgraph
$F\subseteq G\cutt C$ there
is a unique naturally corresponding subgraph $\hat F\subseteq G$
where $\hat F$ is induced by the edge set $E(\hat F)=E(F)$.
We call $\hat F$ the {\em lift of $F$ into~$G$}.

The ``cutting through'' operation is a form of a standard surface surgery 
in topological graph theory,
and we shall be using it in the dual form too, as follows.
Let $G$ be a graph embedded in a surface $\Sigma$
and $\gamma\subseteq G^*$ a dual cycle
such that $\gamma$ is $\Sigma$-nonseparating. Now cut
the surface along $\gamma$, discarding the set $E'$ of edges of $G$ that are
severed in the process. This yields an embedding of $G-E$ in a surface
with two holes. Then paste two discs, one along the boundary of each
hole, to get back to a compact surface. We denote the resulting
embedding by $G\cutt\gamma$, and say that this is obtained by {\em
  cutting $G$ along $\gamma$}. Note that we may equivalently define
$G\cutt\gamma$ as the embedded graph $(G^*\cutt\gamma)^*$,
that is, $(G\cutt\gamma)^*=G^*\cutt\gamma$. 
Note also that $G\cutt\gamma$ is a spanning subgraph of $G$, 
and that the previous definition of a {\em lift} applies also to this case.

\subsection{Graph crossing number}

We further look at drawings of graphs (in the plane) that allow edge crossings.
To resolve ambiguity, we only consider drawings where
no three edges intersect in a common point other than a vertex. 
The {\em crossing number} $\crg(G)$ of a graph $G$ is then
the minimum number of edge crossings in a drawing of $G$ in the plane.

For the general lower bounds we shall derive on the crossing number of
graphs
we use the following results
on the crossing number of toroidal grids (see~\cites{BeR,JS01,KR,RBe}).

\begin{theorem}
\label{thm:crossing-CpCq}
For all nonnegative integers $p$ and $q$, $\crg(C_p\Box C_q)\geq
\frac12(p-2)q$. Moreover, 
$\crg(C_p\Box C_q) = (p-2)q$ for $p=3,4,5$.
\end{theorem}

We note that this result already yields the easy part of
Theorem~\ref{thm:main-overview}\,(a):
\begin{corollary}
\label{cor:crossing-texp}
Let $G$ be a graph embedded on a surface. Then $\crg(G)\geq\frac1{12}\Tex(G)$.
\end{corollary}
\begin{proof}
Let $q \ge p \ge 3$ be integers that witness $\Tex(G)$ (that is, $G$
contains $C_p\Box C_q$ as a minor, and $\Tex(G)=pq$). 
It is known~\cite{GS01} that if $G$ contains $H$ as a minor, and
$\Delta(H)=4$, then
$\crg(G)\geq\frac14\crg(H)$. 
We apply this bound with $H=C_p\Box C_q$. By
Theorem~\ref{thm:crossing-CpCq}, we then have for $p\in\{3,4,5\}$ that
$\crg(G)\geq\frac14(p-2)q\geq\frac1{12}pq$, and for 
$p\geq6$ we obtain
$\crg(G)\geq\frac14\cdot\frac12(p-2)q\geq\frac1{12}pq$.
\end{proof}

\subsection{Curves on surfaces and embedded cycles}

Note that in an embedded graph, paths are simple curves and
cycles are simple closed curves in the surface, and hence it makes good sense 
to speak about their homotopy.

If $B$ is a path or a cycle of a graph, 
then the {\em length} $\len(B)$ of $B$ is its number of edges.
We recall that the {\em edge-width} $\ew(G)$ of an embedded graph $G$ is the length
of a shortest noncontractible cycle in $G$. 
The {\em nonseparating edge-width} $\ewn(G)$ is the length of a shortest
nonseparating (and hence also noncontractible) cycle in $G$. 
It is trivial to see that the face-width $\fw(G)$ of $G$ equals one 
half of the edge-width of the vertex-face incidence graph of~$G$.
In this paper, we are primarily interested in graphs of bounded degree.
In such a case it is useful to regard $\ew(G^*)$ as a suitable
(easier to deal with) asymptotic replacement for $\fw(G)$:

\begin{lemma}\label{lem:ewdfw}
If $G$ is an embedded graph of maximum degree~$\Delta$, then
$\ew(G^*)\geq\fw(G)\geq\frac{\ew(G^*)}{\lfloor\Delta(G)/2\rfloor}$.
The same inequalities hold for nonseparating edge-width and face-width.
\end{lemma}
\begin{proof}
$\ew(G^*)\geq\fw(G)$ follows since any dual cycle $\alpha$ in $G^*$ makes a
loop intersecting $G$ in $\len(\alpha)$ points.
On the other hand, any loop $\lambda$ intersecting $G$ in $\fw(G)$
points can be locally modified to a homotopic loop $\lambda'$ which does not
contain vertices of $G$, at the cost of intersecting at most
$\lfloor\Delta(G)/2\rfloor$ new edges for every vertex of $G$ on~$\lambda$.
Since $\lambda'$ corresponds to a dual cycle in $G^*$, we conclude that
$\ew(G^*)\leq\fw(G)\cdot\lfloor\Delta(G)/2\rfloor$.
\end{proof}

For a cycle (or an arbitrary subgraph) 
$C$ in a graph $G$, we call a path $P\subset G$ a {\em$C$-ear} 
if the ends $r,s$ of $P$ belong to $C$, but the rest of $P$ is disjoint from $C$.
We allow $r=s$, i.e., a $C$-ear can also be a cycle.
A $C$-ear $P$ is a \mbox{\em$C$-switching ear} (with respect to an orientable
embedding of $G$) if the two edges of $P$ 
incident with the ends $r,s$ are embedded on opposite sides of~$C$.
The following simple technical claim is useful.

\begin{lemma}\label{lem:kl2}
If $C$ is a nonseparating cycle in an
embedded graph $G$ of length $\len(C)=\ewn(G)$,
then all $C$-switching ears in $G$ have length at least $\frac12\ewn(G)$.
\end{lemma}
\begin{proof}
Seeking a contradiction, we suppose that there is a
$C$-switching ear $D$ of length $<\frac12\ewn(G)$.
The ends of $D$ on $C$ determine two subpaths
$C_1,C_2\subseteq C$ (with the same ends as $D$), labeled 
so that $\len(C_1) \le \len(C_2)$.
Then $D\cup C_1$ is a nonseparating cycle, as witnessed by $C_2$.
Since $\len(C_1)\leq\frac12\len(C)$, we have
$$\len(D\cup C_1)\leq\len(D)+\frac12\len(C)
 <\biggl(\frac12+\frac12\biggr)\len(C)=\ewn(G)\,,$$
a contradiction.
\end{proof}

Even though surface surgery can drastically decrease (and also increase, of
course) the edge-width of an embedded graph in general, 
we now prove that this is not the case if we cut through a short cycle
(later, in Lemma~\ref{lem:cutdew}, we shall establish a surprisingly 
powerful extension of this simple claim).

\begin{lemma}
\label{lem:dew2}
Let $G$ be a graph embedded in the surface $\Sigma_g$ of genus $g\geq2$,
and let $C$ be a nonseparating cycle in $G$ of length $\len(C)=\ewn(G)$.
Then $\ewn(G\cutt C)\geq\frac12\ewn(G)$.
\end{lemma}
\begin{proof}
Let $c_1,c_2$ be the two vertices of $G\cutt C$ that result from cutting through $C$,
i.e., $\{c_1,c_2\}=V(G\cutt C)\sem V(G)$.
Let $D\subseteq G\cutt C$ be a nonseparating cycle of length
$\ewn(G\cutt C)$.
If $D$ avoids both $c_1,c_2$, then its lift $\hat D$ in $G$
is a nonseparating cycle again, and so $\ewn(G)\leq\len(D)=\ewn(G\cutt
C)$.
If $D$ hits both $c_1,c_2$ and $P\subseteq D$
is (any) one of the two subpaths with the ends $c_1,c_2$, then the lift $\hat P$
is a $C$-switching ear in $G$.
Thus, by Lemma~\ref{lem:kl2}, 
$$\ewn(G\cutt C)=\len(D)\geq\len(\hat P)\geq\frac12\ewn(G)\,.$$

In the remaining case $D$, up to symmetry, hits $c_1$ and avoids~$c_2$.
Then its lift $\hat D$ is a $C$-ear in $G$.
If $\hat D$ itself is a cycle, then we are done as above.
Otherwise, $\hat D\cup C\subseteq G$ is the
union of three nontrivial internally disjoint paths with common ends, forming exactly three cycles
$A_1,A_2,A_3\subseteq\hat D\cup C$.
Since $D$ is nonseparating in $G\cutt C$, each of $A_1,A_2,A_3$ is nonseparating
in~$G$, and hence $\len(A_i)\geq\ewn(G)$ for $i=1,2,3$.
Since every edge of $\hat D\cup C$ is in exactly two of $A_1,A_2,A_3$,
we have
$\len(A_1)+\len(A_2)+\len(A_3)=2\len(C)+2\len(\hat D)
	=2\ewn(G)+2\len(\hat D)$ and $\len(A_1)+\len(A_2)+\len(A_3)\geq3\ewn(G)$,
from which we get 
\begin{equation}\ewn(G\cutt C)=\len(D)=\len(\hat D)\geq\frac12\ewn(G)\,.\tag*{\qedhere}\end{equation}
\end{proof}

Many arguments in our paper exploit the mutual position of two graph
cycles in a surface.
In topology, the {\em geometric intersection number}%
\footnote{%
Note that this quantity is also called the ``crossing number'' of the
curves, and a pair of curves may be said to be ``$k$-crossing''.
Such a terminology would, however, conflict with the graph crossing number,
and we have to avoid it. Following~\cite{HC10}, we thus use the term ``$k$-leaping'', instead.
}
$i(\alpha,\beta)$ of two (simple) closed curves
$\alpha,\beta$ in a surface is defined as $\min\{ \alpha'\cap \beta'\}$, where the minimum is
taken over all pairs $(\alpha',\beta')$ such that $\alpha'$
(respectively, $\beta'$) is homotopic to $\alpha$ (respectively, $\beta$). 
For our purposes, however, we prefer the following slightly adjusted discrete
view of this concept. 

Let $A\not=B$ be cycles of a graph embedded in a surface $\Sigma$.
Let $P\subseteq A\cap B$ be a connected component of the graph intersection
$A\cap B$ (a path or a single vertex), and let $f_A,f_A'\in E(A)$
(respectively, $f_B,f_B'\in E(B)$) be the edges immediately preceding and succeeding 
$P$ in $A$ (respectively, $B$).
See Figure~\ref{fig:defineleap}. 
Then $P$ is called a {\em leap of A,B} if there is a sufficiently
small open neighborhood $\Omega$ of $P$ in $\Sigma$ such that the mentioned
edges meet the boundary of $\Omega$ in this cyclic order;
$f_A,f_B,f_A',f_B'$
(i.e., $A$ and $B$ {\em meet transversely in~$P$}).
Note that $A\cap B$ may contain other components besides $P$ that are
not leaps.

\begin{figure}[t]
\centering
\includegraphics{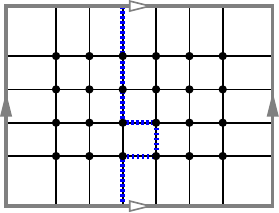}
\hfill
\includegraphics{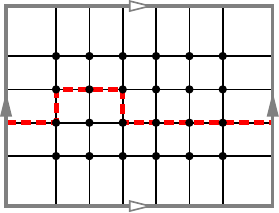}
\hfill
\includegraphics{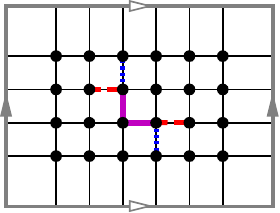}
\caption{A toroidal embedding of $C_4\square C_6$. The left and the center
pictures show two cycles $A$ and $B$ in thick dashed lines. The
intersection of $A$ and $B$ is the $2$-edge path emphasized in the right picture 
with a thick solid line. This path is a leap of $A$ and $B$.}
\label{fig:defineleap}
\end{figure}

\begin{definition}[$k$-leaping]
\label{def:leaping}
Two cycles $A,B$ of an embedded graph are in a {\em$k$-leap position}
(or simply {\em$k$-leaping}),
if their intersection $A\cap B$ has exactly $k$ connected components that
are leaps of $A,B$. If $k$ is odd, then we say that $A,B$ are in an
{\em odd-leap position}.
\end{definition}

We now observe some basic properties of the $k$-leap concept:
\begin{itemize}
\item
If $A,B$ are in an odd-leap position, then necessarily each of $A,B$ is
noncontractible and nonseparating.
\item
It is not always true that $A,B$ in a $k$-leap position have geometric
intersection number exactly $k$, but the parity of the two numbers is preserved.
Particularly, $A,B$ are in an odd-leap position if and only if 
their geometric intersection number is odd.
(We will not directly use this fact herein, though.)
\item
We will later prove (Lemma~\ref{lem:3pp})
that the set of embedded cycles that are odd-leaping a given cycle $A$
satisfies the useful {\em$3$-path condition} (cf.\ \cite[Section~4.3]{MT01}).
\end{itemize}

\subsection{Stretch of an embedded graph}

In the quest for another embedding density parameter suitable
for capturing the two-dimensional character of the toroidal expanse
and crossing number problems,
we put forward the following concept improving upon the original 
``orthogonal width'' of \cite{HS07}.

\begin{definition}[Stretch]
\label{def:stretch}
Let $G$ be a graph embedded in an orientable surface $\Sigma$.
The {\em stretch $\stretch(G)$ of $G$} 
is the minimum value of $\len(A)\cdot\len(B)$ over all pairs of cycles
$A,B\subseteq G$ that are in a one-leap position in~$\Sigma$.
\end{definition}
We remark in passing that although our paper does not use nor provide 
an algorithm to compute the stretch of an embedding, this can be done
efficiently on any surface by~\cite{CCH13}.

As we noted above, if $A,B$ are in an odd-leap position, then both $A$
and $B$ are noncontractible and nonseparating. Thus it follows that $\stretch(G)\geq\ewn(G)^2$.
We postulate that stretch is a natural two-dimensional analogue of edge-width,
a well-known and often used embedding density parameter.
Actually, one may argue that the dual edge-width is a more suitable
parameter to measure the density of an
embedding, and so we shall mostly deal with {\em dual
stretch}---the stretch of the topological dual $G^*$---later in this paper 
(starting at Lemma~\ref{lem:cr-stretch-torus} and Section~\ref{sec:mainresults}).
Analogously to face-width, one can also define the {\em face stretch} of $G$ as
one quarter of the stretch of the vertex-face incidence graph of $G$,
and this concept is to be briefly discussed in the last
Section~\ref{sec:concluding}.

We now prove several simple basic facts about the stretch of an embedded
graph, which we shall use later. We start with an easy observation.
\begin{lemma}
\label{lem:thstr}
If $C$ is a nonseparating cycle in an embedded graph $G$,
and $P$ is a $C$-switching ear in $G$,
then $\stretch(G)\leq\len(C)\cdot\big(\len(P)+\frac12\len(C)\big)$.
If, moreover, $\len(C)=\ewn(G)$ then $\stretch(G)\leq2\len(C)\cdot\len(P)$.
\end{lemma}
\begin{proof}
The ends of $P$ partition $C$ into two paths $C_1,C_2\subseteq C$,
which we label so that $\len(C_1) \le \len(C_2)$. (In a degenerate case,
$C_1$ can be a single vertex). 
Thus
$\len(C_1)\leq\frac12\len(C)$.
Since $C$ and $P\cup C_1$ are in a one-leap position, we have
$\stretch(G)\leq\len(C)\cdot(\len(P)+\len(C_1))$, as claimed.
In the case of $\len(C)=\ewn(G)$, Lemma~\ref{lem:kl2} furthermore implies
$\stretch(G)\leq\len(C)\cdot(\len(P)+\frac12\len(C))\leq\len(C)\cdot2\len(P)$.
\end{proof}

A tight relation of stretch to the topic of our paper can be illustrated 
by the following two claims regarding graphs on the torus.
While they are not directly used in our paper, we believe that they
may be found interesting by the readers.

\begin{lemma}
\label{lem:cr-stretch-torus}
If $G$ is a graph embedded in the torus, then $\crg(G)\leq\stretch(G^*)$.
\end{lemma}
\begin{proof}
Let $\alpha,\beta\subseteq G^*$ be a pair of dual cycles witnessing
$\stretch(G^*)$, and let $K:=E(\alpha)^*$, $L:=E(\beta)^*\sem K$,
and $M:=E(\alpha\cap \beta)^*$.
Note that $K,L$, and $M$ are edge sets in~$G$.
Then, by cutting $G$ along $\alpha$, we obtain a plane (cylindrical) embedding
$G_0$ of $G-K$.
It is natural to draw the edges of $K$ into $G_0$
in one parallel ``bunch'' along the fragment of $\beta$ such that they
cross only with edges of $L$ and $M\subseteq K$ (indeed, crossings between
edges of $K$ are necessary when $M\not=\emptyset$),
thus getting a drawing of $G$ in the plane.
See Figure~\ref{fig:addedges}.
The total number of crossings in this particular drawing, and thus the crossing
number of $G$, is at most $|K|\cdot|L|+|K|\cdot|M|=
|K|\cdot(|L|+|M|)=\len(\alpha)\cdot\len(\beta)=\stretch(G^*)$.
\end{proof}

\begin{figure}[t]
\hbox to \hsize{%
\raise18mm\hbox{$\beta$}
\includegraphics{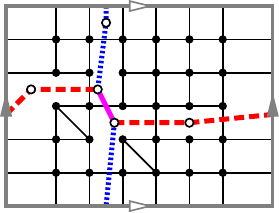}
\hfill
\includegraphics{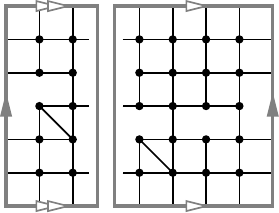}
\hfill
\includegraphics{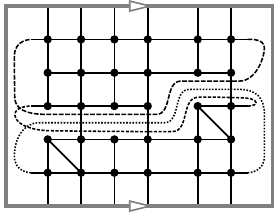}
\raise20mm\hbox{$K$}
}\hbox{\hspace*{21mm}$\alpha$\hspace{61mm}$G_0$}
\caption{An illustration of the proof of Lemma~\ref{lem:cr-stretch-torus}.
  The left picture shows a graph $G$ embedded in the torus (black vertices and
  thin solid edges), together with dual cycles $\alpha,\beta$ witnessing the
  dual stretch (white vertices and thick dashed edges).
  The thick dual edge is common to $\alpha$ and $\beta$.
  We let $K$ denote the set of three edges in $G$ that correspond to
  the edges of $\alpha$. In the center picture, we have cut the torus along the curve
  defined by $\alpha$, to obtain a cylindrical embedding of $G_0:=G-K$.
  In the right picture, we start with the 
  same embedding of $G_0$ as in the center---we have simply identified the black
  arrows; the three severed edges of $K$ can be drawn along the remaining fragment
  of $\beta$, to get a cylindrical drawing of $G$.
  Notice that the bunch of edges of $K$ should follow the whole fragment of
  $\beta$, including the section common to $\alpha$ and $\beta$---this is to
  maintain the right order of edges in~$K$.
  Although not always being optimal, such a solution is very simple.}
\label{fig:addedges}
\end{figure}

\begin{corollary}
\label{cor:texp-stretch-torus}
If $G$ is a graph embedded in the torus, then $\Tex(G)\leq12\stretch(G^*)$.
\end{corollary}
\begin{proof}
This follows immediately using Corollary~\ref{cor:crossing-texp}.
\end{proof}

We finish this section by proving an analogue of Lemma~\ref{lem:dew2}
for the stretch of an embedded graph, showing that this parameter cannot decrease too much if we cut the embedding through a short cycle.
This will be important to us since cutting through handles of embedded graphs
will be our main inductive tool in the proofs of lower bounds on $\crg(G)$
and $\Tex(G)$.

\begin{lemma}
\label{lem:str4}
Let $G$ be a graph embedded in the surface $\Sigma_g$ of genus $g\geq2$,
and let $C$ be a nonseparating cycle in $G$ of length $\len(C)=\ewn(G)$.
Then $\stretch(G\cutt C)\geq\frac14\stretch(G)$.
\end{lemma}

\begin{proof}
Let $c_1,c_2$ be the two vertices of $G\cutt C$ that result from cutting through $C$,
i.e., $\{c_1,c_2\}=V(G\cutt C)\sem V(G)$.
Suppose that $\stretch(G\cutt C)=ab$ is attained by a pair of one-leaping cycles
$A,B$ in~$G\cutt C$, with $a=\len(A)$ and $b=\len(B)$.
Our goal is to show that $\stretch(G)\leq4ab$.
Using Lemma~\ref{lem:dew2} and the fact that both $A,B$ are nonseparating, we get 
\begin{equation}\label{eq:abrho}
a,b\geq\ewn(G\cutt C)\geq\frac12\ewn(G)=\frac12\len(C)
.\end{equation}

Suppose first that both $c_1,c_2\in V(A\cup B)$.
Then there exists a path $P\subseteq A\cup B$ connecting $c_1$ to
$c_2$ such that $\len(P)\leq\frac12(a+b)$.
Clearly, its lift $\hat P$ is a $C$-switching ear in $G$,
and so by Lemma~\ref{lem:thstr} and \eqref{eq:abrho},
\begin{align*}
\stretch(G)\;&\leq\; \len(C)\cdot\big(\len(\hat P)+\frac12\len(C)\big)
 \leq\len(C)\cdot\frac12(a+b+\len(C))
\\
  &\leq\; \frac12(2ba+2ab+4ab)=4ab=4\,\stretch(G\cutt C)
.\end{align*}

Otherwise, up to symmetry, $c_2\not\in V(A\cup B)$
but possibly $c_1\in V(A\cup B)$.
The lift $\hat A$ of $A$ in $G$ is a $C$-ear in the case $c_1\in V(A)$,
and $\hat A$ is a cycle otherwise.
The same holds for~$B$.
We define $\bar A$ to be $\hat A$ if $\hat A$ is a cycle, and 
otherwise $\bar A=\hat A\cup C_A$ where $C_A\subseteq C$ is
a shortest subpath with the same ends in $C$ as $\hat A$.
We define $\bar B$ and possibly $C_B$ analogously.
We prove by a simple case-analysis that $\bar A,\bar B$ form a one-leaping pair in $G$:
consider $P$ a connected component of $A\cap B$ (as in
Definition~\ref{def:leaping}).
The goal is to show that $P$ is a leap of $A,B$ if and only if a component of
$\bar A\cap\bar B$ corresponding to $P$ in $G$ is a leap of $\bar A,\bar B$.
If $c_1\not\in V(P)$, then a small neighborhood of $P$ in the embedding 
$G\cutt C$ is the same as in $G$, and so $P$ is a leap of $A,B$ iff $\hat P$ 
is a leap of $\bar A,\bar B$.
If $c_1$ is an internal vertex of the path $P$, then $C_A=C_B$ and again
$P$ is a leap of $A,B$ iff $\hat P\cup C_A$ is a leap of $\bar A,\bar B$.
Suppose that $c_1$ is an end of~$P$.
It might happen that $C_A\cap C_B=\emptyset$ if $P$ is a single vertex which
is {\em not} a leap of $A,B$.
Otherwise, $\bar P:=(C_A\cap C_B)\cup\hat P$ is a component of $\bar A\cap\bar B$.
Comparing a small neighborhood of $\bar P$ in $G$ to a small neighborhood of
$P$ which results by contracting $C_A\cap C_B$, we again see that
$P$ is a leap of $A,B$ iff $\bar P$ is a leap of $\bar A,\bar B$.

Since $\bar A,\bar B$ form a one-leaping pair in $G$, we conclude
with help of \eqref{eq:abrho},
\begin{align*}
\stretch(G)\;&\leq\; \len(\bar A)\cdot\len(\bar B)\leq (a+\frac12\len(C))\cdot(b+\frac12\len(C))
\\  &\leq\; (a+a)\cdot(b+b)=4ab=4\,\stretch(G\cutt C).\tag*{\qedhere}
\end{align*}
\smallskip
\end{proof}

\section{Breakdown of the proof of Theorem~\ref{thm:main-overview}}\label{sec:mainresults}

In this section we shall state the results leading to the proof of
Theorem~\ref{thm:main-overview}, which is given in Section~\ref{sub:mainproof}.
The proofs of (most of) these
statements are long and technical, and so they are deferred to the later
sections of the paper.

We start by telling the overall (and so far only rough) ``big picture'' of our arguments.
Here we use the following notation.
For functions $f,g$ we write $f(x)\preceq_{c}g(x)$ if, for all given $x$, it holds
$f(x)\leq K_c\cdot g(x)$ where $K_c$ is a constant depending on~$c$.
Then, for any integers $g,\Delta$ and every graph $G$ of maximum degree
$\Delta$ and with a sufficiently dense embedding in $\Sigma_g$, 
we show the following chain of estimates
\begin{align}\label{eq:bigpicture}
\stretch(G^*) ~\preceq_{g,\Delta}~ \Tex(G) 
	~\preceq_1~ \crg(G) ~\preceq_g~ \text{\sl Pcost}(G^*)
	~\preceq_g~ \stretch(G_1^*)
\end{align}
where $\text{\sl Pcost}(G^*)$ is the {\em cost} of some planarizing sequence
of $G^*$---as defined further in Definition~\ref{def:good-planarizing},
and $G_1$ is a suitable subgraph of~$G$ which again has a sufficiently dense
embedding in some surface 
(an embedding derived from $G$, but not necessarily in $\Sigma_g$).

Since the chain \eqref{eq:bigpicture} can be ``closed'' by implied
$\stretch(G_1^*) \preceq_{g,\Delta} \Tex(G_1) \leq \Tex(G)$, these
estimates will immediately lead to a
full proof Theorem~\ref{thm:main-overview}(a) in Theorem~\ref{thm:main-parta}.
Moreover, since $\text{\sl Pcost}(G^*)$ can be efficiently computed, this
also provides an approximation algorithm for the other quantities 
in Theorem~\ref{thm:main-dense}.

Note also the role of the subgraph $G_1$ in \eqref{eq:bigpicture}:
for instance, a graph $G$ embedded in the double torus could
have a large toroidal grid living on one of the handles, and yet 
small dual stretch due to a very small dual edge width on the other handle.
This shows that taking a suitable subembedding $G_1$
(the one which exhibits a large value of stretch in the dual)
instead of $G$ itself at the end of the chain \eqref{eq:bigpicture} 
is necessary for the claim to hold.

In the rest of the paper we prove the claimed estimates from \eqref{eq:bigpicture} in order.

\smallskip
Since we will frequently deal with dual graphs in our arguments,
we introduce several conventions in order to help comprehension.
When we add an adjective {\em dual\/} to a graph term, we mean this term in
the topological dual of the (currently considered) graph.
We will denote the faces of an embedded graph $G$ using lowercase letters,
treating
them as vertices of its dual $G^*$. As we already mentioned in Section~\ref{sub:gdes},
we use lowercase Greek letters to refer to subgraphs (cycles or paths)
of~$G^*$, and
when there is no danger of confusion, we do not formally distinguish
between a graph and its embedding.
In particular, if $\alpha\subseteq G^*$ is a dual cycle, then $\alpha$ also refers
to the loop on the surface determined by the embedding~$G^*$.
Finally, we will denote by $\ewnd(G):=\ewn(G^*)$ the nonseparating edge-width
of the dual~$G^*$ of~$G$, and by $\stretchd(G):=\stretch(G^*)$
the dual stretch of~$G$.

\subsection{Estimating the toroidal expanse}
\label{sub:toridalexp}

Recall that we have already seen the relation 
$\Tex(G)\preceq_1\!\crg(G)$ in Corollary~\ref{cor:crossing-texp}.
In this section we finish the left-hand side of \eqref{eq:bigpicture},
namely the estimate ``$\stretch(G^*)\preceq_{g,\Delta}\Tex(G)\preceq_1\crg(G)$''.

We first give some basic lower bound estimates for the toroidal
expanse of graphs in the torus. These estimates ultimately rely on
the following basic result, which appears to be of independent
interest. Loosely speaking, it states that if a graph has two
collections of cycles that mimic the topological properties of the
cycles that build up a $p\times q$-toroidal grid, then the graph does
contain such a grid as a minor. We say that a pair $(C,D)$ of
curves in the torus is a {\em basis} (for the fundamental group) if
there are no integers $m,n$ such that $C^m$ is homotopic to $D^n$.

\begin{thm}
\label{thm:two-cycle-families}
Let $G$ be a graph embedded in the torus.
Suppose that $G$ contains a collection $\{C_1,\dots,C_p\}$ of $p\geq3$ pairwise disjoint, pairwise
homotopic cycles, and a collection
$\{D_1,\dots,D_q\}$ of $q\geq3$ pairwise
disjoint, pairwise homotopic cycles. Further suppose that the pair
$(C_1,D_1)$ is a basis.
Then $G$ contains a $p\times q$-toroidal grid as a minor.
\end{thm}

We prove this statement in Section~\ref{sec:grids}.

In the torus, $\ewn(G)=\ew(G)$ and so by Lemma~\ref{lem:ewdfw} we have
$\fw(G)\geq\frac{\ewnd(G)}{\lfloor\Delta(G)/2\rfloor}$.
Hence, for instance, one can formulate Theorem~\ref{thm:deGS} in terms of
nonseparating dual edge-width.
Along these lines we shall derive the following as a consequence of 
Theorem~\ref{thm:two-cycle-families}; its proof is also in Section~\ref{sec:grids}:

\begin{thm}
\label{thm:agrid-torus}
Let $G$ be a graph embedded in the torus and
assume $k:=\ewnd(G)\geq5\lfloor\Delta(G)/2\rfloor$.
If there exists a dual cycle $\alpha\subseteq G^*$ of length $k$ 
such that a shortest $\alpha$-switching dual ear has length~$\ell$
(recall from Lemma~\ref{lem:kl2} that $\ell\geq k/2$),
then $G$ contains as a minor the toroidal grid of size
$$
\left\lceil\frac \ell{\lfloor\Delta(G)/2\rfloor}\right\rceil
		\>\times\>
	\left\lfloor\frac23\left\lceil
		\frac k{\lfloor\Delta(G)/2\rfloor}
        \right\rceil\right\rfloor
\,.\smallskip$$
\end{thm}

Hence the toroidal expanse of $G$ is at least 
$\big\lceil\frac \ell{\lfloor\Delta(G)/2\rfloor}\big\rceil
\cdot\big\lfloor\frac23\lceil\frac k{\lfloor\Delta(G)/2\rfloor}\rceil\big\rfloor$.
On the other hand, 
by Lemma~\ref{lem:ewdfw} and Theorem~\ref{thm:deGS} it follows that
the toroidal expanse of $G$ is at least 
$\big\lfloor\frac23\big\lceil\frac
k{\lfloor\Delta(G)/2\rfloor}\big\rceil\big\rfloor^2$. Therefore
our estimate becomes useful roughly whenever $\ell>\frac23k$. Now by
Lemma~\ref{lem:thstr} (applied to~$G^*$), we have $\stretch^*(G) \le
k\cdot(\ell + k/2)$, and so  $\ell>\frac23k$ whenever $\stretch^*(G)
> \frac{7}{6}k^2$. 

Moreover, Theorem~\ref{thm:agrid-torus} can be reformulated 
in terms of $\stretch^*(G)$ (instead of ``$\ell\cdot k$''). This
reformulation is important for the general estimate on the toroidal
expanse of $G$:

\begin{corollary}
\label{cor:agrid-torus}
Let $G$ be a graph embedded in the torus with
$\ewnd(G)\geq5\lfloor\Delta(G)/2\rfloor$. Then
$$
\Tex(G)\>\geq\> \frac{2}{7}\,
	\big\lfloor{\Delta(G)}/2\big\rfloor^{-2} \cdot\stretchd(G)
	\>\geq\> \frac87\Delta(G)^{-2} \cdot\stretchd(G)
\,.$$
Furthermore, for any $\varepsilon>0$ there is a
$k_0:=k_0(\Delta,\varepsilon)$ such that if  $\ewnd(G)> k_0$, then
$\Tex(G)\geq(\frac8{21}-\varepsilon)\cdot
	\lfloor{\Delta(G)}/2\rfloor^{-2}\cdot\stretchd(G)$.
\end{corollary}

For the proof of this statement, we again refer to Section~\ref{sec:grids}.

\smallskip
Stepping up to orientable surfaces of genus $g>1$, we 
can now easily derive the general estimate
``$\stretch(G^*)\preceq_{g,\Delta}\Tex(G)$'' of \eqref{eq:bigpicture}
from the previous results:
\begin{corollary}
\label{cor:agrid-all}
Let $G$ be a graph embedded in the surface $\Sigma_g$,
such that $\ewnd(G)\geq5\cdot2^{g-1}\lfloor\Delta(G)/2\rfloor$.
Then 
\begin{equation}
\Tex(G)\>\geq\> \frac{1}{7}\,2^{3-2g}
	\big\lfloor{\Delta(G)}/2\big\rfloor^{-2} \cdot\stretchd(G)
\,.\end{equation}
\end{corollary}
\begin{proof}
We proceed by a simple induction on~$g\geq1$.
The base case of $g=1$ is done in Corollary~\ref{cor:agrid-torus}.
Assume now some $g>1$.
Let $\alpha$ be any nonseparating dual cycle in $G^*$ of length $\ewnd(G)$,
and let $G':=G\cutt\alpha$ embedded in $\Sigma_{g-1}$.
Since $\ewnd(G')\geq\frac12\ewnd(G)\geq
	5\cdot2^{g-2}\lfloor\Delta(G)/2\rfloor$ by
Lemma~\ref{lem:dew2}, from the induction assumption we get
$$\Tex(G')\>\geq\> \frac{1}{7}\,2^{3-2(g-1)}
        \big\lfloor{\Delta(G)}/2\big\rfloor^{-2} \cdot\stretchd(G')
= \left( \frac{1}{7}\,2^{3-2g}
        \big\lfloor{\Delta(G)}/2\big\rfloor^{-2}\right) \cdot 4\,\stretchd(G')
.$$
To finish it remains to observe that $\Tex(G)\geq\Tex(G')$ since
$G'\subseteq G$, and that $\stretchd(G)\leq4\,\stretchd(G')$ by
Lemma~\ref{lem:str4}.
\end{proof}

\subsection{Algorithmic upper estimate for higher surfaces}

It remains to tackle the right-hand side of the chain \eqref{eq:bigpicture},
that is, to argue that
``$\crg(G)\preceq_g\text{\sl Pcost}(G^*)\preceq_g\stretch(G_1^*)$''
in any fixed genus $g\geq1$.
We start with explaining the term $\text{\sl Pcost}(G^*)$, 
which refers to planarizing an embedded graph, and its historical relations.

Peter Brass conjectured
the existence of a constant $c$ such that the crossing number of a
toroidal graph on $n$ vertices is at most $c\Delta n$. This conjecture 
was proved by Pach and T\'oth~\cite{pachtoth}. Moreover, Pach and
T\'oth showed that for every orientable surface $\Sigma$ there is a constant
$c_\Sigma$ such that the crossing number of an $n$-vertex graph embeddable on
$\Sigma$ is at most $c_\Sigma \Delta n$; this result was extended to any
surface by B{\"o}r{\"o}czky, Pach, and T\'oth~\cite{BPT06}. The
constant $c_\Sigma$ proved in these papers is exponential in the
genus of $\Sigma$. This
was later refined by Djidjev
and Vrt'o~\cite{DV12}, who decreased the bound to $\ca O(g\Delta n)$, and
proved that this is tight within a constant factor.

At the heart of these results lies the technique of (perhaps
recursively) cutting along a suitable 
{\em planarizing} subgraph (most naturally, a set of short cycles),
and then redrawing the missing edges without introducing too many
crossings. Our techniques and aims are of a similar spirit, although
our cutting process is more delicate, due to our need to
(eventually) find
a matching lower bound for the number of crossings in the resulting
drawing. Our cutting paradigm is formalized in the following definition.

\begin{definition}[Good planarizing sequence]
\label{def:good-planarizing}
Let $G$ be a graph embedded in the surface $\Sigma_g$.
A~sequence $(G_1,C_1),(G_2,C_2), \dots, (G_g,C_g)$ 
is called a {\em good planarizing sequence for $G$} if the following holds
for $i=1,\dots,g$, letting $G_0=G$:
\begin{itemize}
\item $G_i$ is a graph embedded in $\Sigma_{g-i}$,
\item $C_i$ is a nonseparating cycle in $G_{i-1}$ of length $\ewn(G_{i-1})$, and
\item $G_i$ results by cutting the embedding $G_{i-1}$ through $C_i$.
\end{itemize}
We associate $G$ and its planarizing sequence with the values
$\{k_i,\ell_i\}_{i=1,\ldots,g}$, where
$k_i=\len(C_i)$ and $\ell_i$ is the length of a shortest $C_i$-switching
ear in $G_{i-1}$, for $i=1,\dots,g$.
Then we may shortly denote by $\text{\sl Pcost}(G):=\max\{k_i\cdot\ell_i\}_{i=1,2,\ldots,g}$,
implicitly referring to the considered planarizing sequence.
\end{definition}

Good planarizing sequences in the dual graph can be used to provide
the estimate ``$\crg(G)\preceq_g\text{\sl Pcost}(G^*)$'' of \eqref{eq:bigpicture},
as stated precisely in the following theorem.
In regard to algorithmic aspects and runtime complexity, we emphasise that we
expect the embedded input graph $G$ to be represented by its rotation scheme,
and the output drawing to be represented by a planar graph obtained by
replacing each crossing with a new (specially marked) subdividing vertex.

\begin{theorem}
\label{thm:upper-cr}
Let $G$ be a graph embedded in $\Sigma_g$.
Let $(G_1^*,\gamma_1), \dots, (G_g^*,\gamma_g)$ be any good planarizing sequence
for the topological dual $G^*$ with associated lengths
$\{k_i,\ell_i\}_{i=1,\ldots,g}$ (Definition~\ref{def:good-planarizing}).
Then 
\begin{equation}\label{eq:uppercr}
\crg(G)\>\leq\> 3\cdot\left(2^{g+1}-2-g\right)\cdot\text{\sl Pcost}(G^*)
	= 3\cdot\left(2^{g+1}-2-g\right)\cdot\max
		 \{k_i\cdot\ell_i\}_{i=1,2,\ldots,g}
\,.\end{equation}
Furthermore, there is an algorithm that, for some good planarizing sequence
$(G_1^*,\gamma_1), \dots, (G_g^*,\gamma_g)$ of $G^*$,
produces a drawing of $G$ in the plane with at most the number of
crossings claimed in~\eqref{eq:uppercr},
and such that the subgraph $G-E(\gamma_1)\cup\dots\cup E(\gamma_g)$
(i.e., $G$ without the edges severed by this planarizing sequence)
is drawn planarly within it.
This algorithm runs in time $\ca O\big(n(\log\log n+\Delta^3)\big)$ for
fixed~$g$, with $n=|V(G)|$ and $\Delta=\Delta(G)$.
\end{theorem}

We remark that the dependence of the algorithm's runtime on $\Delta$
(which is anyway assumed bounded in our main theorems)
is necessary in Theorem~\ref{thm:upper-cr} due to the input size of $G$ and,
more importantly, due to the potential size of the output drawing.
Besides that, the only reason to have superlinear time complexity
$\ca O(n\log\log n)$ with fixed $\Delta$ is a subroutine for computing 
a shortest nonseparating cycle in graphs embedded in an orientable surface.
Strictly saying, our algorithm is also FPT with respect to the genus $g$ as a
parameter, but it can be run in overall polynomial time as well
(if the embedding of $G$ is given).
The proof of this theorem is given in Section~\ref{sec:drawing-upper}.

\subsection{Bridging the approximation gap}

Let us briefly revise where we stand now with respect to the big picture
given in \eqref{eq:bigpicture}.
We have already proved all the inequalities of it except the last one
``$\text{\sl Pcost}(G^*)\preceq_g\stretch(G_1^*)$''.
It may appear that our next task is to bridge the gap by simply proving that
$\stretchd(G) = \Omega(\text{\sl Pcost}(G^*))$.
Unfortunately, no such statement is true in general.
We need to find a way around this difficulty, namely, by restricting
to a suitable subgraph of~$G$. 
The following key technical claim gets us closely to the desired estimate.

\begin{lemma}
\label{lem:kl-to-stretch}
Let $H$ be a graph embedded in the surface $\Sigma_g$.
Let $k:=\ewnd(H)$ and assume $k\geq2^g$.
Let $\ell$ be the largest integer such that
there is a cycle $\gamma$ of length $k$ in $H^*$ whose shortest
$\gamma$-switching ear has length~$\ell$.
Then there exists an integer $g'$, $0< g'\leq g$, and a subgraph $H'$
of $H$ embedded in $\Sigma_{g'}$ such that 
$$
\ewnd(H')\geq 2^{g'-g}k
	\qquad\mbox{and}\qquad
\stretchd(H')\geq 2^{2g'-2g}\cdot k\ell
\,.$$
\end{lemma}

In a nutshell, the main idea behind the proof of this statement is to
cut along handles that (may) cause small stretch, until we arrive to
the desired toroidal $\Omega(k\times\ell)$ grid.
In particular, the claimed embedding of $H'$ is inherited from that of~$H$.

The arguments required to prove Lemma~\ref{lem:kl-to-stretch} span three sections. 
In Section~\ref{sec:more} we establish several
simple results on the stretch of an embedded graph. As we believe this
new parameter may be of independent interest, it makes sense to gather
these results in a standalone section for possible further
reference. The whole proof of Lemma~\ref{lem:kl-to-stretch} is then
presented in Sections~\ref{sec:finding} and~\ref{sec:findingB}.

Using Lemma~\ref{lem:kl-to-stretch} as the last missing ingredient,
we may now informally wrap up the whole chain of estimates 
\eqref{eq:bigpicture} as follows
$$
\text{\sl Pcost}(G^*) ~\preceq_g~ \stretchd(H')
	~\preceq_{g,\Delta}~ \Tex(H') ~\leq~ \Tex(G)
,$$
where a subgraph $H'\subseteq G$ is found with help of Lemma~\ref{lem:kl-to-stretch}, 
and $H'$ now stands for $G_1$ from~\eqref{eq:bigpicture}.
We remark that $H'$ indeed has a sufficiently dense 
embedding for the middle inequality to hold.
Formally, we are now proving:
\begin{lemma}
\label{lem:kl-to-tex}
Let $G$ be a graph embedded in $\Sigma_g$. Let 
 $(G_i^*,\gamma_i)\,_{i=1,\ldots,g}$ 
be a good planarizing sequence of $G^*$, 
with associated lengths $\{k_i,\ell_i\}_{i=1,\ldots,g}$.
Suppose that $\ewnd(G)\geq5\cdot2^{g-1}\lfloor\Delta(G)/2\rfloor$.
There exists $g'$, $0< g'\leq g$, and a subgraph $H'$
of $G$ embedded in $\Sigma_{g'}$ such that
$\ewnd(H')\geq5\cdot2^{g'-1}\lfloor\Delta(G)/2\rfloor$ and
$$
\Tex(G)	\>\geq\> 
 \frac{1}{7}\,2^{3-2g'}
        \big\lfloor{\Delta(G)}/2\big\rfloor^{-2} \cdot\stretchd(H')
  \>\geq\> \frac{1}{7}\,2^{3-2g}
        \big\lfloor{\Delta(G)}/2\big\rfloor^{-2}
	\cdot \max \{k_i\cdot\ell_i\}_{i=1,2,\dots,g}
\,.$$
Consequently,
$$~~~~
\crg(G)	\>\geq\> 
\frac{1}{21}\,2^{1-2g}
        \big\lfloor{\Delta(G)}/2\big\rfloor^{-2}
	\cdot \max \{k_i\cdot\ell_i\}_{i=1,2,\dots,g}
\,.$$
\end{lemma}

\begin{proof}
Let $j$ be the smallest integer such that $k_j\ell_j=\max \{k_i\ell_i\}_{
i=1,2,\dots,g}$, and let $H:=G_{j-1}$ (in case $j=1$, recall that
we set $G_0:=G$). 
Thus $H$ is a spanning subgraph of $G$ (recall that we deal with a {\em
dual} planarizing sequence), and $H$ is embedded in a surface of genus
$g_1=g-j+1$. An iterative application of Lemma~\ref{lem:dew2} yields
that $\ewnd(H)/\lfloor\Delta(G)/2\rfloor\geq
5\cdot2^{g-1}\cdot2^{g_1-g} =5\cdot2^{g_1-1}$.

We now apply Lemma~\ref{lem:kl-to-stretch} to $H$.
Thus the resulting graph $H'$ of genus $g'\geq1$ satisfies
\smallskip
$\ewnd(H')/\lfloor\Delta(H')/2\rfloor\geq
 \ewnd(H')/\lfloor\Delta(G)/2\rfloor\geq5\cdot2^{g'-1}$
and $\stretchd(H')\geq 2^{2g'-2g_1}\cdot k_j\ell_j\geq 2^{2g'-2g}\cdot
k_j\ell_j$.
Since $H'\subseteq H\subseteq G$, we also have $\Tex(G) \geq \Tex(H')$. 
Using Corollary~\ref{cor:agrid-all} we finally get
\begin{align*}
\Tex(G)\geq& \,\Tex(H') \>\geq\> \frac{1}{7}\,2^{3-2g'}
	\big\lfloor{\Delta(H')}/2\big\rfloor^{-2} \cdot\stretchd(H')
		 \>\geq\> \frac{1}{7}\,2^{3-2g'}
	\big\lfloor{\Delta(G)}/2\big\rfloor^{-2} \cdot\stretchd(H')
\\
	\geq&\>  \frac{1}{7}\,2^{3-2g'}
	\big\lfloor{\Delta(G)}/2\big\rfloor^{-2}
		\cdot2^{2g'-2g}k_j\ell_j
	\>=\>\frac{1}{7}\,2^{3-2g}
	\big\lfloor{\Delta(G)}/2\big\rfloor^{-2} \cdot k_j\ell_j
\,.\end{align*}
The subsequent estimate on $\crg(G)$ then results from Corollary~\ref{cor:crossing-texp}.
\end{proof}

\subsection{Proof of the main theorem}
\label{sub:mainproof}

Having deferred the long and technical proofs of the previous
subsections for the later sections of the paper, all the ingredients
are now in place to prove Theorem~\ref{thm:main-overview}(a). 
In the coming formulation, recall that $\ewnd(G)\geq\fw(G)$
by Lemma~\ref{lem:ewdfw}.

\begin{theorem}[Theorem~\ref{thm:main-overview}(a) with
		 $r_0=5\cdot2^{g-1}\lfloor\Delta/2\rfloor$]
\label{thm:main-parta}
Let $g>0$ and $\Delta$ be integer constants.
There exist a universal constant $c_1>0$, 
and a constant $c_0>0$ depending only on $g$ and $\Delta$,
such that the following holds
for any graph $G$ of maximum degree $\Delta$ embedded in $\Sigma_g$
with nonseparating dual edge-width at least
$5\cdot2^{g-1}\lfloor\Delta/2\rfloor:$
\begin{equation}\label{eq:main-parta}
c_0\cdot\crg(G) \leq \Tex(G) \leq c_1\cdot\crg(G) 
\end{equation}
\end{theorem}
\begin{proof}

The right hand side inequality in \eqref{eq:main-parta} follows at once 
from Corollary~\ref{cor:crossing-texp} (with $c_1=12$).
The left hand side follows by combining Theorem~\ref{thm:upper-cr} and
Lemma~\ref{lem:kl-to-tex};
$\crg(G)\leq 3\left(2^{g+1}-2-g\right)\cdot\max
          \{k_i\ell_i\}_{i=1,2,\ldots,g}\leq
 3\left(2^{g+1}-2-g\right)\cdot
        7\big\lfloor{\Delta}/2\big\rfloor^{2}
          \cdot2^{2g-3}\cdot\Tex(G)$,
which determines $c_0$.
\end{proof}

It could also be interesting to similarly compare $\crg(G)$ and $\Tex(G)$ 
to the dual stretch (of~$G$).
Unfortunately, as discussed at the beginning of this section,
there cannot be any such two-way inequality as \eqref{eq:main-parta} with $\stretchd(G)$.
Although, following Lemma~\ref{lem:kl-to-tex}, we can give a weaker relation.

\begin{theorem}
Let $g>0$ and $\Delta$ be integer constants.
There are constant $c_0',c_1'>0$ and $c_0'',c_1''>0$, depending 
on $g$ and $\Delta$, such that the following holds
for any graph $G$ of maximum degree $\Delta$ embedded in $\Sigma_g$
with nonseparating dual edge-width at least
$5\cdot2^{g-1}\lfloor\Delta/2\rfloor$:
There exists $g'$, $0< g'\leq g$, and a subgraph $H'$
of $G$ embedded in $\Sigma_{g'}$ such that
$\ewnd(H')\geq5\cdot2^{g'-1}\lfloor\Delta/2\rfloor$ and
\begin{equation}\label{eq:main-stra}
c_0'\cdot\crg(G) \leq \stretchd(H') \leq c_1'\cdot\crg(G) 
\,.\end{equation}
Consequently,
\begin{equation}\label{eq:main-strat}
c_0''\cdot\Tex(G) \leq \stretchd(H') \leq c_1''\cdot\Tex(G)
\,.\end{equation}
\end{theorem}

We note in passing that the embedding of the graph $H'$
in this theorem is, in fact, the embedding inherited from the rotation
scheme of~$G$ (cf.~Section~\ref{sec:finding}).
\begin{proof}
Let $H'\subseteq G$ be the graph claimed by Lemma~\ref{lem:kl-to-tex}.
The right hand side inequality in \eqref{eq:main-stra} is implied by
$\Tex(G) \>\geq\> 
 \frac{1}{7}\,2^{3-2g'}
        \big\lfloor{\Delta(G)}/2\big\rfloor^{-2} \cdot\stretchd(H')$
of Lemma~\ref{lem:kl-to-tex} combined with \eqref{eq:main-parta}.
Likewise, the left hand side of \eqref{eq:main-stra} follows from
Theorem~\ref{thm:upper-cr} and again Lemma~\ref{lem:kl-to-tex}.
\eqref{eq:main-strat} then follows at once from \eqref{eq:main-stra} and
\eqref{eq:main-parta}.
\end{proof}

As for the algorithmic part of Theorem~\ref{thm:main-overview}, we can now
provide only a weaker conclusion requiring a dense embedding.
Since removing this restriction requires tools very different from the core
of this paper, we leave the full proof of Theorem~\ref{thm:main-overview}(b)
till Section~\ref{sec:nodensity}.
Again, as in Theorem~\ref{thm:upper-cr}, we represent the output drawing 
by a planar graph obtained by
replacing each crossing with a new (specially marked) vertex.

\begin{theorem}[Weaker version of Theorem~\ref{thm:main-overview}(b)]
\label{thm:main-dense}
Let $g>0$ and $\Delta$ be integer constants.
Assume $G$ is a graph of maximum degree $\Delta$ embeddable in the surface 
$\Sigma_g$ with $\ewnd(G)\geq5\cdot2^{g-1}\lfloor\Delta/2\rfloor$.
There is an algorithm that, in time $\ca O(n\log\log n)$ where $n=|V(G)|$, 
outputs a drawing of $G$ in the plane with at most $c_2'\cdot\crg(G)$ crossings,
where $c_2'>0$ depends only on $g$ and~$\Delta$.
\end{theorem}
\begin{proof}
First,
although our algorithm in Theorem~\ref{thm:upper-cr} takes an embedded graph as its
input, we might as well take a non-embedded graph as input without any
loss of efficiency; indeed, 
Mohar~\cite{Mo99} showed that, for any fixed genus $g$, there is a linear
time algorithm that takes as input any graph $G$ embeddable in $\Sigma_g$
and outputs an embedding of $G$ in $\Sigma_g$.

Therefore, by Theorem~\ref{thm:upper-cr} we get a drawing of $G$ with the
number of crossings as in \eqref{eq:uppercr}.
Using Lemma~\ref{lem:kl-to-tex} we can then estimate
$3\left(2^{g+1}-2-g\right)\cdot\max
          \{k_i\ell_i\}_{i=1,2,\ldots,g}\leq
 3\left(2^{g+1}-2-g\right)\cdot
        21\big\lfloor{\Delta}/2\big\rfloor^{2}
          \cdot2^{2g-1}\cdot\crg(G)=
 c_2'\cdot\crg(G)$.
\end{proof}

\section{Finding grids in the torus}
\label{sec:grids}

In this section we prove Theorems~\ref{thm:two-cycle-families},
\ref{thm:agrid-torus}, and Corollary~\ref{cor:agrid-torus}.

\begin{figure}[b]
\begin{center}
\def\svgwidth{7cm}
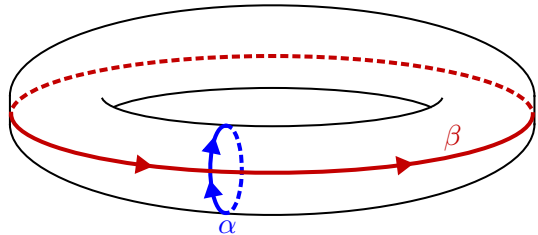
\vspace{-5mm}
\end{center}
\caption{A basis $(\alpha,\beta)$ of the torus.}
\label{fig:torusab}
\end{figure}

\def\cylin{\mbox{$\Pi$}}
\begin{proof}[\bf Proof of Theorem~\ref{thm:two-cycle-families}]
Let $\alpha, \beta$ be oriented simple closed curves such that
$(\alpha,\beta)$ is a basis and $\alpha,\beta$
intersect (cross) each other exactly once; see Figure~\ref{fig:torusab}.
Using a standard surface
homeomorphism argument (cf.~\cite[Section 6.3.2]{St93}), 
we may assume without loss of generality that each $C_i$ has the same homotopy 
type as $\alpha$ (we assign an orientation to the cycles $C_i$ to ensure this). 
Thus it follows that the cycles $D_j$ may be oriented in such a way that
there exist integers $r\ge 0,s\ge 1$
such that the homotopy type of each $D_j$ is $\alpha^r\beta^s$.

We assume without loss of generality that $p\geq q\geq3$. We let $C_+:=C_1\cup C_2\cup\dots\cup
C_p$ and $D_+:=D_1\cup D_2\cup\dots\cup D_q$. We shall assume that
among all possible choices of the collections $\{C_1,\dots,C_p\}$ and
$\{D_1,\ldots,D_q\}$ that satisfy the conditions in the theorem (for
the given values of $p$ and $q$), our collections $\cc:=\{C_1,\ldots,C_p\}$
and $\{D_1,\ldots,D_q\}$ minimize $|E(C_+)\sem E(D_+)|$.

The indices of the $C_i$-cycles (respectively, the $D_j$-cycles) are read modulo
$p$ (respectively, modulo~$q$). 
We may assume that the cycles $C_1,C_2,\dots,C_p$ appear in this cyclic order
around the torus; that is, for each $i=1,2,\ldots,p$, one of the cylinders 
bounded by $C_i$ and $C_{i+1}$ does not intersect any other curve in $\cc$. 
We say that $C_i$ is to the {\em left of} $C_{i+1}$,
and $C_{i+1}$ is to the {\em right of} $C_{i}$.
Moreover, we may choose orientations 
such that $\beta$ intersects $C_1,C_2,\ldots,C_p$ in this cyclic order.

At first glance it may appear that it is easy to get the desired grid
as a minor of $C_+\cup D_+$, since every
$D_j$ has to intersect
each $C_i$ in some vertex of $G$ (this follows since each pair
$(C_i,D_j)$ is a basis). There are, however, two
possible complications. First, two cycles $C_i, D_j$ could have many
``zigzag'' intersections, with $D_j$ intersecting $C_i$, then
$C_{i+1}$, then $C_i$ again, etc.
See, for example, the fragment depicted in Figure~\ref{fig:quasimove} (left). 
Second, $D_j$ may ``wind'' many times in the direction orthogonal to~$C_i$. 
These are the main problems to overcome in the upcoming proof.

\smallskip

We start by showing that, even though we may intersect some
$C_i$ several times when traversing some $D_j$, it follows from the choice of $\cc$ that, after $D_j$ intersects $C_i$, it must hit either $C_{i-1}$
or $C_{i+1}$ before coming back to $C_i$.

\begin{claim}\label{clm_ear}
No $C_+$-ear contained in $D_+$ has both ends on the same cycle $C_i$.
\end{claim}

\begin{proof}\useInnerQed
Suppose that there is a  $C_+$-ear $P\subset D_+$ with both ends on
the same $C_i$. Modify 
$C_i$ by following $P$ in the appropriate
section, and let $C_i'$ denote the resulting cycle. The families
$\{C_1,\ldots,C_{i-1},C_i',C_{i+1},\ldots ,C_p\}$ and
$\{D_1,\ldots,D_q\}$ satisfy the conditions in the theorem. The fact that 
$|E(C_1\cup \cdots \cup C_{i-1} \cup C_{i}' \cup C_{i+1} \cdots \cup C_p)
\setminus E(D_+)| < |E(C_+ \setminus D_+)|$
contradicts the choice of $\{C_1,\ldots,C_p\}$.
\end{proof}

To proceed with our proof, we need to relax the requirement that
$D_1,\ldots,D_q$ are cycles.
A {\em quasicycle} is a closed walk $D'$ in a graph such that every two
consecutive edges of $D'$ are distinct.
As with cycles, we assign each quasicycle an implicit orientation.
For $j\in\{1,2,\dots,q\}$,
let $D_j'$ be a quasicycle in $G$ homotopic to $D_1$, with the same orientation.
The {\em rank} $s_j$ of $D_j'$ is the number of connected components
of $C_+\cap D_j'$. By
traversing $D_j'$ once and registering each time it intersects a curve
in $\cc$, starting with (some intersection with) $C_1$, we obtain 
an {\em intersection sequence} $a_j(i)$, $i=1,\dots,s_j$, of length $s_j$ 
where each $a_j(i)$ is from $\{1,\dots,p\}$. 
To simplify our notation, we introduce the following convention:
the index $i$ in the sequence $a_j$ is read modulo $s_j$ -- meaning that
$a_j(i+s_j)=a_j(i)$, and the value of $a_j(i)$ is read modulo~$p$ ``plus 1'', i.e.,
if $a_j(i)=p$ then $a_j(i)+1=1$.

Since we chose the starting point of
the traversal of $D_j'$ so that the first curve of $\cc$ it intersects
is $C_1$, it follows that $a_j(1)=1$. 
We denote by $Q_{j,t}$, $t=1,2,\dots,s_j$, the path of $D_j'$ 
(possibly a single vertex)
forming the corresponding intersection with the cycle $C_{a_j(t)}$,
and by $T_{j,t}$ the path of $D_j'$ between $Q_{j,t}$ and $Q_{j,t+1}$.
We say that $D_j'$ is {\em $C_+$-ear good} if 
no $C_+$-ear contained in $D_j'$ has both ends on the same~$C_i$ (cf.\ Claim~\ref{clm_ear}). 
Hence if $D_j'$ is $C_+$-ear good then $a_j(t+1)\not=a_j(t)$ 
and thus $a_j(t+1)=a_j(t)\pm1$ for $t=1,2,\dots,s_j$.

We also need to slightly relax the property that $D'_i$ is disjoint from
$D'_j$ if $i\not=j$.
(This is, for example, useful in zig-zag situations like the one depicted in
Figure~\ref{fig:quasimove} (left), in which no ``local improvement'' is
possible without introducing another intersection between some quasicycles.)
We define the following restriction.
A collection of $C_+$-ear good quasicycles $\{D_1',D_2',\dots,D_q'\}$ in $G$
is called {\em quasigood} if it moreover satisfies the following 
for any $m,n\in\{1,2,\dots,q\}$:
whenever $D_n'$ intersects $D_m'$ in a connected component $P$, this
component $P$ is a path (in the case $m=n$, which is possible due to
$D_n'$ being a quasicycle, the path $P$ is repeated within the walk $D'_n$)
and the following two conditions hold, possibly exchanging $n$ with $m$ in
both of them:
\begin{itemize}
\item[(Q1)]
there exists $x$, an index of the intersection sequence of $D_n'$, such that
$a_n(x-1)=a_n(x+1)=a_n(x)-1$ and $P\subseteq Q_{n,x}$
(in particular, $P$ belongs to $C_{a_n(x)}$);
\item[(Q2)]
the subembedding of $S:=T_{n,x-1}\cup Q_{n,x}\cup T_{n,x}$,
which is a subpath of $D_n'$, stays locally on one side of the embedding of $D_m'$
(by (Q1), $S$ is to the left with respect to~$C_{a_n(x)}$).
\end{itemize}

Informally, this means that if $D_n'$ intersects $D_m'$ in~$P$,
then $D_n'$ makes a $C_{a_n(x-1)}$-ear $S$ ``touching'' $D_m'$ in 
$P\subseteq D_m'\cap C_{a_n(x)}$, and $S$ is to the left of $C_{a_n(x)}$.
Such a situation can be seen with the thick solid fragments of $D_2'$
and $D_3'$ in Figure~\ref{fig:quasimove} (right).

Since the cycles in $\{D_1,D_2,\ldots,D_q\}$ are pairwise disjoint and 
$D_j$ is clearly a $C_+$-ear good quasicycle for each $j=1,2,\ldots,q$, 
it follows that $\{D_1,\ldots,D_q\}$ is a quasigood collection. 
Now among all choices of a quasigood collection $\dd:=\{D_1',D_2',\dots,D_q'\}$ 
in $G$, we select $\dd$ minimizing the sum of the ranks of its quasicycles.
For each $D_j'$, as above, we let $s_j$ denote its rank.

\begin{figure}[t]
\begin{center}
\input{repairAtX.pdf_tex}\qquad\qquad\qquad\qquad
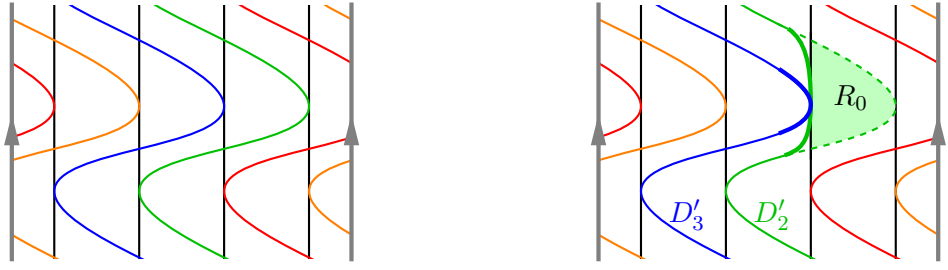\smallskip
\end{center}
\caption{(left)
A zig-zag arrangement of four quasicycles $D'_1,D'_2,D'_3,D'_4$
(horizontal, colored) and four disjoint cycles $C_{i-2},C_{i-1},C_{i},C_{i+1}$
(vertical, black). (right)
The result of a local move on $D'_2$ which makes a quasigood
intersection with $D'_3$, as used in the proof of Claim~\ref{clm_quasi}.}
\label{fig:quasimove}
\end{figure}

\begin{claim} \label{clm_quasi}
For all $1\leq j\leq q$ the intersection sequence of $D_j'$ satisfies 
$a_j(t-1)\not=a_j(t+1)$ for any $1\leq t\leq s_j$.
Consequently, $\dd$ is a collection of pairwise disjoint cycles in~$G$.
\end{claim}

\begin{proof}\useInnerQed
The conclusion that $\dd=\{D_1',D_2',\ldots,D_q'\}$ is
a collection of pairwise disjoint cycles directly follows from the
first statement in the claim, since $\dd$ is a quasigood collection.
We hence focus on the first statement, $a_j(t-1)\not=a_j(t+1)$, in the proof.

The main idea in the proof is quite simple: if $a_j(t-1)=a_j(t+1)$,
then we could modify $D_j'$ rerouting it through $C_{a_j(t-1)}$ instead of
$T_{j,t-1}\cup Q_{j,t}\cup T_{j,t}$, thus decreasing $s_j$ (and hence
the total sum of the ranks) by $2$, and
consequently contradicting the minimum choice of $\dd$ above.
This move is illustrated in Figure~\ref{fig:quasimove}.
We now formalize this rough idea.

Recall that, if for some $j,t$ we have \mbox{$a_j(t-1)=a_j(t+1)=i$}
then $a_j(t)=i\pm1$.
If $a_j(t)=i-1$ was true, for some other $t'$ we would 
necessarily have $a_j(t'-1)=a_j(t'+1)$ and $a_j(t')=a_j(t'-1)+1$.
So, seeking a contradiction, we may assume that $a_j(t)=i+1$.
Let $\cylin_i$ denote the cylinder bounded by $C_i$ and $C_{i+1}$.
Then the path $T:=T_{j,t-1}\cup Q_{j,t}\cup T_{j,t}$ is drawn in
$\cylin_i$ with both ends on $C_i$ and ``touching''~(i.e., not
intersecting transversely) $C_{i+1}$.
We denote by $R_0\subset\cylin_i$ the open region bounded by $T$ and $C_i$,
and by $T'$ the section of the boundary of $R_0$ not belonging to~$D_j'$
(hence $T'\subset C_i$).

Assuming that $R_0$ is minimal over all choices of~$j,t$ for which $a_j(t-1)=a_j(t+1)$,
we show that no $D_m'$, $m\in\{1,\dots,q\}$, intersects~$R_0$. Indeed, 
if some $D_m'$ intersected $R_0$, then $D_m'$ could not enter $R_0$
across $T$ by the definition (Q2) of a quasigood collection.
Hence $D_m'$ should enter and leave $R_0$ across~$T'$,
but not touch $Q_{j,t}\subseteq C_{i+1}$ by the minimality of $R_0$.
But then, $D_m'$ would make a $C_+$-ear with both ends on $C_i$, contradicting
the assumption that $D_m'$ was $C_+$-ear good.

Now we can form $D_j^o$ as the symmetric difference of $D_j'$ with the boundary
of~$R_0$ (so that $D_j^o$ follows $T'$).
To argue that $\dd^o:=\{D_1',\dots,D'_{j-1},D_j^o,D'_{j+1},\dots,D_q'\}$ 
is a quasigood collection again, it suffices to verify the conditions of a
quasigood collection for all possible new intersections of $D_j^o$ along~$T'\subset C_i$.
Suppose that there is some $D_n'$ such that $Q_{n,x}$ (the local intersection
of $D_n'$ with~$C_i$ for an appropriate index $x$) intersects also~$T'$.
If $Q_{n,x}$ contains (at least) one of the ends of $T'$, then $Q_{n,x}$ intersects $D'_j$.
Since $D_n'$ is disjoint from the open region $R_0$, assumed validity
of (Q1),(Q2) for $D_n',D'_j$ immediately implies their validity for $D_n',D_j^o$.
Similarly, if $Q_{n,x}$ is contained in the interior of $T'$
(in which case $Q_{n,x}$ is disjoint from $D'_j$), then the fact that $D_n'$
is disjoint from the open region $R_0$ implies that $D_n'$ is locally to the
left of~$C_i$.
Hence, it is $a_n(x)=i$ and $a_n(x-1)=a_n(x+1)=i-1$ by Claim~\ref{clm_ear}, 
conforming to (Q1). (Q2) now follows trivially.

Finally, since $\dd^o$ is quasigood as well, but the sum of the ranks of its elements is
strictly smaller than it was for~$\dd$ (by~$2$), 
we get a contradiction to the choice of $\dd$.
\end{proof}

\begin{claim} \label{clm_wonce}
There exists a collection of $q$ pairwise disjoint, pairwise homotopic noncontractible cycles
in~$G$,
each of which has a connected nonempty intersection with each cycle in $\cc$.
\end{claim}

\begin{figure}[t]
\begin{center}
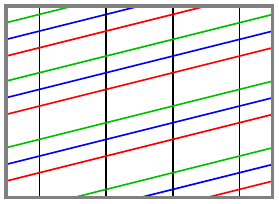
\hfill
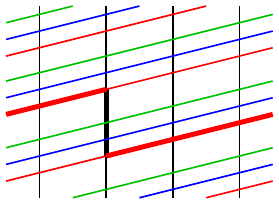
\hfill
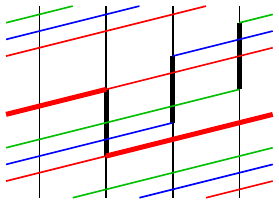
\end{center}
\caption{Proof visualization for Claim~\ref{clm_wonce}. 
(left) An initial situation with cycles $\{D'_1,D'_2,D'_3\}$ crossing $\mathcal{C}$ multiple times. 
(middle) Finding the new cycle $D''_1:=T_1\cup W1$ (bold lines), crossing each of $\mathcal{C}$ only once. 
(right) The final pairwise homotopic collection $\{D''_1,D''_2,D''_3\}$ (bold lines), each crossing each of $\mathcal{C}$ only once}
\label{fig:repairhomotopy}
\end{figure}

\begin{proof}\useInnerQed
It follows from Claim~\ref{clm_quasi} that the intersection sequence
of each $D_j'$ is a $t$-fold repetition of the subsequence 
$\langle 1,2,\ldots,p\rangle$, for some nonnegative integer $t$. 
If $t=1$, we are obviously done, so assume $t \ge 2$. 
Informally, our task is to ``shortcut'' each $D_j'$
such that it ``winds only once'' in the direction orthogonal to $\alpha$. 
See an illustration in Figure~\ref{fig:repairhomotopy}.

Note that, for all $i=1,\dots,p$ and $j=1,\dots,q$,
every $C_i$-ear contained in $D_j'$ is $C_i$-switching by
Claim~\ref{clm_quasi}, and so it intersects $C_i, C_{i+1},
 \ldots, C_{i-1}$ in this order before returning to~$C_i$.
Let $T_1 \subset D_1'$ be any $C_1$-ear, and let $x_1,y_1$ be the
end points of~$T_1$. Then let $W_1\subset C_1$ be (any)
one of the two paths contained in $C_1$ with the end points $x_1, y_1$. 
It is clear that the cycle $D_1''=T_1\cup W_1$ is a simple closed curve 
that has a connected nonempty intersection with each $C_i$, as required.

Since $D_1''$ is not homotopic to $D_1'$,
every $D_j'$ has to intersect $D_1''$ in $W_1$ (once).
Let $x_j$ denote the vertex of $D_j'\cap W_1$ closest to $x_1$.
Without loss of generality, assume that the vertices $x_1,x_2,\ldots,x_q$
appear on $W_1$ in this order.
Recall that $q\leq p$.
For $j=1,\dots,q$, let $T_j\subseteq D'_j$ be the (unique) $C_j$-ear
coming right after~$x_j$, and let $W_j\subset C_j$ be the path joining the  
ends of $T_j$ that is disjoint from $T_1$.
Since $T_j$ is disjoint from $W_1$, the cycle $D_j''=T_j\cup W_j$
is indeed disjoint from $D_1''$ and homotopic to~$D_1''$.
Similarly, $D_j''$ is disjoint from $D_k''$ for $k\not=j$, and
$D_j''$ has a connected nonempty intersection with each $C_i$, as required.
\end{proof}

To conclude the proof of Theorem~\ref{thm:two-cycle-families}, we use the
collection $\{D_1'', D_2'', \ldots, D_q''\}$ guaranteed by Claim~\ref{clm_wonce}.
For each $i=1,2,\ldots,p$ and $j=1,2,\ldots,q$, we
contract the path $C_i\cap D_j''$ to a single vertex (unless it already is a single
vertex). Since the curves $D_1'',D_2'',\ldots,D_q''$ are pairwise disjoint and pairwise
homotopic, it directly follows that the
resulting graph is isomorphic to a subdivision of the $p\times q$-toroidal grid.
\end{proof}

\begin{proof}[\bf Proof of Theorem~\ref{thm:agrid-torus}]
First we show the following.

\begin{claim}\label{cl:ellcycles}
$G$ has a set of at least
$\frac \ell{\lfloor\Delta/2\rfloor}$
pairwise disjoint cycles,
all homotopic to $\alpha$. 
\end{claim}

\begin{proof}\useInnerQed
Let $F$ be the set of those edges of $G$ intersected by $\alpha$.  Let $\alpha_1, \alpha_2$ be loops very close to and homotopic to
$\alpha$, one to each side of $\alpha$, so that the cylinder bounded by $\alpha_1$ and
$\alpha_2$ that contains $\alpha$ intersects $G$ only in the edges of
$F$. Now we cut the torus by removing the (open) cylinder bounded by
$\alpha_1$ and $\alpha_2$, thus leaving an embedded graph $H:=G-F$ on a
cylinder $\cylin$  with boundary curves (``rims'') $\alpha_1$ and $\alpha_2$. 
Let $\delta$ be a curve on $\cylin$ connecting a point of $\alpha_1$
to a point of $\alpha_2$, such that $\delta$ has the fewest possible points in
common with the embedding $H$. We note that we may clearly assume that
the $p$ points in which $\delta$ intersects $H$ are vertices.

We claim that 
$p\geq \frac \ell{\lfloor\Delta/2\rfloor}$.
Indeed, if $p< \frac \ell{\lfloor\Delta/2\rfloor}$,
then the union of all faces incident with the $p$ vertices 
intersected by $\delta$ would contain a dual path $\beta$ of
length at most $p\cdot \lfloor\Delta/2\rfloor <
\frac \ell{\lfloor\Delta/2\rfloor}\cdot
\lfloor\Delta/2\rfloor=\ell$. Such $\beta$ would be an
$\alpha$-switching dual ear in $G^*$ of length less than $\ell$, a
contradiction.  

We now cut open the cylinder $\cylin$ along $\delta$, duplicating each
vertex intersected by $\delta$. As a result we obtain a graph $H'$
embedded in the rectangle with sides $\alpha_1,\delta_1,\alpha_2,\delta_2$ in
this cyclic order, so that $\delta_1$ (respectively, $\delta_2$)
contains $p$ vertices $w_i^1, i=1,2,\ldots,p$ (respectively, $w_i^2, i=1,2,\ldots,p$).

We note that there is no vertex cut of size at most $p-1$ in $H'$
separating 
$\{w_1^1,\dots,w_p^1\}$ from
$\{w_1^2,\ldots,w_p^2\}$, as such a vertex cut would imply the
existence of a curve $\varepsilon$ from $\alpha_1$
to $\alpha_2$ on $\cylin$ intersecting $H$ in fewer than $p$ points,
contradicting our choice of $\delta$.
Thus applying Menger's Theorem 
we obtain $p$ pairwise disjoint paths 
from $\{w_1^1,\dots,w_p^1\}$ to
$\{w_1^2,\ldots,w_p^2\}$ in $H'$.
Moreover, it follows by planarity of $H'$ that each of these paths
connects $w_i^1$ to the corresponding $w_i^2$ for $i=1,\dots,p$. By
identifying back $w_i^1$ and $w_i^2$ for $i=1,\dots,p$, we get 
a collection of $p$ pairwise disjoint cycles in~$H$, each of them
homotopic to $\alpha$.
\end{proof}

We have thus proved the existence of a collection $\cc$ of $\ell/\dee$
pairwise disjoint, pairwise homotopic noncontractible cycles. 
We have $\fw(G) \ge \ewnd(G)/\dee=k/\dee\geq5$ by Lemma~\ref{lem:ewdfw}.
So, by Theorem~\ref{thm:deGS}, we obtain
that $G$ also contains two collections $\dd,\ee$ of cycles such that:
(i) the cycles in $\dd$ are noncontractible, pairwise disjoint, and
pairwise homotopic; (ii)
the cycles in $\ee$ are noncontractible, pairwise disjoint, and
pairwise homotopic; (iii) for any $D\in\dd$ and $E\in\ee$, the pair
$(D,E)$ is a basis; and (iv) each of $|\dd|$ and $|\ee|$ is at least
$\big\lfloor\frac23\lceil\frac k{\lfloor\Delta(G)/2\rfloor}\rceil\big\rfloor$.

Let $C\in\cc$, $D\in\dd$, and $E\in\ee$. From
properties (i)--(iii) it follows that either $(C,D)$ or $(C,E)$ is a basis. 
Therefore, Theorem~\ref{thm:two-cycle-families} guarantees the existence of
a toroidal grid minor of size
$$
\left\lceil\frac \ell{\lfloor\Delta(G)/2\rfloor}\right\rceil
		\>\times\>
	\left\lfloor\frac23\left\lceil
		\frac k{\lfloor\Delta(G)/2\rfloor}
        \right\rceil\right\rfloor
\,.\qedhere$$\smallskip
\end{proof}

\begin{proof}[\bf Proof of Corollary~\ref{cor:agrid-torus}]
Let $k:=\ewnd(G)$, and let $\ell$ and $\alpha$ be as in
Theorem~\ref{thm:agrid-torus}. 
By Lemma~\ref{lem:thstr}, \mbox{$\stretchd(G)\leq 2k\ell$}.
Let $r=\big\lceil\frac k{\lfloor\Delta(G)/2\rfloor}\big\rceil$. Since
$r\geq5$, it follows that $\lfloor2r/3\rfloor\geq\frac67(2r/3)=\frac47r$
(with equality at $r=7$).
Letting $s=\big\lceil\frac\ell{\lfloor\Delta(G)/2\rfloor}\big\rceil$
we then get, by Theorem~\ref{thm:agrid-torus},
$$\Tex(G) \geq s\cdot\left\lfloor\frac23r\right\rfloor\geq 
	\frac47rs \geq \frac47k\ell \cdot\big\lfloor{\Delta(G)}/2\big\rfloor^{-2}
	\geq \frac27\stretchd(G) \cdot\big\lfloor{\Delta(G)}/2\big\rfloor^{-2}
\,.$$

The previous unconditional estimate can be improved, for any fixed $\varepsilon>0$ and
all sufficiently large $k$, as follows.
If $\ell\geq\frac23k$ then, by Lemma~\ref{lem:thstr},
$\stretchd(G)\leq k(\ell+k/2)\leq k(\ell+3\ell/4)=\frac74k\ell$ and
\begin{eqnarray*}
\Tex(G) &\geq& s\cdot\left\lfloor\frac23r\right\rfloor
	\geq \frac\ell{\lfloor\Delta(G)/2\rfloor} \cdot
		\left\lfloor\frac23 
		 \frac k{\lfloor\Delta(G)/2\rfloor}\right\rfloor
	\geq \ell\cdot\left\lfloor\frac23k\right\rfloor
		\cdot\big\lfloor{\Delta(G)}/2\big\rfloor^{-2}
\\
 	&\geq& \left(\frac23\cdot\frac47-\varepsilon\right) \cdot\frac74 k\ell
		 \cdot\big\lfloor{\Delta(G)}/2\big\rfloor^{-2}
 	\geq \left(\frac8{21}-\varepsilon\right) \cdot\stretchd(G)
		 \cdot\big\lfloor{\Delta(G)}/2\big\rfloor^{-2}
\,.\end{eqnarray*}
Otherwise, $\ell<\frac23k$ and we let $k=\alpha\ell$ where $\alpha>\frac32$.
We similarly have $\stretchd(G)\leq k(\ell+k/2)=\frac{\alpha+2}2k\ell$
and we can directly use Theorem~\ref{thm:deGS} to estimate
(where $\varepsilon'=\frac{\alpha+2}{2\alpha}\varepsilon>0$)
\begin{eqnarray*}
\Tex(G) &\geq& \left\lfloor\frac23r\right\rfloor^2 \geq 
	\left(\frac49-\varepsilon'\right) k\alpha\ell
		 \cdot\big\lfloor{\Delta(G)}/2\big\rfloor^{-2}
\\
	&\geq& \left(\frac49-\varepsilon'\right)\frac{2\alpha}{\alpha+2}
		\cdot \frac{\alpha+2}2 k\ell
		 \cdot\big\lfloor{\Delta(G)}/2\big\rfloor^{-2}
	= \left(\frac{8\alpha}{9\alpha+18}-\varepsilon\right)
		\cdot \stretchd(G)
		 \cdot\big\lfloor{\Delta(G)}/2\big\rfloor^{-2}
\,.\end{eqnarray*}
Now for $\alpha\geq\frac32$ we have
$\frac{8\alpha}{9\alpha+18}\geq\frac8{21}$.
\end{proof}

\section{Drawing embedded graphs into the plane}
\label{sec:drawing-upper}

In this section, we prove Theorem~\ref{thm:upper-cr}. That is, we provide an efficient
algorithm that, given a graph $G$ embedded in some orientable surface, yields a drawing of $G$
(with a controlled number of crossings) in the plane.
We start with an informal outline of the proof.

We proceed in $g$ steps, working at the $i$-th step with the pair
$(G_i^*,\gamma_i)$. For convenience, let $G_0=G$, and define
$F_i=E(G_{i-1})\sem E(G_{i})=E(\gamma_i)$. The idea at the $i$-th step
is to cut from $G_{i-1}$ the edges intersected by $\gamma_i$ (that is, the set
$F_i$). We could then draw these edges into the embedded graph $G_i$ along
the route determined by a $\gamma_i$-switching ear of length $\ell_i$
in $G_{i-1}$. This would result in at most $k_i(\ell_i+k_i)$ new crossings in $G_i$
(similarly as in Figure~\ref{fig:addedges}).
We consider routing
all the edges of $F_i$ in one bunch (i.e., along the same route), 
even though routing every edge separately could perhaps save a small number of crossings.
We have two reasons for this treatment; it makes the proofs simpler
(and it would be very hard to gain any improvement in the
worst-case approximation bound by individual routing, anyway),
and the algorithm has slightly better runtime.

In reality, the situation is not as simple as in the previous sketch.
The main complication comes from the fact that subsequent cutting (in step
$j>i$) could ``destroy'' the chosen route for $F_i$.
Then it would be necessary to perform further re-routing for a part or all of the edges 
of $F_i$ in step $j$ along a route for $F_j$ (costing up to $k_i\ell_j$ additional crossings).
This could essentially happen in each subsequent step until the end of
the process at~$j=g$.

We handle this complication in two ways:
Proof-wise, we track a possible insertion route (and its necessary modifications)
for $F_i$ through the full cutting process. In particular, we show that
the final insertion route for $F_i$ is never longer than 
$\ell_i+\ell_{i+1}+\dots+\ell_g$, for each index $i$, 
which constitutes an upper bound on the final algorithmic solution.
We also have to take care of the following detail;
that a detour for the route of $F_i$ at any step $j>i$ does not produce
significantly more additional crossings than $k_j\ell_j$ --
this holds as long as $k_j$ is never much smaller than $k_i$ (cf.~Lemma~\ref{lem:dew2}).

Algorithmically, we will reinsert all the edges $E(G)\sem E(G_{g})$
only at the very end, into $G_g$. 
For that we find shortest insertion routes for the
(subsets of the) edges of $F_i$,
$i\in\{1,\dots,g\}$ independently,
which is algorithmically a very easy solution, and we moreover iteratively ensure
that no two insertion routes cross each other more than once.
\smallskip

\begin{proof}[{\bf Proof of Theorem~\ref{thm:upper-cr}}]
As outlined in the sketch above, we proceed in $g$ steps. At the
$i$-th step, for $i=1,2,\ldots,g$, we take the embedded graph
$G_{i-1}$ and cut the surface open along
$\gamma_i$, thus severing the edges in the set
$F_i:=E(G_{i-1})\setminus E(G_i)=E(\gamma_i)$. This decreases the
genus by one, and creates two holes, which we repair by pasting a
closed disc on each hole. Thus we get the graph $G_i$ embedded in a
compact surface with no holes.

\begin{claim}\label{cl:sumell}
Let $i\in\{1,\ldots,g\}$, and let $f$ be an edge in $F_i$. Then
$f$ can be drawn into the plane graph $G_g$ with at most $\sum_{j=i}^g\ell_j$ crossings. 
\end{claim}

\begin{proof}\useInnerQed
Let $i\in\{1,\ldots,g\}$ be fixed. In the graph $G_i$, we let $a,b$ denote the two new faces created by
cutting $G_{i-1}$ along $\gamma_i$ (thus each of these faces contains
one of the pasted closed discs). Let $f$ be an edge in $F_i$, with
end vertices $f_a$ (incident with face $a$ in $G_i$) and $f_b$ (incident with
face $b$ in $G_i$). 

For each $j\in\{i,\ldots,g\}$, we associate two unique 
faces $a_j(f), b_j(f)$ of $G_j$ with the edge $f$. Loosely speaking, these faces are
the natural heirs in $G_j$ of the faces $a$ and $b$, if we stand in
$G_j$ on the vertices $f_a$ and $f_b$.
This can be (still rather informally) defined recursively as follows. 
First, let $a_i(f) = a$ and $b_i(f) = b$. 
Now suppose $a_{j-1}(f), b_{j-1}(f)$ have been defined for some $j$, $i<j\le g$. 
We then let $a_j(f)$ be the unique face of $G_j$ which contains
the points of face $a$ in a small neighborhood of~$f_a$.
The face $b_j(f)$ is defined analogously. 
In regard of this definition, we point out that $a,b$ are faces in $G_i$, 
but by the further cutting process, they may not be faces in $G_j$ for some $j>i$.
An alternative formal (and discrete) definition of $a_j(f)$ may be given as follows:
let $e_1,e_2$ be the two edges of $a_{j-1}(f)$ incident with~$f_a$ in $G_{j-1}$. 
In the cyclic ordering of edges of $G_{j-1}$ around~$f_a$,
we assume that $e_1$ is right before $e_2$, and
we find $e_1',e_2'$ such that (i) $e_1'$ is the last edge preceding or equal
to $e_1$ and $e_2'$ is the first edge succeeding or equal to $e_2$, and (ii)
$e_1',e_2'\in E(G_j)$.
Then $e_1',e_2'$ are consecutive edges in the cyclic ordering of edges
around~$f_a$ in the graph $G_j\subseteq G_{j-1}$, and hence $e_1',e_2'$ in
this order define a unique face $a_j(f)$ of $G_j$ incident to~$f_a$.

The vertex $f_a$ (respectively, $f_b$) is incident to the face $a_g(f)$ (respectively, $b_g(f)$) in the plane embedding $G_g$.
To finish the claim, it suffices to show that the dual distance
between $a_g(f)$ and $b_g(f)$ in $G_g$ is at most 
$\sum_{j=i}^g\ell_j$. 
We prove this by induction over $j=i,i+1,\dots,g$, i.e., we show that 
the dual distance between $a_j(f)$ and $b_j(f)$ in $G_j$ is at most
$\ell_i+\ell_{i+1}+\dots+\ell_j$.

This holds (with equality) for $j=i$ by the definition of $\ell_i$.
For $j>i$, take a shortest dual path $\pi$ in $G_{j-1}$
connecting $a_{j-1}(f)$ to $b_{j-1}(f)$.
If $\pi$ does not intersect $\gamma_j$, then $\pi$ is also a dual path
from $a_{j}(f)$ to $b_{j}(f)$ in $G_j$ and we are done.
Otherwise, we denote by $\pi_a$ and $\pi_b$ the subpaths of $\pi$ from 
the ends $a_{j-1}(f)$ and $b_{j-1}(f)$, respectively, to the nearest
intersections with $\gamma_j$.
It may happen that $\pi_a$ or $\pi_b$ consist of a single vertex.
Let $\pi_a'$ and $\pi_b'$ denote the inherited paths in $G_j$;
they differ from $\pi_a$ and $\pi_b$ (respectively) only in their 
former ends on $\gamma_j$, which are now among the two new faces
$c_1,c_2$ of~$G_j$ that have been created by cutting $G_{j-1}$.
If, say, $c_1$ is an end of both $\pi_a'$ and $\pi_b'$, then
$\pi_a'\cup\pi_b'$ is a dual path from $a_{j}(f)$ to $b_{j}(f)$ 
in $G_j$ of length at most $\len(\pi)$ and we are again done.
On the other hand, if $\pi_a'\cap\pi_b'=\emptyset$, then we can make
a dual path $\pi'$ in $G_j$ as the union of $\pi_a'\cup\pi_b'$ with a
$\gamma_j$-switching ear of length $\ell_j$.
Then the dual distance between $a_j(f)$ and $b_j(f)$ is at most
$\len(\pi)+\ell_j\leq \ell_i+\dots+\ell_{j-1}+\ell_j$, as claimed.
\end{proof}

Now recall that $|F_i|=k_i$, for $i=1,\ldots,g$. 
From Claim~\ref{cl:sumell} it follows that the edges
in $F_i$ can be added to the plane embedding $G_g$ by introducing at
most $k_i \cdot \sum_{j=i}^g\ell_j$ crossings with the edges of
$G_g$. This measure disregards any additionally crossings arising between edges of~$F_i$.
Though, in the worst case scenario each edge of
$F:= F_1 \cup F_2 \cup \cdots \cup F_g =E(G)\sem E(G_{g})$ crosses each other edge from~$F$.
Since, by the natural arc-exchange argument, we may assume that
every two edges of $F$ cross at most once (without impact on the number of
crossings between $F$ and $E(G_g)$),
the edges of $F$ can be added to the plane embedding $G_g$ by introducing at most 
$\sum_{i=1}^g \left(k_i\cdot\sum_{j=i}^g \ell_j+k_i\cdot\sum_{j=i}^g k_j\right)$
crossings. 
Using that $2\ell_i\geq k_i$ (cf.~Lemma~\ref{lem:kl2}), this
process yields a drawing of $G$ in the plane with at most
\begin{align*}
\sum_{i=1}^g \left(k_i\cdot\sum_{j=i}^g (k_j+\ell_j)\right)\;&\leq\;
  \sum_{i=1}^g \left(k_i\cdot\sum_{j=i}^g 3\ell_j\right)
  =\; 3\sum_{j=1}^g \left(\ell_j\cdot\sum_{i=1}^j k_i\right)
\end{align*}
crossings. The inductive application of Lemma~\ref{lem:dew2} yields
$k_i\leq2^{j-i}k_j$ for all $1\leq i<j\leq g$. Therefore
\begin{align}\nonumber
3\sum_{j=1}^g \left(\ell_j\cdot\sum_{i=1}^j k_i\right)\;&\leq\;
 3\sum_{j=1}^g \ell_jk_j(2^{j-1}+\dots+2^1+2^0)
\\\nonumber
  &=\;  3\sum_{j=1}^g k_j\ell_j(2^{j}-1)
\\\nonumber
  &\leq\; 3\max_{1\le i \le g} \{k_i\ell_i\}\cdot(2^1+2^2+\dots+2^g-g)
\\\label{eq:3M}
  &=\; 3\cdot(2^{g+1}-2-g)\cdot \max_{1\le i \le g} \{k_i\ell_i\} =:\chi\,.
\end{align}

We have thus shown that the plane embedding $G_g$ can be extended into
a drawing of $G$ with at most 
$\chi$ crossings as in \eqref{eq:3M}. It remains to show how such a drawing 
can be computed efficiently from an embedding of $G$ in $\Sigma_g$.
The algorithm runs two phases:
\begin{enumerate}
\item
A good planarizing sequence $(G_1^*,\gamma_1), \dots, (G_g^*,\gamma_g)$ for
$G^*$ is computed using $g$ calls to an $\ca O(n\log n)$ algorithm of Kutz~\cite{Ku06},
or to a faster $\ca O(n\log\log n)$ algorithm of Italiano et al.~\cite{ItalianoNSW11},
which both can find a cycle witnessing nonseparating edge-width in orientable surfaces.
These runtime bounds assume $g$ fixed.
During the computation, we represent $G^*$ by its rotation scheme which
allows a sufficiently fast implementation of the cutting operation as well.
\item
In the planar graph $G_g$, optimal
insertion routes are found for all the missing edges $F=E(G)\sem E(G_g)$
independently using linear-time breadth-first search in $G_g^*$.
A key observation with respect to runtime 
is that we are looking for these insertion routes only between
the predefined pairs of faces $a_g(f)$ and $b_g(f)$ for each $f\in F$, and
each of $\{a_g(f) \ : f\in F_i\ \}$ and $\{b_g(f) \ : f\in F_i\ \}$
has at most $2^{g-i}$ elements for $i=1,2,\ldots,g$.
(From the practical point of view, it may be worthwhile to mention that $|G_g|$
also serves as a natural upper bound for the number of considered faces.)
It follows that we need to perform searches for at most $2^{g-1}+\dots+2^1+2^0<2^g$
routes in total (independently of $|F|$), 
a process that takes an overall linear time for fixed~$g$.
Further, by a routine post-processing after each computed route we ensure
that no two routes cross each other more than once, again in total linear time
since the number of pairs of compared routes is bounded in~$g$.

It follows that we need to perform at most $2^{g-1}+\dots+2^1+2^0<2^g$ searches in total 
(independently of $|F|$), a process that takes an overall linear time for fixed~$g$.
Then, a routine algorithm inserts the individual edges of $F$ along the
computed routes in $G_g$, making dummy vertices for the induced crossings.
This routine takes $\ca O(n+\chi)$ steps where $\chi$ is as in \eqref{eq:3M}.
Moreover, $\chi=\ca O(\crg(G)\Delta^2)$ by Lemma~\ref{lem:kl-to-tex},
and $\crg(G)=\ca O(n\Delta)$ by \cite{DV12}.
\end{enumerate}
In view of this, the overall runtime of the algorithm is 
$\ca O\big(n(\log\log n+\Delta^3)\big)$ for each fixed $g$.
\end{proof}

\section{More properties of stretch}
\label{sec:more}

In this section, we establish several basic properties on the stretch
of an embedded graph. Even though we could have alternatively included
these in the next section, as we only require them in the proof of
Lemma~\ref{lem:kl-to-stretch}, we prefer to present them in a separate
section, for an easier further reference of the basic properties of
this new parameter which may be of independent interest.

We recall that a graph property $\ca P$ satisfies the 
{\em$3$-path condition} (cf.\ \cite[Section~4.3]{MT01}) if the following holds:
Let $T$ be a {\em theta graph} (a union of three internally disjoint paths with
common endpoints) such that two of the three cycles of $T$ do not possess $\ca P$;
then neither does the third cycle.
In the proof of the following lemma we make use of halfedges. A {\em halfedge} is a pair
$\langle e,v\rangle$ (``$e$~at~$v$''), where $e$ is an edge and $v$ is
one of the two ends of $e$.

\begin{lemma}\label{lem:3pp}
Let $G$ be embedded on an orientable surface, and let $C$ be a
cycle of $G$. The set of cycles of $G$ satisfies the $3$-path
condition for the property of odd-leaping $C$.
Furthermore, not all three cycles in any theta subgraph of $G$ can be odd-leaping $C$.
\end{lemma}

\begin{proof}
Let a theta graph $T\subseteq G$ be formed by three paths
$T=T_1\cup T_2\cup T_3$ connecting the vertices $s,t$ in $G$.
If $C$ and $T$ are disjoint, the lemma obviously holds.
We consider a connected component $M$ of $C\cap T$.
If $M=C$, then the $3$-path condition again trivially holds.
Otherwise, $M$ is a path with ends $m_1,m_2$ in $G$.
We denote by $f_1,f_2$ the edges in $E(C)\sem E(M)$ incident with
$m_1,m_2$, respectively, and by $M^+$
the union of $M$ and the halfedges $\langle f_1,m_1\rangle$ and $\langle f_2,m_2\rangle$.
We show that the number $q$ of leaps of $M^+$ summed
over all three cycles in $T$ is always even.

If $m_i\not\in\{s,t\}$ for $i\in\{1,2\}$,
then contracting the edge of $M$ incident to $m_i$ 
clearly does not change the number~$q$.
Iteratively applying this argument, we can assume
that finally either (i) $m_1=m_2$ (and possibly $m_1\in\{s,t\}$),
or (ii) $m_1=s$, $m_2=t$, and $M=T_1$.
In case (i), $M^+$ leaps either none or two of the cycles of $T$ in
the single vertex $m_1$, and so $q\in\{0,2\}$. Thus we assume for the
rest of the proof that (ii) holds.

For $i=1,2,3$, let $e_i$ (respectively, $e_i'$) be the edge of $T_i$
incident with $s$ (respectively, $t$). By relabeling $e_1, e_2, e_3$
if needed, we may assume that the rotation around $s$ is one
of the cyclic 
permutations $(e_1,f_1,e_2,e_3)$ or $(e_1,e_2,f_1,e_3)$. The rotation
around $t$ could be any of the six cyclic permutations
of $e_1',e_2',e_3',f_2$. This yields a total of twelve possibilities
to explore. A routine analysis shows that in every case we get
$q\in\{0,2\}$, except for the case in which the rotation around $s$ is
$(e_1,e_2,f_1,e_3)$ and the rotation around $t$ is
$(e_1',e_2',f_2,e_3')$; in this case, 
$M^+$ leaps twice the cycle $T_2\cup T_3$, and $q=4$.

Altogether, the number of leaps of $C$ summed over all three cycles in
$T$ is even.
Hence the number of cycles of $T$ which are odd-leaping with $C$ 
is also even, and the $3$-path condition follows.
\end{proof}

The next claim shows that stretch (Definition~\ref{def:stretch}) could have
been equivalently defined as an {\em odd-stretch}, using pairs of odd-leaping
cycles instead of one-leaping cycles.

\begin{lemma}[Odd-stretch equals stretch]
\label{lem:odd-stretch}
Let $G$ be a graph embedded in an orientable surface. 
If $C,D$ is an odd-leaping pair of cycles in $G$,
then $\stretch(G)\leq\len(C)\cdot\len(D)$.
\end{lemma}

\begin{proof}
We choose an odd-leaping pair $C,D$ that minimizes $\len(C)\cdot\len(D)$.
Up to symmetry, $\len(C)\leq\len(D)$.
Since $C\cap D\not=\emptyset$, there is a set
$\ca D=\{D_1,\dots,D_k\}$ of pairwise edge-disjoint
$C$-ears in $D$, such that 
$E(D_1)\cup\dots\cup E(D_k)=E(D)\sem E(C)$.
By a simple parity argument, there exists a $C$-switching ear in $\ca D$.
Hence if $|\ca D|=1$, then $C,D$ are one-leaping,
and the lemma immediately follows.

If more than one $C$-ear in $\ca D$ is switching, then we pick,
say, $D_1$ as the shorter of these.
By the choice of $D$ we have $\len(D_1)\leq\frac12\len(D)$, and so by
Lemma~\ref{lem:thstr} we have
$$\stretch(G)\leq\len(C)\cdot\left(\len(D_1)+\frac12\len(C)\right)
 \leq\len(C)\cdot\left(\frac12\len(D)+\frac12\len(D)\right) =
 \len(C)\cdot\len(D)
\,.$$

In the remaining case, we have that $|\ca D|>1$ and exactly one
$C$-ear in $\ca D$, say $D_1$, is switching.
Note that $D'$, the cycle formed by $D_1$ and a shorter section of~$C$,
is one-leaping~$C$, and it is thus enough to show that $\len(D')\leq\len(D)$.
Let $d_j$ for $j\in\{1,\dots,k\}$ denote the distance on $C$ between the
ends of~$D_j$.
Assume that for some $j\in\{2,\dots,k\}$ it is $d_j>\len(D_j)$.
Then both cycles of $C\cup D_j$ containing $D_j$ are shorter than $\len(C)$,
and one of them is odd-leaping with $D$ by Lemma~\ref{lem:3pp}.
This contradicts the choice of~$C$ (for the pair $C,D$, that is).
Hence $\len(D_j)\geq d_j$ for all $j\in\{2,\dots,k\}$.
Now, let $D''\subseteq D$ be the path formed by edges of $E(D)\sem E(D_1)$.
Then $D''$ has the same ends as~$D_1$, and since $\len(D_j)\geq d_j$ for
$j\in\{2,\dots,k\}$, we get that $\len(D'')=\len(D)-\len(D_1)\geq d_1$.
Altogether,
\begin{equation}
\stretch(G) \leq \len(C)\cdot\len(D')
 = \len(C)\cdot\left(\len(D_1)+d_1\right)
  \leq \len(C)\cdot\len(D).\tag*{\qedhere}\end{equation}
\end{proof}

\begin{lemma}
\label{lem:cutdew}
Let $H$ be a graph embedded in an orientable surface of genus $g\geq2$,
and let $A,B\subseteq H$ be a one-leaping pair of cycles
witnessing the stretch of $H$, such that
$\len(A)\leq\len(B)$.
Then $\ewn(H\cutt A)\geq\frac12\ewn(H)$.
\end{lemma}

\begin{proof}
Let $C$ be a nonseparating cycle in $H\cutt A$ of length $\ewn(H\cutt A)$.
If its lift $\hat C$ is a cycle again, then (since $\hat C$ is
nonseparating in $H$) 
$\ewn(H)\leq\len(\hat C)=\ewn(H\cutt A)$, and we are done.
Thus we may assume that $\hat C$ contains an $A$-ear $P\subseteq\hat C$
such that $A\cup P$ is a theta graph. Let 
$A_1,A_2\subseteq A$ be the subpaths into which the ends of $P$ divide $A$.
By Lemma~\ref{lem:3pp}, exactly two of the three cycles of $A\cup P$ are
odd-leaping with $B$. One of these cycles is $A$; let the other one, without loss of generality, be
$A_1\cup P$.
Then $\len(A_1\cup P)\geq\len(A)$ using Lemma~\ref{lem:odd-stretch},
and so $\len(P)\geq\len(A_2)$.
Furthermore, $A_2\cup P$ is nonseparating in $H$, and we conclude that
\begin{equation}\ewn(H)\leq\len(A_2\cup P)\leq
 2\len(P)\leq2\len(\hat C)=2\ewn(H\cutt A)
\,.\tag*{\qedhere}\end{equation}
\end{proof}

\smallskip
At this point, one may wonder why we do not use the
cutting paradigm 
as in Lemma~\ref{lem:cutdew} in a good planarizing sequence for
Theorem~\ref{thm:upper-cr} (Section~\ref{sec:drawing-upper}).
Indeed, it would seem that the same proof as in Section~\ref{sec:drawing-upper} works
in this new setting, and the added benefit would be an immediately matching
lower bound in the form provided by Corollary~\ref{cor:agrid-all}.
The caveat is that the proof of Theorem~\ref{thm:upper-cr} strongly
uses the fact that subsequent
cuts in a planarizing sequence do not involve {\em much fewer} edges
(recall ``$k_i\leq2^{j-i}k_j$ for all $1\leq i<j\leq g$'' from the proof).
If one cuts along the shortest cycle of a pair that witnesses the dual
stretch, then
the number of cut edges may jump up or down arbitrarily. Thus an attempted proof
along the lines of the proof we gave in
Section~\ref{sec:drawing-upper} would (inevitably?) fail at this
point.

\section{Proof of Lemma~\ref{lem:kl-to-stretch}}
\label{sec:finding}

Our aim in this section is to prove Lemma~\ref{lem:kl-to-stretch}. For easy
reference within this section, let us repeat its statement here:

{\def\thethm{\ref{lem:kl-to-stretch}}\addtocounter{thm}{-1}%
\begin{lemma}
Let $H$ be a graph embedded in the surface $\Sigma_g$.
Let $k:=\ewnd(H)$, and
let $\ell$ be the largest integer such that
there is a cycle $\gamma$ of length $k$ in $H^*$ whose shortest
$\gamma$-switching ear has length~$\ell$.
Assume $k\geq2^g$.
Then there exists an integer $g'$, $0< g'\leq g$, and a subgraph $H'$
of $H$ embedded in $\Sigma_{g'}$ such that 
$$
\ewnd(H')\geq 2^{g'-g}k
	\qquad\mbox{and}\qquad
\stretchd(H')\geq 2^{2g'-2g}\cdot k\ell
\,.$$
\end{lemma}}

We show that this lemma follows (quite easily, in fact, as we will see shortly) from
the statement of coming Lemma~\ref{lem:cutstep}, that involves the concept of {\em polarity} of a subgraph of an embedded graph. The proof of this auxiliary lemma will be presented in the next section.

Even though we might formally simply give the definition of polarity, state
Lemma~\ref{lem:cutstep}, and then give the proof of
Lemma~\ref{lem:kl-to-stretch}, it seems worthwhile to first devote a little time
to explaining the intuition behind the proof.  In particular, this will give us
the opportunity to argue how the notion of polarity arises naturally in the
process.

\subsection{Intuition}\label{subsec:intuition}

Recall that in the statement of Lemma~\ref{lem:kl-to-stretch} we have got a dual
cycle $\gamma$ that attains the dual edge-width $k:=\ewn(H^*)$, and a
$\gamma$-switching ear (say $\omega$) of length $\ell$.  One way to read the
lemma is the following.  There is no reason why $\gamma$ and $\omega$ should
witness $\stretchd(H)$; however, there is always a subgraph $H'$ of $H$,
embedded in a surface of genus $g'$ with $0<g'\le g$, such that
$\stretchd(H')$ is at least a constant times $k\ell$.

Now if $\stretchd(H)$ is already witnessed by $\gamma$ and a cycle
constructed with $\omega$ (and possibly a part of $\gamma$), then we are
done at once by letting $H':=H$.  Thus suppose that $\stretchd(H)$ is
witnessed by another pair $\alpha,\beta$ of dual cycles, with $\len(\alpha)\le
\len(\beta)$.  The idea is then to cut $H$ (and hence its host surface) along
$\alpha$, and analyze the possible outcomes.

Suppose that we cut along $\alpha$, and $\gamma,\omega$ remain intact but
still do not witness the stretch of the resulting graph.  Moreover, suppose
we repeatedly apply this process (keeping the good luck of affecting neither
$\gamma$ nor $\omega$, at any step) until we reach the torus.  Then it is
easy to see that we are done (by a repeated application of
Lemma~\ref{lem:cutdew}) by setting $H'$ to be the toroidal subgraph of $H$
obtained at this point.

The difficulty arises if, at some point in this process, we cut along a dual
cycle that intersects $\gamma$ or $\omega$ (or both).  
This is not necessarily bad; imagine, say, that $\alpha$ may be one- or odd-leaping
$\gamma$ (or a cycle constructed from $\omega$ and a part of~$\gamma$).
Then we can argue, using the technical tools from the previous sections,
that $\len(\alpha)\cdot\len(\beta)$ cannot be much smaller than $k\ell$.
On the other hand, there are situations (such as one in~Figure~\ref{fig:fig0})
in which we might seem to be doomed, since we get no usable relation between
$\len(\alpha)\cdot\len(\beta)$ and $k\ell$ straight away and, moreover,
we do not inherit from $\gamma$ any usable dual cycle which would host both
ends of~$\omega$ (to continue the cutting process).

The key point that saves the day is that, regardless of what happens to this
dual cycle $\gamma$ (either at the first cutting step, or at later ones),
the structure inherited from $\gamma$ still maintains enough resemblance to
a two-sided cycle, in the sense that we can still give meaningful sense to the idea of
a $\gamma$-switching ear.

To illustrate this idea, consider the scenario given in
Figure~\ref{fig:fig0}.  On the left hand side we have the dual cycle
$\gamma$, a $\gamma$-switching ear $\omega$, and a dual cycle $\alpha$.  In
this example $\gamma$ intersects $\alpha$, and so after cutting $H$ through
$\alpha$ to obtain $H\cutt\alpha$ (and $(H\cutt\alpha)^*$), the dual
subgraph $\gamma'$ inherited from $\gamma$ consists of two cycles (see the
right hand side of the figure).  Now in this (still relatively easy)
scenario we have that $\omega$ survives the process intact; in general this
is not the case, but let us assume this for the current illustration
purposes in which we focus on what happens to $\gamma$.

\begin{figure}
\def\svgwidth{0.48\textwidth}
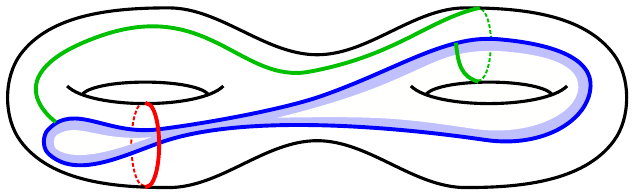
\hfill
\def\svgwidth{0.48\textwidth}
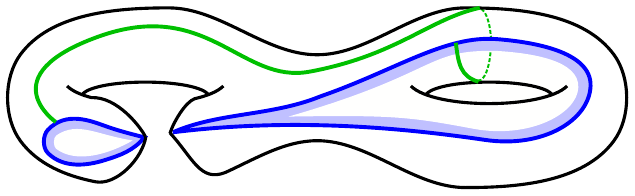
\caption{On the left hand side we show a cycle $\gamma$ and a $\gamma$-switching ear $\omega$.
  When we cut along $\alpha$, $\gamma$~gets broken into two pieces (right hand side of the figure), but even in this disconnected subgraph the notion of a ``positive side'' and a ``negative side'' is meaningful, naturally inherited from $\gamma$.}
\label{fig:fig0}
\end{figure}

Continuing in this example of Figure~\ref{fig:fig0}, we note that we cannot meaningfully
say that $\omega$ is a $\gamma'$-switching ear, as $\gamma'$ is not a cycle. 
However, at a closer inspection we note that the property that $\gamma$
``has a left-side and right-side'' is inherited to $\gamma'$.  (With an eye
on things to come, let us refer to these instead as a ``positive side'' and
a ``negative side'', yielding the idea that every cycle has {\em polarity}). 
In the figure we illustrate with a shade one of the sides of $\gamma$ (the
other side is unshaded), and we see that these naturally yield meaningful
``sides'' of $\gamma'$.  Thus $\gamma'$ inherits from $\gamma$ its polarity,
that is, a ``positive side'' and a ``negative side'' for each component of $\gamma'$.

It goes without saying that the scenarios one could encounter during the
cutting process could be considerably more complicated.
However, the crucial point is that, as we will see later, 
this polarity property makes good and consistent sense throughout the
whole cutting process. Informally, at each step we can keep
track of the original sides of~$\gamma$ even when $\gamma$ itself is
shattered into pieces.

In a nutshell, the proof of Lemma~\ref{lem:kl-to-stretch} then consists of
following the subgraph induced by $\gamma$ along the cutting process, and
showing that at some point we can successfully stop the process since the
dual cycles witnessing the dual stretch are long enough in terms of the
lengths of original $\gamma$ and $\omega$, as required in the statement
of the lemma.
This informal explanation now allows us to smoothly proceed to a formal
definition of the polarity concept, and to a statement of the workhorse
behind the proof (namely Lemma~\ref{lem:cutstep}).
Lemma~\ref{lem:kl-to-stretch} then follows as a rather easy consequence.

\subsection{Polarity}\label{subsec:polarity}

Let $G$ be a connected graph embedded in a surface $\Sigma$. If $D$ is a
(not necessarily connected) subgraph of $G$, then we may regard $D$ as an
embedded graph on its own right, by removing all the edges and vertices of
$G$ that are not in $D$ and inherting the corresponding restriction of the
rotation scheme of~$G$.
Athough, note that each connected component of $D$ would be embedded on its ``own''
surface, which is not necessarily~$\Sigma$. 
For our purpose, it is enough that this view consistently identifies the
facial walks of~$D$ (with respect to~$G$).

A {\em sign assignment} on $D$ is a mapping $\Gamma$ that assigns to each
facial walk of $D$ a sign $+$ or $-$ (thus making each facial walk {\em
positive} or {\em negative}).  A sign assignment is {\em bipolar} if for
each edge $e$ of $D$, one of the facial walks incident with $e$ is positive,
and the other is negative.  If $D$ has a bipolar sign assignment, then we
say that $D$ is {\em bipolar}.  The simplest example of a bipolar subgraph
is a two-sided cycle; under the current framework, a cycle has two facial walks (its
``sides''), and by making one of this facial walks positive, and the other
one negative, we obtain a bipolar sign assignment.  It is easy to see that
if $D$ is bipolar then $D$ is Eulerian.

We now consider a fixed bipolar sign assignment of a subgraph $D$ of an
embedded graph $G$, and let $e=uv$ be an edge in $G$ that is not in $D$, but
is incident with a vertex $u$ in $D$.  We use the common artifice of
interpreting an edge $e$ as being the union of two {\em half-edges}, one half-edge
$h_u$ incident with $u$, and one half-edge $h_v$ incident with $v$; each
half-edge is incident then with exactly one vertex, and its other end is a
loose end that attaches to no vertex.  
The half-edge $h_u$ is then incident with exactly one facial walk of $D$
(and the same possibly holds for $h_v$ if $v$ is also in~$D$).  
In this situation we speak about {\em$D$-polarity} of~$h_u$:
if the facial walk which $h_u$ is incident with is positive (respectively,
negative), then we say that $h_u$ itself is {\em $D$-positive}
(respectively, $D$-{\em negative}).

\begin{remark}
{\sl Formally, one should not say that a half-edge (or an edge) is
$D$-positive or $D$-negative, as this depends not only on $D$ but on the sign
assignment $\Gamma$ under consideration; we should then say something like
``$\Gamma$-negative'' instead.  This complication will turn out to be
unnecessary, as for each subgraph we handle we will consider only one
fixed sign assignment.}
\end{remark}

For the rest of this subsection, $D$ is a subgraph of an embedded graph $G$, and (in line with the previous Remark) we work under a fixed sign assignment for $D$.

As hinted in the informal discussion in Section~\ref{subsec:intuition}, we
need to extend the notion of switching, which we defined for cycles, to the
arbitrary bipolar subgraph $D$ of $G$ under consideration.  The definition,
as one would expect, is that a $D$-ear $P$ is {\em $D$-switching} if one
end-half-edge of $P$ is $D$-positive, and the other end-half-edge is
$D$-negative.

We also need to extend the concept of leaping, from cycles to the arbitrary bipolar subgraph $D$ of $G$. Roughly speaking the idea is (as with cycles) that as we traverse a walk we suddenly ``enter'' $D$, stay on $D$ for a while, and then leave $D$. If the half-edge in the walk just before entering $D$ and the half-edge in the walk just after leaving $D$ are of distinct polarities, then the subwalk that we traversed inside $D$ is a {leap}.

To define this formally, let $W=v_0e_0v_1\ldots e_{n-1}v_n$ be a walk in $G$. We remark that in a walk, repetitions of vertices and edges are allowed. If $v_0=v_n$ then we read indices modulo $n$, so that, for instance, we consider $v_i e_i v_{i+1}\ldots e_{n-1} v_0 e_0 v_1\ldots e_{j-1} v_j$ a valid subwalk of $W$, for any $i,j\in\{0,1,\ldots,n-1\}$ with $i>j$.

Now let $M=v_k e_k v_{k+1} \ldots e_{\ell-1} v_\ell$ be a maximal subwalk of $W$ contained in $D$. That is, (i) $M$ is a subwalk of $W$; (ii) regarded as a subgraph of $G$, $M$ is a subgraph of $D$; and (iii) neither $e_{k-1}$ nor $e_\ell$ are in $D$. Then $M$ is a {\em leap} (of $W$ and $D$) if the half-edge of $e_{k-1}$ incident with $v_k$, and the half-edge of $e_\ell$ incident with $v_\ell$, have distinct polarities.

We say that the walk $W$ is {\em odd-leaping} $D$ if the number of subwalks
of $W$ which are leaps is odd; otherwise $W$ is {\em even-leaping} $D$.  
The following observation is worth highlighting.
Assume that $G$ contains no $D$-switching ear, and $W$ is a closed walk in~$G$.
Then, traversing $W$, we must encounter leaps of $D$ in an alternating
manner -- leaps from positive to negative $D$-polarity followed by leaps
from negative to positive $D$-polarity, and vice versa.
Hence the total number of leaps along closed~$W$ must be even in such a case.
We can thus conclude:
\begin{remark}
\label{rem:switchingleap}
{\sl 
If there is a closed walk that odd-leaps $D$, then there
exists a $D$-switching ear.}
\end{remark}

\subsection{The workhorse}

With the notion of polarity formally laid out, we can now proceed with the proof of Lemma~\ref{lem:kl-to-stretch}. As we briefly outlined in Section~\ref{subsec:intuition}, the idea is to start with the dual cycle $\gamma$, and iteratively keep cutting along a dual cycle (the short one) witnessing the dual stretch of the current graph, until we reach a graph $H'$ with the conditions required in the lemma.

The workhorse behind the proof is Lemma~\ref{lem:cutstep} below, which keeps
track of (the remains of) a bipolar dual subgraph $\delta$ as we go
through the cutting process.  Let us now state this
auxiliary lemma and then, before proceeding to the proof of Lemma~\ref{lem:kl-to-stretch}, 
have an informal discussion on how we make use of it.

\begin{lemma}
\label{lem:cutstep}
Let $H$ be a graph embedded in an orientable surface $\Sigma$ of genus $g$. Suppose that:
\begin{itemize}\parskip0pt
\item[(a)] $g\ge 1$;
\item[(b)] $\delta$ is a bipolar dual subgraph of $H^*$;
\item[(c)] there exists a closed walk in $H^*$ odd-leaping $\delta$; and
\item[(d)] the minimum length of a $\delta$-switching ear in $H^*$ equals~$\oldh$.
\end{itemize}
Let $\alpha,\beta$ be a one-leaping pair of dual cycles in $H^*$ such that
$\len(\alpha)\leq\len(\beta)$ and
$\stretchd(H)=\len(\alpha)\cdot\len(\beta)$.  
If $\len(\beta)<\oldh$, then all the following hold:
\begin{itemize}\parskip0pt
\item[(a')] $g\ge2$ (and hence the genus of $H\cutt\alpha$ is $\geq1$);
\item[(b')] there is a bipolar dual subgraph $\delta_2$ of $(H\cutt\alpha)^*$;
\item[(c')] there exists a closed walk in $(H\cutt\alpha)^*$ odd-leaping $\delta_2$; and
\item[(d')] the minimum length of a $\delta_2$-switching ear in $(H\cutt\alpha)^*$
	is $\oldh_2\geq \oldh-\frac12\len(\alpha)$.
\end{itemize}
\end{lemma}

In this statement, whose proof is deferred to Section~\ref{sec:findingB},
condition (d) is well-defined since (c) implies the existence of a
$\delta$-switching ear.
Besides, it
might be odd-looking that the objects in (b'), (c'), and (d') are labelled
$\delta_2$ and $\oldh_2$ (instead of, say, $\delta_1,\oldh_1$ or
$\delta',\oldh'$).  This is intentional, as in the proof we wish to
reserve the notation $\delta_1$ for an intermediate object we use to arrive
from $\delta$ to $\delta_2$.

The idea to prove Lemma~\ref{lem:kl-to-stretch} is to apply
Lemma~\ref{lem:cutstep} iteratively.  We start by letting $\delta$ be the
dual cycle $\gamma$ with an arbitrary bipolar sign assignment, and keep iteratively
applying Lemma~\ref{lem:cutstep} to the resulting bipolar graph $\delta_2$
in place of~$\delta$, each step replacing $H$ with $H\cutt\alpha$. 
One should note that this iterative process is not our objective by itself, 
but only a means to eventually violate the assumption $\len(\beta)<\oldh$.  
The reason for which it is a desirable outcome will
become clear in the upcoming short proof of Lemma~\ref{lem:kl-to-stretch};
informally, it yields a situation in which we can jump into the conclusion
that $\len(\alpha)\cdot\len(\beta)$ is not much smaller than $k\ell$.
On the other hand, it is important to make it clear why $\len(\beta)<\oldh$
must be violated, at some point.
This is simply because the genus $g$ of our graph $H$ is finite at the beginning,
and at every iteration we decrease it by~$1$,
hence eventually making the only other option ``(a')~$g\ge2$~\dots'' fail.

We are now ready to present the proof of Lemma~\ref{lem:kl-to-stretch}.

\begin{proof}[\bf Proof of Lemma~\ref{lem:kl-to-stretch}]

We proceed by iteratively using Lemma~\ref{lem:cutstep}.
Notice that all the conditions (a),(b),(c),(d) of Lemma~\ref{lem:cutstep} are
satisfied by the graph $H$, its bipolar dual cycle $\delta:=\gamma$,
and by $\oldh:=\ell$.
Let $H_0=H$ and $\gamma_0=\gamma,\, \ell_0=\ell,\, g_0=g$.

For $i=1,2,\dots$, we apply Lemma~\ref{lem:cutstep} to $H:=H_{i-1}$ and
$\delta:=\gamma_{i-1},\, \oldh:=\ell_{i-1}$,
assuming that $\len(\beta)< \oldh$ still holds. 
So, we can set $H_{i}:=H\cutt\alpha$
and $\gamma_{i}:=\delta_2,\, \ell_i:=\oldh_2$,
and the conditions (a),(b),(c), (d) of Lemma~\ref{lem:cutstep} are again 
satisfied by those (hence leaving room for the next iteration).
Since the genus of $H_{i-1}$ is $g_0-i+1$, the condition (a')~$g\ge2$
of Lemma~\ref{lem:cutstep} surely fails at iteration $i=g_0$,
and hence this iterative process must stop after less than $g_0$ iterations
-- this is only possible with achieving $\len(\beta)\geq \oldh$.

In a summary,
after the last successful iteration number $i<g_0$, we have got:
\begin{itemize}
\item
 the graph $H_i$ (a subgraph of~$H$) which is of genus $g_0-i\geq1$,
\item
 its nonseparating dual edge-width is
 $\ewnd(H_i)\geq2^{-i}\cdot\ewnd(H_0)>1$ which
 follows by iterating Lemma~\ref{lem:cutdew} $i$ times,
\item
 the shortest $\gamma_i$-switching ear in $H_i^*$ has
 length at least $\ell_i\geq 2^{-i}\cdot\ell$, since one can iterate (c')
  $\oldh_2\geq \oldh-\frac12\len(\alpha)\geq \oldh-\frac12\len(\beta)\geq \frac12\oldh$ at
 each of the previous $i$ steps, and
\item
 there exists a one-leaping pair of dual cycles $\alpha,\beta$ in $H_i^*$ such that
 $\len(\alpha)\leq\len(\beta)$, $\stretchd(H_i)=\len(\alpha)\cdot\len(\beta)$,
 and $\len(\beta)\geq \oldh=\ell_i$ hold.
\end{itemize}

Consequently,
$$
\stretchd(H_i)=\len(\alpha)\cdot\len(\beta)\geq
	\ewnd(H_i)\cdot \ell_i\geq 2^{-i}\ewnd(H)\cdot2^{-i}\ell
	= 2^{-2i}\cdot k\ell
\,.$$
By setting $H'=H_i$ for $g'=g_0-i$, Lemma~\ref{lem:kl-to-stretch} follows.
\end{proof}

\subsection{A few facts on polarity}\label{subsec:oddleap}

We close this section by stating a few simple facts around the notion of polarity. 
These facts will be used in the proof of Lemma~\ref{lem:cutstep}.

\begin{observation}\label{o:col}
Let $G$ be a graph embedded on a surface $\Sigma$, and let $D$ be a bipolar subgraph of $G$, with a fixed bipolar sign assignment. Then the following hold:
\begin{enumerate}
\item\label{a} If $e$ is a non-loop edge of $D$, then $D/e$ is a bipolar subgraph of $G/e$.
\item\label{b} If a non-loop edge $e$ does not belong to $D$, and is not a $D$-switching ear, then $D/e$ is a bipolar subgraph of $G/e$.
\item\label{e} If there is a closed walk that odd-leaps $D$, then there exists a $D$-switching ear.
\item\label{f} Suppose that $W,W'$ are closed walks in $G$, such that
  $W$ odd-leaps $D$ and $W'$ even-leaps $D$.
  Suppose further that $W$ and $W'$ have a common vertex $v$. Then the concatenation of $W$ and $W'$ (that is, the walk obtained by starting at $v$, traversing $W$, and then traversing $W'$) is a closed walk that odd-leaps $D$.
\item\label{o:oddleap} If $\Sigma$ is the sphere, then no closed walk in $G$ odd-leaps $D$.
\end{enumerate}
\end{observation}

All these facts follow from the definition of polarity. Facts (\ref{a}) and (\ref{b}) are totally straightforward. Fact (\ref{e}) was actually already noted at the end of Section~\ref{subsec:polarity}. Fact (\ref{f}) follows by an easy case analysis; it also follows easily using routine surface homology arguments.

The less straigthforward of these is perhaps Fact (\ref{o:oddleap}), but even this is hardly more than a simple exercise from the definition of polarity. First, note that if the faces of $G$ can be ``colored'' positive and negative, so that each $D$-positive (respectively, $D$-negative) facial walk is incident with a positive (respectively, negative) face, then (\ref{o:oddleap}) follows from a simple parity argument. Let us call this a {\em good} sign assignment on $D$; thus if the sign assignment on $D$ is good we are done. If the sign assignment on $D$ is not good we proceed as follows. Change the sign assignment (positive to negative, and vice-versa) of the facial walks of one connected component of $D$. It is easy to see that a closed walk odd-leaps $D$ with the original sign assignment if and only if it odd-leaps $D$ with the new sign assignment. We can then apply this sign-change to the connected components of $D$, as many times as needed, until we obtain a good sign assignment on $D$, and so (\ref{o:oddleap}) follows.

\section{Proof of Lemma~\ref{lem:cutstep}}
\label{sec:findingB}

In a nutshell, to prove the lemma we will show that the bipolar dual
subgraph $\delta_2$ naturally induced by the edges of $\delta$ that
survive in $(\Halpha)^*$, satisfies the required conditions 
(in particular, $\delta_2$ is bipolar, yielding (b')). 
As we will see, we may assume that $\delta_2$ is not trivial (that is, not an empty dual
subgraph), as otherwise the condition $\len(\beta)<\oldh$ in the statement 
of the lemma is violated.  
We note that then $\delta_2$ is clearly well-defined; every edge in $\Halpha$
corresponds naturally to an edge in $H$, and so every edge in $(\Halpha)^*$
corresponds naturally to an edge in $H^*$.

In order to obtain the closed walk required in (c'), we will make use of the
closed walk guaranteed from (c).  However, an obvious problem reveals itself
immediately: a closed walk (in particular, the one odd-leaping $\delta$) in
$H^*$ need not be a closed walk in $(\Halpha)^*$, since the dual cycle
$\alpha$ gets destroyed in the process of obtaining $(\Halpha)^*$.  This is
perhaps the most notorious difficulty that must be overcome, together with
the corresponding difficulty of trying to associate $\delta_2$-switching
ears in $(\Halpha)^*$ with $\delta$-switching ears back in $H^*$, in order
to prove (d).

These difficulties will be overcome by a detailed understanding of how the
structures in $H^*$ are affected by the removal of $\alpha$.  This
understanding will be often be aided via a specific ``intermediate'' embedded graph
$H_1$ (and its dual $H_1^*$).  Since we will be relating dual subgraphs and
walks from $H^*$ to their corresponding structures in $H_1^*$, and then to
their corresponding structures in $(H\cutt\alpha)^*$, it will greatly help
understanding the arguments if we denote $(H\cutt\alpha)^*$ simply by $H_2^*$.  In
this way we can follow the practice of labelling objects associated to
$H^*$ (and the primal graph $H$) without subscripts; then we can label
objects associated to $H_1^*$ (and its primal $H_1$) with the subscript
$1$ and objects associated to $H_2^*$ (and the primal $H_2$)
with the subscript $2$.

\begin{remark}
{\sl For the rest of the proof of Lemma~\ref{lem:cutstep}, 
we use $H_2$ to denote the graph $\Halpha$ that results from cutting $H$ along $\alpha$. 
With this convention, we have $H_2^*=(\Halpha)^*$.}
\end{remark}

\subsection{Setup}
\label{sub:setup}

The embedded graph $H_2=H\cutt\alpha$ is obtained by removing from $H$ all
the edges whose dual edges form $\alpha$, and then cutting the host surface
along the resulting cylinder.  We obtain the ``intermediate'' graph $H_1$ by
stopping this process short, removing all the edges of $E(\alpha)^*$
except for a single last edge.

Formally, let $F^*=\{f_1^*, f_2^*,\ldots,f_m^*\}\subseteq E(H^*)$ denote the set of
dual edges of the cycle $\alpha$, and let $F=\{f_1,f_2,\ldots,f_m\}$ be
the corresponding set of edges in the primal graph $H$.  We assume without
loss of generality that these edges are labelled so that they occur in this
cyclic order in $\alpha$.  Now the intermediate embedded graph $H_1$ is
obtained by pausing right before removing the last edge $f_1$ from $F$.  (As
we will see shortly, $f_1$ will be carefully chosen, but this is irrelevant
at this point).  That is, $H_1$ is simply $H\sem\{f_2,f_3,\ldots,f_m\}$
(embedded on~$\Sigma$).

We also make an observation regarding the dual $H_1^*$ of $H_1$. Since
the removal of an edge in a graph corresponds to the contraction of the
corresponding dual edge in the dual graph, it follows that in order to
obtain $H_1^*$ we contract from $H^*$ the edges
$f_2^*,f_3^*,\ldots,f_m^*$.

In the proof of (a'),(b'),(c') and (d') (of Lemma~\ref{lem:cutstep}) we
frequently analyze objects in $H_2^*$ via their corresponding objects in
$H_1^*$.  We refer the reader to Figure~\ref{fig:fig2} for an illustration
of this process of getting $H_1^*$ from $H^*$, and then $H_2^*$ from
$H_1^*$.  On the left-hand side of this figure, we depict the relevant parts
of $H$ and $H^*$; $H$ is drawn with thin edges, and $H^*$ with thick
(colored) edges. 
Inside the cylinder $C$ of $H$ (bounded by two cycles) we have the dual
cycle $\alpha$, and each edge $f_i$ in $H$ drawn across $C$ corresponds to a
dual edge $f_i^*$ in $\alpha$.  The middle part of this figure illustrates
$H_1$ and $H_1^*$: all the edges $f_m,f_{m-1},\ldots,f_2$ have been removed
(respectively, all the edges $f_m^*, f_{m-1}^*, \ldots, f_2^*$ have been
contracted).  The only edge drawn across $C$ that remains is $f_1$, and so
its corresponding dual edge $f_1^*$ is a loop-edge, based on a dual vertex
$j_1^*$.  Finally, on the right-hand side we depict $H_2$ and $H_2^*$.  The
graph $H_2$ is obtained by removing $f_1$, cutting inside the (now
edge-free) interior of $C$ and pasting disks to the resulting holes.  In the
dual $H_2^*$, each of this disks (which yield faces in $H_2$) yields a dual
vertex, thus obtaining a ``left'' dual vertex $\ell_2^*$ and a ``right''
dual vertex $r_2^*$.

\begin{figure}
        \def\svgwidth{0.3\textwidth}
        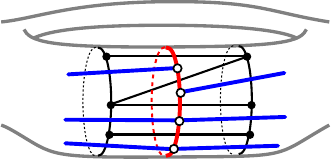\hfill
        \def\svgwidth{0.3\textwidth}
        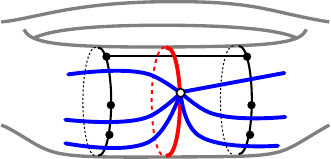\hfill
        \def\svgwidth{0.3\textwidth}
        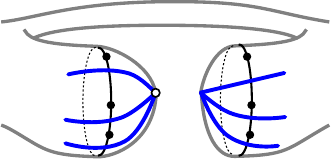
\caption{An illustration of the process of obtaining $H_1^*$ (center) from $H^*$ (left), and then $H_2^*$ (right) from $H_1^*$. The primal edges (that is, the edges of $H, H_1$, and $H_2$ are drawn thin, and the dual edges (those of $H^*, H_1^*$, and $H_2^*$) are thick.}
\label{fig:fig2}
\end{figure}

Now we pay close attention to the dual $H_2^*$.
We note that it may be correctly argued that $H_2^*$ is simply described as
the dual of $H_2$, and we have just described how to obtain $H_2$; however, for
our purposes it will be very helpful to visualize how $H_2^*$ is obtained
from $H_1^*$ without referring to~$H_2$.

We let the rotation around $j_1^*$
be $f_1^*, e_1^*, e_2^*, \ldots, e_s^*, f_1^*, g_1^*, g_2^*, \ldots, g_t^*$,
so that $e_1^*,e_2^*,\ldots,e_s^*$ lie to the left (relative to the loop
orientation) of $f_1^*$, and $g_1^*, g_2^*,\ldots, g_t^*$ lie to the right 
of $f_1^*$.
To get $H_2^*$ from $H_1^*$, the first step is to contract the edge
$f_1^*$, obtaining temporarily a pinched surface whose pinched point is the
vertex $j_1^*$.  The second step is to split naturally $j_1^*$ into two
vertices $\ell_2^*$ and $r_2^*$, where $\ell_2^*$ is incident with the left
edges $e_1^*,\ldots,e_s^*$, and $r_2^*$ is incident with the right edges
$g_1^*,\ldots,g_t^*$.  As a result of this two-step process we obtain
exactly the dual graph $H_2^*$.

Note that this is a valid process regardless of which edge $f_1$ we chose to
be the final edge to be removed from $\{f_1,f_2,\ldots,f_m\}$ (equivalently,
which dual edge $f_1^*$ we chose to be the final edge to be contracted from
$\{f_1^*,f_2^*,\ldots,f_m^*\}$).  However, for our purposes we choose $f_1^*$
(thus implicitly choosing $f_1$) so that it satisfies a particular
condition: {\em $f_1^*$ is not in $\beta$}.  The reason to choose $f_1^*$
with this property will become clear later in the proof.  For now
we just state that such an edge must exist, simply because
the fact that $\alpha,\beta$ is a one-leaping pair implies
that not every edge of $\alpha$ is also in $\beta$.

\paragraph{The key objects: $\delta_1$ and $\delta_2$. }

We let $\delta_1$ denote the dual subgraph in $H_1^*$ induced by $\delta$
(that is, induced by the edges of $\delta$ that are not in $\{f_m^*,
f_{m-1}^*, \ldots, f_2^*\}$).  We note that $\delta_1$ cannot be trivial,
that is, an empty subgraph of $H_1^*$, for this would imply that $\delta$
is contained in a path in $\alpha$ (namely the path formed by the edges
$f_2^*,f_3^*,\ldots,f_m^*$).  But this is impossible since $\delta$ is
bipolar and, as it follows immediately from the definition of polarity, no
path is bipolar.

The final goal in Lemma~\ref{lem:cutstep} will be to establish
properties of $\delta_2$, which is simply the subgraph of $H_2^*$ induced by
the edges in $\delta$ that are still edges in $H_2^*$
(i.e., those which survived the cutting process along~$\alpha$).  
These are also the edges of $\delta_1$ with the (possible) exception of $f_1^*$.

We claim that we may assume also $\delta_2$ to be nontrivial, that is, not an
empty subgraph of $H_2^*$.  To see this, we note that the only alternative
is that $\delta\subseteq \alpha$.  If this were the case then necessarily we
would have $\delta=\alpha$, since no proper subgraph of $\alpha$ is bipolar
(a path or a disjoint union of paths is never bipolar).  Now suppose that
$\delta=\alpha$.  Since $\alpha,\beta$ is a one-leaping pair, it follows
that $\beta$ contains an $\alpha$-switching ear (and thus a
$\delta$-switching ear, since $\delta=\alpha$).  By (d), this
$\delta$-switching ear has length at least $\oldh$, and thus $\len(\beta) \ge \oldh$,
violating the assumption $\len(\beta)<\oldh$.  Therefore, for the rest we may assume that
$\delta_2$ is not trivial.

In the previous paragraph we have used the argument that if $\beta$ contains a
$\delta$-switching ear, then (d) holds.  Similarly, if $\alpha$ contains a
$\delta$-switching ear, then it follows that
$\len(\alpha)\ge \oldh$, and since $\len(\beta)\ge \len(\alpha)$ this implies that
$\len(\beta)\ge \oldh$, again violating the assumption.  Thus we may assume that
neither $\alpha$ nor $\beta$ contain a $\delta$-switching ear.  This is
worth highlighting for future reference:
\begin{remark}\label{rem:noswitchingalpha}
{\sl We continue the proof  of Lemma~\ref{lem:cutstep} under the assumptions that: 
(i) $\alpha$ does not contain a $\delta$-switching ear; and 
(ii) $\beta$ does not contain a $\delta$-switching ear.} 
\end{remark}

\paragraph{Proof of (a').}
The claim $g\ge2$, which is equivalent to saying that $H_2^*$ is not embedded in the sphere, 
will automatically follow from later (c') under Observation~\ref{o:col}\,(\ref{o:oddleap}).

\subsection{Proof of (b')}

Here we show that $\delta_1$ in $H_1^*$ naturally inherits a bipolar sign assignment
from the bipolar sign assignment of $\delta$.  After this, we show that
$\delta_2$ in $H_2^*$ naturally inherits a bipolar sign assignment from the one of
$\delta_1$ (and thus, from the one of $\delta$).

Let $(H^{m-1}){{}^*},\ldots, (H^1){{}^*}$ be the dual graphs obtained from
$H^*$ by iterative contraction of the edges $f_m^*, f_{m-1}^*,\ldots,
f_2^*$.  Formally, we let $(H^m)^*:=H^*$, and let
$(H^j)^*:=(H^{j+1})^*/f_{j+1}^*$ for $j=m-1,m-2,\ldots,1$.  To keep track to
what happens to $\delta$ and $\alpha$ throughout this contraction process,
we analogously let $\delta^m:=\delta$ and $\alpha^m:=\alpha$, and
$\delta^{j}:=\delta^{j+1}/f_{j+1}^*$ and
$\alpha^{j}:=\alpha^{j+1}/f_{j+1}^*$ for $j=m-1,m-2,\ldots,1$.  Thus
$\delta_1=\delta^1$.

After each contraction step, we claim that $\delta^j$ has a natural bipolar sign
assignment inherited from $\delta^{j+1}$ (and thus all the way to the top,
from $\delta$).  
We argue from Observations~\ref{o:col}\,(\ref{a}) and (\ref{b}):
\begin{itemize}
\item If $f_{j+1}^*\in E(\delta^{j+1})$ then we simply apply (\ref{a}).
\item If $f_{j+1}^*\not\in E(\delta^{j+1})$ then we actually have
$f_{j+1}^*\in E(\alpha^{j+1})\setminus E(\delta^{j+1})$,
and we would like to apply (\ref{b}) here.
For this it is enough to show that $f_{j+1}^*$ is not a
$\delta^{j+1}$-switching ear.
From Remark~\ref{rem:noswitchingalpha} we know that $\alpha=\alpha^m$
contains no $\delta^m$-switching ear, and this is easily inherited down the
process, so that $\alpha^i$ contains no $\delta^i$-switching ear for
$i=m-1,\ldots,j+1$.
In particular, no edge in $E(\alpha^{j+1})\setminus E(\delta^{j+1})$ is a
$\delta^{j+1}$-switching ear.
\end{itemize}

Thus at the end of the process we obtain that the dual subgraph
$\delta_1=\delta^1$ of $H_1^*=(H^1)^*$ inherits all the way down from $\delta$ a
bipolar sign assignment.  Moreover, we have an additional piece of
information on the sign assignment of $\delta_1$, which turns out to be
crucial in the next step: if $f_{1}^*\not\in E(\delta_1)$, then $f_{1}^*$ is not
a $\delta_1$-switching ear.

In the next step, we show that $\delta_2$ gets a sign assignment naturally 
inherited from the one of~$\delta_1$.
We first consider the case of $f_{1}^*\in E(\delta_1)$.
Note that $f_1^*$ cannot appear twice in one facial walk of $\delta_1$ since
it is bipolar.
In this case the facial walks of $\delta_1$ and $\delta_2$ are the
same, with the following exception:
$f_1^*$ is removed from the two walks containing it (and such a walk
disappears from $\delta_2$ completely if it was formed only by $f_1^*$).
Hence, $\delta_2$ indeed inherits the bipolar sign assignment on~$\delta_1$.

We are left with the case of $f_{1}^*\not\in E(\delta_1)$.
Recall, from Section~\ref{sub:setup}, that the rotation of edges around
the end vertex of the loop-edge $f_1^*$ in $H_1^*$ is
$f_1^*, e_1^*, e_2^*, \ldots, e_s^*, f_1^*, g_1^*, g_2^*, \ldots, g_t^*$.
Let $a$ (respectively,~$b$) be the smallest (respectively, largest) index 
such that $e_a^*$ (respectively, $e_b^*$) is in $\delta_2$.  
Similarly, let $c$ (respectively,~$d$) be the smallest (respectively, largest) 
index such that $g_c^*$ (respectively,~$g_d^*$) belongs to $\delta_2$.
If there are no such edges of $\delta_2$, then we simply leave $a,b$ or $c,d$ undefined.

In this case we see that the facial walks of $\delta_1$ and $\delta_2$ are
again the same, unless all indices $a,b,c,d$ are defined.
In the latter case, there is a facial walk $U_1$ of
$\delta_1$ that includes $e_{a}^*$ followed by~$g_d^*$, and a facial walk
$W_1$ of $\delta_1$ that includes $e_b^*$ followed by $g_c^*$.  
Now (and this is essential), since $f_1^*$ is not a $\delta_1$-switching ear, it
follows that $U_1$ and $W_1$ have the same sign.
In $\delta_2$, on the other hand; instead of $U_1$ we have a facial walk $U_2$ that 
includes $e_a^*$ followed by $e_b^*$, and instead of $W_1$ we have a
facial walk $W_2$ that includes $g_c^*$ followed by $g_d^*$.  
Then we simply give to $U_2$ and to $W_2$ the common sign of $U_1$ and
$W_1$, which makes no change to local sign situation of any edge of $\delta_2$.
So, we are again done with $\delta_2$ inheriting the bipolar sign
assignment of~$\delta_1$ (which was itself naturally inherited from the one
on~$\delta$).

\subsection{Proof of (c')}

To prove (c') we take a closed walk $\omega$ that odd-leaps $\delta$, and
use $\omega$ to produce, via a closed walk that odd-leaps $\delta_1$, a
closed walk that odd-leaps $\delta_2$.  The main difficulty here is that the
subgraph in $H_2^*$ induced by $\omega$ may not be a closed walk, after the
cutting-through-$\alpha$ process.  To resolve this complication, we will need
to ``re-join'' the components of the subgraph induced by $\omega$ in
$H_2^*$, in such a way that the final result maintains the odd-leapiness
property.

The existence of a closed walk $\omega$ in $H^*$ that odd-leaps $\delta$
is hypothesis (c) in the statement of the lemma.  
We let $\omega_1$ denote the subgraph in $H_1^*$ induced by the edges of
$\omega$.  Since $\omega_1$ is obtained by contracting edges of $\omega$,
then $\omega_1$ indeed is a closed walk in $H_1^*$. 
We mention that $\omega_1$ might contain the edge $f_1^*$,
which we will resolve later.

Similarly as in the proof of (b'), 
to keep track of what happens to $\omega$ throughout the contraction
process, we let $\omega^m:=\omega$, and $\omega^{j}:=\omega^{j+1}/f_{j+1}^*$
for $j=m-1,m-2,\ldots,1$.  Note that $\omega_1=\omega^1$.
Using the fact (established in the proof of (b')) that $f_{j}^*$ is not a
$\delta^{j}$-switching ear for $j=m,m-1,\ldots,2$, it is easily seen that
the property that $\omega=\omega^m$ is a closed walk that odd-leaps
$\delta=\delta^m$ is inherited to $\omega^j,\delta^j$ for all $j=m-1,\ldots,1$.
Thus $\omega_1=\omega^1$ is a closed walk that odd-leaps $\delta_1=\delta^1$.

Now $\omega_1$ may contain the loop $f_1^*$, but since $f_1^*$ is not a
$\delta_1$-switching ear (again, we showed this in the proof of (b')), in
this case we may remove $f_1^*$ from $\omega_1$ (as many times as it occurs),
and it is immediately verified that the result is a closed walk $\psi_1$
that odd-leaps $\delta_1$ and does not contain $f_1^*$.

Having achieved the first intermediate goal of finding suitable $\psi_1$ that odd-leaps $\delta_1$, 
we moreover find a closed walk $\phi_1$ in $H_1^*$ that even-leaps~$\delta_1$; 
this $\phi_1$ will be useful in the ``re-joining'' process outlined above.  
Recall from Remark~\ref{rem:noswitchingalpha} that $\beta$ contains no
$\delta$-switching ear, and so $\beta$ even-leaps $\delta$
(cf.~Remark~\ref{rem:switchingleap}).
Similarly as above, we let $\phi^m:=\beta$, and
$\phi^{j}:=\phi^{j+1}/f_{j+1}^*$ for $j=m-1,m-2,\ldots,1$.
An identical argument to the one we used above to show that $\omega^j$ odd-leaps
$\delta^j$, now yields that each $\phi^{j}$ is a closed walk that even-leaps
$\delta^j$, for all $j=m-1,\ldots,1$.
We let $\phi_1:=\phi^{1}$.

We actually have (and will need) more information on the walk $\phi_1$. 
Since $\alpha$ and $\beta$ are in a one-leap position, it follows that 
(i) $\phi_1$ as a heir of $\beta$ passes through
the vertex $j_1^*$ incident with the loop-edge $f_1^*$ exactly once; 
and (ii) one of the edges in $\phi_1$ incident with $j_1^*$ is in
$\{e_1^*,e_2^*,\ldots,e_s^*\}$, and the other is in
$\{g_1^*,g_2^*,\ldots,g_t^*\}$.  
At an informal level, $\phi_1$ ``crosses''
(as opposed to ``tangentially intersects'') the loop edge $f_1^*$ at
$j_1^*$, and it does so only once.

We finally use $\psi_1$ and $\phi_1$ together to achieve the last goal,
constructing a closed walk $\tau_2$ in $H_2^*$ that odd-leaps $\delta_2$.
The overall idea is very simple. 
Denote by $\phi_2$ the open walk in $H_2^*$ induced by the edges of
$\phi_1$, with the ends $\ell_2^*$ and $r_2^*$ (recall that $j_1^*$ has been
split into $\ell_2^*$ and $r_2^*$ while constructing $H_2^*$).
Now, to construct the desired $\tau_2$ in $H_2^*$, 
we follow the edges of $\psi_1$ in order, and each time we encounter an
edge from $\{e_1^*,\ldots,e_s^*\}$ immediately followed by one from
$\{g_1^*,\ldots,g_t^*\}$ (or vice versa), we suspend at $\ell_2^*$ and
concatenate one whole turn of $\phi_2$ from $\ell_2^*$ to $r_2^*$
(respectively, one reversed turn of $\phi_2$ from $r_2^*$ to $\ell_2^*$).
Then we resume from $r_2^*$ (respectively, $\ell_2^*$) with the 
next edge of $\psi_1$, until finishing all of its edges.
This clearly results in a closed walk $\tau_2$ in $H_2^*$.

The last step is to prove that, indeed, $\tau_2$ odd-leaps $\delta_2$.
Let $\tau_1$ be the closed walk in $H_1^*$ induced by the edges of $\tau_2$
(i.e., the lift of $\tau_2$).
By the construction of $H_2^*$ and $\delta_2$ from $H_1^*$ and~$\delta_1$
(see in the proof of (b')), 
it is clear that $\tau_2$ odd-leaps $\delta_2$ iff $\tau_1$ odd-leaps~$\delta_1$.
The latter holds true simply by applying (possibly iteratively)
Observation~\ref{o:col}(\ref{f}) with $W:=\psi_1$, $W':=\phi_1$,
$D:=\delta$ and $v:=j_1^*$.

\subsection{Proof of (d')}

By (c'), we know that there is a closed walk in $H_2^*$ that
odd-leaps~$\delta_2$, and so there exists a $\delta_2$-switching ear
there by Observation~\ref{o:col}(\ref{e}).
Let $\sigma_2$ denote a shortest $\delta_2$-switching ear in $H_2^*$.
To prove (d'), it suffices to show that then there
exists a $\delta$-switching ear $\varrho$ in~$H^*$ of length at most
$\len(\sigma_2)+\frac12\sizealpha$.

Let $u^*$ and $v^*$ be the ends of~$\sigma_2$, and let
$h_u^*$ and $h_v^*$ be their corresponding end-half-edges.
The strategy here is to understand the lift of $\sigma_2$ in $H^*$,
that is the subgraph $\sigma$ of $H^*$ induced by the edges of $\sigma_2$.  
First, note that $\sigma$ in $H^*$ is not necessarily connected;
in fact, it may consist of up to three components if both
of the dual vertices $\ell_2^*,r_2^*$ belong to~$\sigma_2$
(cf.~Figure~\ref{fig:fig2}).
On the other hand, it is easy to see that the ``breaking points'' of
$\sigma$ are all incident to $\alpha$.
More precisely, $\sigma\cup\alpha$ is connected.
Second, we pay attention to $\delta_2$ and $\delta$;
clearly, $\delta_2$ contains one or both of $\ell_2^*,r_2^*$
if and only if $\delta$ intersects $\alpha$ in~$H^*$.

We now start a case analysis of the relation of $\sigma$ and $\delta$
to the vertices $\ell_2^*,r_2^*$.
In the easiest instance, $\sigma_2$ includes neither $\ell_2^*$ nor~$r_2^*$.
Then $\sigma=\sigma_2$ (regardless of~$\delta$), 
and by the way the bipolar sign assignment on 
$H_2^*$ is inherited from the bipolar sign assignment on $H^*$, it follows 
that the $\delta_2$-polarity of each of the half-edges $h_u^*$ and $h_v^*$
equals its $\delta$-polarity.
Therefore, $\sigma$ itself is a $\delta$-switching ear of the same length,
and we are done by setting $\varrho:=\sigma$.

The situation is similarly easy when $\sigma_2$ contains one or both of
$\ell_2^*,r_2^*$, but $\delta_2$ includes neither $\ell_2^*$ nor~$r_2^*$.
Then $\delta$ is disjoint from $\alpha$ in $H^*$ and
so the ends $u^*,v^*$ are not on~$\alpha$.
Consequently, unless $\sigma$ itself is a $\delta$-switching ear as previously,
we can reconnect the two components of $\sigma$ incident to $\delta$
(a possible third component of $\sigma$ can be safely ignored here)
with a shorter subpath of $\alpha$.
This results in a $\delta$-switching ear $\varrho$ 
of length at most $\len(\sigma_2)+\frac12\sizealpha$.

In the rest of the proof we may thus assume that each of $\sigma_2$ and
$\delta_2$ includes at least one of $\ell_2^*,r_2^*$ (not necessarily the
same one).
This in particular means that $\delta$ intersects $\alpha$ in~$H^*$.
It will be useful in the upcoming arguments to notice that the meaning
of $\delta$-polarity can be consistently {\em extended}
to any half-edge incident to $V(\alpha)\sem V(\delta)$.
Indeed, any dual vertex $w^*\in V(\alpha)\sem V(\delta)$ belongs to some
$\delta$-ear $\pi\subseteq\alpha$, and the ends of $\pi$ are of the same
$\delta$-polarity by Remark~\ref{rem:noswitchingalpha}.
We assign this $\delta$-polarity value (of the ends of~$\pi$)
to any half-edge incident with~$w^*$.

Our more detailed strategy for finishing the proof of (d') is to find a
suitable $(\delta\cup\alpha)$-ear $\vartheta\subseteq\sigma$ in~$H^*$, 
and apply the following claim to it (which readily implies (d')):

\begin{claim}\label{clm:extpolar}
Let $\vartheta$ be a $(\delta\cup\alpha)$-ear such that the end-half-edges 
of $\vartheta$ are of distinct extended $\delta$-polarity.
Then there exists a $\delta$-switching ear $\varrho\supseteq\vartheta$ of
length at most $\len(\vartheta)+\frac12\sizealpha$.
\end{claim}
\begin{proof}\useInnerQed
Let $x^*,y^*$ be the ends of $\vartheta$.
At least one of the ends, say $x^*$, belongs to $V(\alpha)\sem V(\delta)$,
or else we are trivially done.
Let $\pi_x\subseteq\alpha$ denote the $\delta$-ear on $\alpha$ that $x^*$
belongs to.
We choose a shortest subpath $\pi_x'\subseteq\pi_x$ from $x^*$ to $\delta$.
If $y^*\in V(\alpha)\sem V(\delta)$, too, then we analogously find
$\pi_y$ and $\pi_y'$ and we remark that $\pi_y\not=\pi_x$ since the ends of
$\vartheta$ are of distinct extended $\delta$-polarity.
Otherwise, we set $\pi_y=\pi_y':=\emptyset$.
Clearly, $\varrho:=\pi_x'\cup\vartheta\cup\pi_y'$ is a $\delta$-switching ear.
Furthermore, $\len(\pi_x')+\len(\pi_y')\leq
 \frac12\len(\pi_x)+\frac12\len(\pi_y)\leq\frac12\len(\alpha)$.
\end{proof}

Based on the assumption that each of $\sigma_2$ and  
$\delta_2$ includes at least one of $\ell_2^*,r_2^*$,
the following complete case analysis will be considered in finishing the proof:

(i) $V(\sigma_2)\cap V(\delta_2)\supseteq \{\ell_2^*,r_2^*\}$,

(ii) $V(\sigma_2)\cap \{\ell_2^*,r_2^*\}=\{\ell_2^*\}$,
up to symmetry between $\ell_2^*$ and $r_2^*$, and $\ell_2^*\in V(\delta_2)$,

(iii) $V(\delta_2)\cap \{\ell_2^*,r_2^*\}=\{\ell_2^*\}$,
up to symmetry between $\ell_2^*$ and $r_2^*$, and $r_2^*\in V(\sigma_2)$.

\smallskip
In case (i), since $\sigma_2$ is a $\delta_2$-ear, we get that 
the ends $u^*$ and $v^*$ of $\sigma_2$ are $\ell_2^*$ and $r_2^*$
(in either order), and so $\sigma$ is a $(\delta\cup\alpha)$-ear.
By the way the bipolar sign assignment on $H_2^*$ is inherited from the one
on $H^*$, we see that the mutually distinct $\delta_2$-polarity values of $h_u^*$ and $h_v^*$
are the same as their extended $\delta$-polarities.
Hence Claim~\ref{clm:extpolar} applies to $\vartheta:=\sigma$.

Case (ii) is similar to previous case (i); we again get that $\ell_2^*$ is one of the
ends of $\sigma_2$, and so $\sigma$ is a $(\delta\cup\alpha)$-ear.
Hence, analogously, Claim~\ref{clm:extpolar} applies to $\vartheta:=\sigma$.

Case (iii) is the most interesting one.
Here $r_2^*\not\in V(\delta_2)$ is not an end of $\sigma_2$, and so $\sigma_2$ 
consists of two subpaths $\sigma_2'$ and $\sigma_2''$ sharing $r_2^*$.
Their lifts $\sigma'$ and $\sigma''$ in~$H^*$ are each a $(\delta\cup\alpha)$-ear.
Note that $\ell_2^*$ may be one of the ends of $\sigma_2$,
if $V(\sigma_2)\supseteq \{\ell_2^*,r_2^*\}$, but this does not harm our arguments.
Since $r_2^*\not\in V(\delta_2)$, the ends of $\sigma'$ and $\sigma''$ on
$\alpha$ are of the same extended $\delta$-polarity (though they need not
belong to the same $\delta$-ear of $\alpha$).
Now, since the $\delta$-polarities (this time not extended) of the other ends 
$h_u^*,h_v^*$ of $\sigma'$ and $\sigma''$ are mutually distinct,
one of $\sigma'$, $\sigma''$ has ends of distinct extended
$\delta$-polarity. We thus again finish by Claim~\ref{clm:extpolar}.

\section{Removing the density requirement}
\label{sec:nodensity}

Our algorithmic technique proving Theorem~\ref{thm:main-dense} in Section~\ref{sec:drawing-upper} 
starts with a graph on a nonplanar surface, and brings the
graph to the plane without introducing too many crossings.  
As mentioned before, focusing only on surface surgery in this planarization
process would inevitably require a certain lower bound on the density 
of the original embedding.  However, we can
naturally combine this algorithm with some previous algorithmic results on inserting a
\emph{small} number of edges into a planar graph, to obtain a polynomial
algorithm with essentially the same approximation ratio but without the
density requirement.
This combined approach can be sketched as follows:
\begin{enumerate}
\item
As long as the embedding density requirement of Theorem~\ref{thm:main-dense} is
violated, we cut the surface along the violating loops.
Let $K\subseteq E(G)$ be the set of edges affected by this;
we know that $|K|$ is small, bounded by a function of $g$ and $\Delta(G)$.
Let $G_K:=G-K$.
\item
By Theorem~\ref{thm:upper-cr}, applied to $G_K$, 
we obtain a suitable set $F\subseteq E(G_K)$ such that $G_{KF}:=G_K-F$ is plane.
\item
In \cite{CH17}, we have developed fine-tuned efficient methods of dealing with
insertion of multiple edges into a planar graph, such that the resulting
crossing number is not too high.
Here, we would like to apply those methods to insert the edges of $K$
(a small set) back to $G_{KF}$, 
and simultaneously apply Theorem~\ref{thm:upper-cr} to insert $F$
(a rather large set compared to $K$) again to $G_{KF}$.
The number of possibly arising mutual crossings $|F|\cdot|K|$ is negligible,
but the real trouble is that the methods of \cite{CH17} are generally 
allowed to change the planar embedding of $G_{KF}$ with the insertion of~$K$, 
and hence the insertion routes assumed by Theorem~\ref{thm:upper-cr} 
may no longer exist after the application.
Fortunately, the number of the insertion routes for $F$
is bounded in the genus (unlike $|F|$), and so the methods of \cite{CH17} 
can be adapted to respect these computed routes for $F$ without
a big impact on its approximation ratio.
\end{enumerate}

Unfortunately, turning this simple sketch into a formal proof would not be short,
due to the necessity to bring up many technical definitions and fine algorithmic 
details from \cite{CH17}.
That is why we consider another option, allowing a short self-contained proof
at the expense of giving a weaker approximation guarantee.
We use the following simplified formulation of the main result of
\cite{CH17}.
For a graph $H$ and a set of edges $K$ with ends in $V(H)$, but $K\cap
E(H)=\emptyset$, let $H+K$ denote the graph obtained by adding
the edges $K$ into $H$.

\begin{theorem}[Chimani and Hlin\v{e}n\'y \cite{CH17}]
\label{thm:apxmei}
Let $H$ be a connected planar graph with maximum degree~$\Delta$, 
$K$ an edge set with ends in $V(H)$ but $K\cap E(G)=\emptyset$.
There is a polynomial-time algorithm that finds a drawing of $H+K$ in the plane
with at most $d\cdot\crg(H+K)$ crossings, where $d$ is a constant depending 
only on $\Delta$ and $|K|$ (more precisely, linear in $\Delta$ and quadratic
in~$|K|$).
In this drawing, subgraph $H$ is drawn planarly, i.e., all crossings involve at least one edge of $K$.
\end{theorem}

To complete the proof of the main result, we will now prove the theorem below, which
is a reformulation of Theorem~\ref{thm:main-overview}(b).

\begin{theorem}
\label{thm:main-nodense}
Let $g>0$ and $\Delta$ be integer constants.
Assume $G$ is a graph of maximum degree $\Delta$ embeddable in $\Sigma_g$.
There is a polynomial time algorithm that outputs a drawing of
$G$ in the plane with at most $c_2\cdot\crg(G)$ crossings,
where $c_2$ is a constant depending on $g$ and~$\Delta$.
\end{theorem}

\begin{proof}
Let $c_2'$, depending on $g$ and $\Delta$, be as in Theorem~\ref{thm:main-dense}.
Note that, since our estimates are only asymptotic, we may always assume $\Delta\geq4$.
Let $r_0=5\cdot2^{g-1}\lfloor\Delta/2\rfloor$.
Since $r_0$ is nondecreasing in $g$, we may just fix it for the rest of the proof
(in which we are going to possibly decrease the genus).
We may assume that $G$ is connected, since otherwise we split the problem
into subproblems on the connected components.
If $\ewnd(G)<r_0$, let $\gamma$ be the witnessing dual cycle of~$G$.
We cut $G$ along $\gamma$, and repeat this operation until we arrive at an
embedded graph $G_K\subseteq G$ of genus $g_K<g$ such that
$\ewnd(G_K)\geq r_0$.
Let $K=E(G)\sem E(G_K)$ be the severed edges, where $|K|\leq gr_0$
is bounded.

If $g_K=0$, then we simply finish by applying Theorem~\ref{thm:apxmei}.
Otherwise, we apply Theorem~\ref{thm:main-dense} to $G_K$ in $\Sigma_{g_K}$
(the surface of genus~$g_K$), 
which yields a drawing $G^\circ_K$ of $G_K$ in the plane with $k\leq
c_2'\cdot\crg(G_K)$ crossings.
Let $F\subseteq E(G_K)$ be the subset of edges that are crossed in $G^\circ_K$.
In the drawing $G^\circ_K$ we now replace each crossing by a new subdividing vertex. 
This gives a planarly embedded graph $G_K'$ that contains a planarly embedded 
subdivision $G_{KF}'$ of $G_K-F$.
Let $F_2=E(G_K')\sem E(G_{KF}')$.
By simple counting of edges incident to the new vertices replacing
crossings in $G^\circ_K$, we get $|F_2|\leq 4k\leq 4c_2'\cdot\crg(G_K)$.
(In fact, using further arguments explicitly considering the planarizing
edges computed by the algorithm of Theorem~\ref{thm:upper-cr} as $F$, this
inequality can be improved to $|F_2|\leq2k$.)

Now we apply Theorem~\ref{thm:apxmei} to $H=G_{KF}'$ and $K$ (from the
beginning of the proof).
This gives us a factor~$d$ and a drawing $G_F'$ of $G_{KF}'+K$ with at most 
$d\cdot\crg(G_{KF}'+K)$ crossings in the plane.
The final task is to put back the edges of $F_2$ into $G_F'$;
note, however, that the planar subembedding of $G_{KF}'$ within $G_F'$ 
is generally different from the original embedding of $G_{KF}'$ within~$G_K'$.

For the latter task we use the following technical claim:
\begin{claim}[Hlin\v{e}n\'y and Salazar~{\cite[Lemma~2.4]{HS06}}, see also
	{\cite[Lemma~5]{CHM12}}]\label{cl:edgeback}
Suppose that $H$ is a connected graph embedded in the plane,
and $f\not\in E(H)$ is an edge joining vertices of $H$ such that 
$H+f$ is a planar graph (not necessarily using the same embedding~$H$).
Then there exists a planar embedding $H_0$ of $H+f$ such that the following
holds for any edge $e\not\in E(H)$ joining vertices of~$H$:
If $e$ can be drawn in $H$ with $\ell$ crossings, then $e$ can be drawn in $H_0$ 
with at most $\ell+2\cdot\lfloor\Delta(H)/2\rfloor$ crossings.
\end{claim}
\noindent
We remark that the original formulation of \cite[Lemma~2.4]{HS06} did not
state that the planar drawing $H_0$ is the same one for all possibly added
edges $e$, but this extension follows already from the proof in \cite{HS06}.
Moreover, Claim~\ref{cl:edgeback} is a special case of
\cite[Lemma~5]{CHM12} (for~$d=2$ there).

To proceed with our proof,
assume that every edge $e\in K$ can be drawn in $G_{KF}'$ with
$\ell_e$ crossings, and let $F_2=\{f_1,f_2,\dots,f_a\}$.
By induction on $i=0,1,\dots,a=|F_2|$,
we show that there is a suitable planar embedding $H_i$ of 
$G_{KF}'+f_1+\dots+f_{i}$, such that every $e\in K$ can be drawn in $H_i$
with at most $\ell_e+2i\cdot\lfloor\Delta(H_{i-1})/2\rfloor$ crossings.
This is trivial for $i=0$, and the induction step is immediate from an
application of Claim~\ref{cl:edgeback} to $H=H_{i-1}$, $f=f_i$, and
$\ell=\ell_e+2(i-1)\cdot\lfloor\Delta(H_{i-2})/2\rfloor$.
Note that the assumptions of Claim~\ref{cl:edgeback} are
satisfied at each step since $G_{KF}'+F_2$ is planar.

We remark that the resulting planar embedding $H_a$ of $G_K'$ may be different
from the embedding of $G_K'$ we started with above.

We now count all crossings on the edges of $K$.
We have $\Delta(H_{a})\leq\Delta$ since $\Delta\geq4$ and $G_K'$ has 
vertex degrees bounded by those of $G$ besides the added crossing vertices.
By the previous inductive argument, every $e\in K$ can be (independently) drawn in $H_a$
with at most $\ell_e+2|F_2|\cdot\lfloor\Delta/2\rfloor$ crossings.
Recall also that $\sum_{e\in K}\ell_e\leq d\cdot\crg(G_{KF}'+K)$
by Theorem~\ref{thm:apxmei}.
Then there are at most $|K|^2/2$ possible crossings between pairs of edges from~$K$.
Altogether, all the edges of $K$ can be drawn (simultaneously) into planar $H_a$
with at most $\sum_{e\in K}\ell_e+2a\cdot{\lfloor\Delta/2\rfloor}\cdot|K|+|K|^2/2
 \leq d\cdot\crg(G_{KF}'+K)+2|F_2|\cdot{\lfloor\Delta/2\rfloor}\cdot|K|+|K|^2/2$
crossings.

The final task is to algorithmically construct the planar embedding $H_a$
as above from starting $G_{KF}'$, and then to draw the edges of $K$ into $H_a$.
The latter can be solved in quadratic time by a breadth-first search 
in the dual of $H_a$, applied to each edge of $K$ separately.
and postprocessing of multiple crossings as in the proof of Theorem~\ref{thm:upper-cr}, 
For the former, a construction of the embedding $H_a$,
one can check that all the steps of the proofs of {\cite[Lemma~2.4]{HS06}}
and {\cite[Lemma~5]{CHM12}} are constructive and can be done in polynomial time.
Furthermore, there is a linear-time algorithm for inserting an edge into a
planar graph \cite{GMW05} which, 
applied to $H=H_{i-1}$ and $f=f_i$, achieves exactly the same planar embedding
as in the proof of Claim~\ref{cl:edgeback}.

By turning the vertices of $V(G_K')\sem V(G_K)$ back into edge crossings of
$G_K$, the whole procedure leads to a drawing of $G_K+K=G$ 
with at most the number of crossings
\begin{align*}
 c_2'&\cdot\crg(G_K) + d\cdot\crg(G_{KF}'+K) + 2\deee\cdot|K|\cdot|F_2| +|K|^2/2 \\
 & \leq c_2'\cdot\crg(G) + d\cdot\crg(G) + \Delta\cdot gr_0\cdot4c_2'\crg(G)
   	+(gr_0)^2/2\\
 & \leq (c_2' + d + 4\Delta gr_0c_2')\cdot\crg(G) +g^2r_0^2/2
\,.\end{align*}
Since $\crg(G)\geq1$, it suffices to set
$c_2=c_2'+d+4\Delta gr_0c_2'+g^2r_0^2/2$
which is a constant depending only on $g$ and~$\Delta$.
For the sake of completeness, we remark the asymptotically dominating term
in the expression for $c_2$ is $\Delta gr_0c_2'$ which grows with $\Delta^4$
and $16^g$.
\end{proof}

\section{Concluding remarks}
\label{sec:concluding}

There are several natural questions that arise in connection with our
research.

\paragraph{Extension to nonorientable surfaces.}
One can wonder whether our results, namely about approximating planar crossing
number of an embedded graph,
can also be extended to nonorientable surfaces of higher genus. 
Indeed, the upper-bound result of \cite{BPT06} holds for any surface, and
there is an algorithm to approximate the crossing number for graphs embeddable in the projective plane~\cite{GHLS08}.
We currently do not see any reason why such an extension would be impossible. 

However, the individual steps
become much more difficult to analyze and tie together, since the ``cheapest'' cut through the embedding 
can cut (a) a handle along a two-sided loop, (b) a twisted handle along a two-sided loop, or (c) a crosscap along a one-sided loop.
Hence it then does not suffice to consider toroidal grids as the sole base case
(and a usable definition of ``nonorientable stretch'' should reflect this), 
but the lower bound may also arise from a projective or Klein-bottle grid minor. 
Already for the latter, 
there are currently no non-trivial results known.
We thus leave this direction for future investigation.

\paragraph{Dependency of the constants in Theorem~\ref{thm:main-overview} on $\Delta$ and $g$.}
Taking a toroidal grid with sufficiently multiplied parallel edges
(possibly subdividing them to obtain a simple graph) easily shows that a relation between
the toroidal expanse and the crossing number must involve a factor of
$\Delta(G)^2$.
Regarding an efficient approximation algorithm for the crossing number,
avoiding a general dependency on the maximum degree seems also difficult---%
as various related algorithmic approximation results for the crossing number, e.g.,
\cite{BPT06,CM11,CHM12,CH17,DV12,GHLS08}, depend on the maximum degree, too.
However, considering the so-called minor crossing number (see
below), one can avoid this dependency at least in a special case of the torus.

The exponential dependency of the constants and the approximation ratio on $g$, 
on the other hand, is very interesting. It pops up independently in
multiple places within the proofs, and these occurrences seem unavoidable on a local scale,
when considering each inductive step independently. 
However, it seems very hard to construct any example
where such an exponential jump or decrease can actually be observed. It might be that a 
different approach with a global view can reduce the dependency 
in Theorem~\ref{thm:main-overview} to some $\mathit{poly}(g)$ factor, cf.
also~\cite{DV12}.

\paragraph{Toroidal grids and minor crossing number.}

The \emph{minor crossing number} $\mcr(G)$~\cite{BFM06} is the smallest
crossing number over all graphs $H$ that
have $G$ as a minor. Hence it is, by definition and in contrast to the
traditional crossing number, a well-behaved minor-monotone parameter.
In general, however, minor crossing number is not any easier to compute~\cite{Hl06}
than ordinary crossing number.

One can, perhaps, build a similar theory as we did in this paper,
with face-width, the minor crossing number and the so called ``face stretch''
(which is a natural counterpart of stretch in this context).
The immediate advantage of such approach would be in removing the dependency
on the maximum degree $\Delta$, which we have discussed just above.

In fact, by adapting the techniques of our paper to this new setting we are able
to prove that the toroidal expanse of a {\em toroidal} graph $G$ is within constant
factor lower and upper bounds of the minor crossing number of~$G$
(independently of~$\Delta(G)$).
Consequently, the minor crossing number can be efficiently approximated 
for toroidal graphs up to a constant factor independent of~$\Delta(G)$.
However, the a priori unexpected technical complications surrounding this
adaptation seem to make it hardly extendable to higher-genus surfaces.
We thus abandon this line of potential research.

\subsection*{Acknowledgments}

We would like to thank the anonymous referee for extensive comments and
suggestions which helped us to improve the quality and readability of this
paper.

\end{document}

%% file: 3dtorusX.pdf_tex
\begingroup%
  \makeatletter%
  \providecommand\color[2][]{%
    \errmessage{(Inkscape) Color is used for the text in Inkscape, but the package 'color.sty' is not loaded}%
    \renewcommand\color[2][]{}%
  }%
  \providecommand\transparent[1]{%
    \errmessage{(Inkscape) Transparency is used (non-zero) for the text in Inkscape, but the package 'transparent.sty' is not loaded}%
    \renewcommand\transparent[1]{}%
  }%
  \providecommand\rotatebox[2]{#2}%
  \ifx\svgwidth\undefined%
    \setlength{\unitlength}{208.80078894bp}%
    \ifx\svgscale\undefined%
      \relax%
    \else%
      \setlength{\unitlength}{\unitlength * \real{\svgscale}}%
    \fi%
  \else%
    \setlength{\unitlength}{\svgwidth}%
  \fi%
  \global\let\svgwidth\undefined%
  \global\let\svgscale\undefined%
  \makeatother%
  \begin{picture}(1,0.43431045)%
    \put(0,0){\includegraphics[width=\unitlength,page=1]{3dtorusX.pdf}}%
    \put(0.37903081,0.00056124){\color[rgb]{0,0,1}\makebox(0,0)[lb]{\smash{  $\alpha$}}}%
    \put(0.82570661,0.16768474){\color[rgb]{0.76470588,0,0}\makebox(0,0)[lb]{\smash{$\beta$}}}%
  \end{picture}%
\endgroup%

%% file: repairAtX.pdf_tex
\begingroup%
  \makeatletter%
  \providecommand\color[2][]{%
    \errmessage{(Inkscape) Color is used for the text in Inkscape, but the package 'color.sty' is not loaded}%
    \renewcommand\color[2][]{}%
  }%
  \providecommand\transparent[1]{%
    \errmessage{(Inkscape) Transparency is used (non-zero) for the text in Inkscape, but the package 'transparent.sty' is not loaded}%
    \renewcommand\transparent[1]{}%
  }%
  \providecommand\rotatebox[2]{#2}%
  \ifx\svgwidth\undefined%
    \setlength{\unitlength}{133.82462509bp}%
    \ifx\svgscale\undefined%
      \relax%
    \else%
      \setlength{\unitlength}{\unitlength * \real{\svgscale}}%
    \fi%
  \else%
    \setlength{\unitlength}{\svgwidth}%
  \fi%
  \global\let\svgwidth\undefined%
  \global\let\svgscale\undefined%
  \makeatother%
  \begin{picture}(1,0.76088104)%
    \put(0,0){\includegraphics[width=\unitlength,page=1]{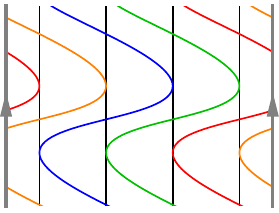}}%
    \put(0.12,-0.06){\makebox(0,0)[lb]{\smash{$C_{i-2}$}}}%
    \put(0.36,-0.06){\makebox(0,0)[lb]{\smash{$C_{i-1}$}}}%
    \put(0.6,-0.06){\makebox(0,0)[lb]{\smash{$C_i$}}}%
    \put(0.82,-0.06){\makebox(0,0)[lb]{\smash{$C_{i+1}$}}}%
    \put(0.69845628,0.12667884){\color[rgb]{1,0,0}\makebox(0,0)[lb]{\smash{$D'_1$}}}%
    \put(0.45845628,0.12667884){\color[rgb]{0,0.75294118,0}\makebox(0,0)[lb]{\smash{$D'_2$}}}%
    \put(0.21933736,0.12667884){\color[rgb]{0,0,1}\makebox(0,0)[lb]{\smash{$D'_3$}}}%
  \end{picture}%
\endgroup%

%% file: repairBtX.pdf_tex
\begingroup%
  \makeatletter%
  \providecommand\color[2][]{%
    \errmessage{(Inkscape) Color is used for the text in Inkscape, but the package 'color.sty' is not loaded}%
    \renewcommand\color[2][]{}%
  }%
  \providecommand\transparent[1]{%
    \errmessage{(Inkscape) Transparency is used (non-zero) for the text in Inkscape, but the package 'transparent.sty' is not loaded}%
    \renewcommand\transparent[1]{}%
  }%
  \providecommand\rotatebox[2]{#2}%
  \ifx\svgwidth\undefined%
    \setlength{\unitlength}{133.82462509bp}%
    \ifx\svgscale\undefined%
      \relax%
    \else%
      \setlength{\unitlength}{\unitlength * \real{\svgscale}}%
    \fi%
  \else%
    \setlength{\unitlength}{\svgwidth}%
  \fi%
  \global\let\svgwidth\undefined%
  \global\let\svgscale\undefined%
  \makeatother%
  \begin{picture}(1,0.76088104)%
    \put(0,0){\includegraphics[width=\unitlength,page=1]{repairBtX.pdf}}%
    \put(0.68318487,0.45546736){\color[rgb]{0,0,0}\makebox(0,0)[lb]{\smash{$R_0$}}}%
    \put(0.45845628,0.12667884){\color[rgb]{0,0.75294118,0}\makebox(0,0)[lb]{\smash{$D'_2$}}}%
    \put(0.21933736,0.12667884){\color[rgb]{0,0,1}\makebox(0,0)[lb]{\smash{$D'_3$}}}%
    \put(0,0){\includegraphics[width=\unitlength,page=2]{repairBtX.pdf}}%
  \end{picture}%
\endgroup%

%% file: repairHomotopyAX.pdf_tex
\begingroup%
  \makeatletter%
  \providecommand\color[2][]{%
    \errmessage{(Inkscape) Color is used for the text in Inkscape, but the package 'color.sty' is not loaded}%
    \renewcommand\color[2][]{}%
  }%
  \providecommand\transparent[1]{%
    \errmessage{(Inkscape) Transparency is used (non-zero) for the text in Inkscape, but the package 'transparent.sty' is not loaded}%
    \renewcommand\transparent[1]{}%
  }%
  \providecommand\rotatebox[2]{#2}%
  \ifx\svgwidth\undefined%
    \setlength{\unitlength}{133.82463074bp}%
    \ifx\svgscale\undefined%
      \relax%
    \else%
      \setlength{\unitlength}{\unitlength * \real{\svgscale}}%
    \fi%
  \else%
    \setlength{\unitlength}{\svgwidth}%
  \fi%
  \global\let\svgwidth\undefined%
  \global\let\svgscale\undefined%
  \makeatother%
  \begin{picture}(1,0.73099128)%
    \put(0,0){\includegraphics[width=\unitlength,page=1]{repairHomotopyAX.pdf}}%
    \put(0.10398847,0.73037665){\color[rgb]{0,0,0}\makebox(0,0)[lb]{\smash{$C_4$}}}%
    \put(0.34310741,0.73037665){\color[rgb]{0,0,0}\makebox(0,0)[lb]{\smash{$C_1$}}}%
    \put(0.58222626,0.7299826){\color[rgb]{0,0,0}\makebox(0,0)[lb]{\smash{$C_2$}}}%
    \put(0.8213452,0.7299826){\color[rgb]{0,0,0}\makebox(0,0)[lb]{\smash{$C_3$}}}%
    \put(1.00068431,0.2824227){\color[rgb]{1,0,0}\makebox(0,0)[lb]{\smash{$D'_1$}}}%
    \put(1.00068431,0.3720923){\color[rgb]{0,0,1}\makebox(0,0)[lb]{\smash{$D'_2$}}}%
    \put(1.00068431,0.4617619){\color[rgb]{0,0.75294118,0}\makebox(0,0)[lb]{\smash{$D'_3$}}}%
    \put(0,0){\includegraphics[width=\unitlength,page=2]{repairHomotopyAX.pdf}}%
  \end{picture}%
\endgroup%

%% file: repairHomotopyBX.pdf_tex
\begingroup%
  \makeatletter%
  \providecommand\color[2][]{%
    \errmessage{(Inkscape) Color is used for the text in Inkscape, but the package 'color.sty' is not loaded}%
    \renewcommand\color[2][]{}%
  }%
  \providecommand\transparent[1]{%
    \errmessage{(Inkscape) Transparency is used (non-zero) for the text in Inkscape, but the package 'transparent.sty' is not loaded}%
    \renewcommand\transparent[1]{}%
  }%
  \providecommand\rotatebox[2]{#2}%
  \ifx\svgwidth\undefined%
    \setlength{\unitlength}{133.82463074bp}%
    \ifx\svgscale\undefined%
      \relax%
    \else%
      \setlength{\unitlength}{\unitlength * \real{\svgscale}}%
    \fi%
  \else%
    \setlength{\unitlength}{\svgwidth}%
  \fi%
  \global\let\svgwidth\undefined%
  \global\let\svgscale\undefined%
  \makeatother%
  \begin{picture}(1,0.73099128)%
    \put(0,0){\includegraphics[width=\unitlength,page=1]{repairHomotopyBX.pdf}}%
    \put(0.4080853,0.11698454){\color[rgb]{0,0,0}\makebox(0,0)[lb]{\smash{$x_1$}}}%
    \put(0.41155081,0.36736637){\color[rgb]{0,0,0}\makebox(0,0)[lb]{\smash{$y_1$}}}%
    \put(0,0){\includegraphics[width=\unitlength,page=2]{repairHomotopyBX.pdf}}%
    \put(0.75615476,0.19320589){\color[rgb]{1,0,0}\makebox(0,0)[lb]{\smash{$T_1$}}}%
    \put(0.25574011,0.30509952){\color[rgb]{0,0,0}\makebox(0,0)[lb]{\smash{$W_1$}}}%
    \put(0,0){\includegraphics[width=\unitlength,page=3]{repairHomotopyBX.pdf}}%
  \end{picture}%
\endgroup%

%% file: repairHomotopyCX.pdf_tex
\begingroup%
  \makeatletter%
  \providecommand\color[2][]{%
    \errmessage{(Inkscape) Color is used for the text in Inkscape, but the package 'color.sty' is not loaded}%
    \renewcommand\color[2][]{}%
  }%
  \providecommand\transparent[1]{%
    \errmessage{(Inkscape) Transparency is used (non-zero) for the text in Inkscape, but the package 'transparent.sty' is not loaded}%
    \renewcommand\transparent[1]{}%
  }%
  \providecommand\rotatebox[2]{#2}%
  \ifx\svgwidth\undefined%
    \setlength{\unitlength}{133.82463074bp}%
    \ifx\svgscale\undefined%
      \relax%
    \else%
      \setlength{\unitlength}{\unitlength * \real{\svgscale}}%
    \fi%
  \else%
    \setlength{\unitlength}{\svgwidth}%
  \fi%
  \global\let\svgwidth\undefined%
  \global\let\svgscale\undefined%
  \makeatother%
  \begin{picture}(1,0.73099128)%
    \put(0,0){\includegraphics[width=\unitlength,page=1]{repairHomotopyCX.pdf}}%
    \put(0.4080853,0.11698454){\color[rgb]{0,0,0}\makebox(0,0)[lb]{\smash{$x_1$}}}%
    \put(0,0){\includegraphics[width=\unitlength,page=2]{repairHomotopyCX.pdf}}%
    \put(0.28332767,0.30633459){\color[rgb]{0,0,0}\makebox(0,0)[lb]{\smash{$x_3$}}}%
    \put(0.27064336,0.23646757){\color[rgb]{0,0,0}\makebox(0,0)[lb]{\smash{$x_2$}}}%
    \put(0,0){\includegraphics[width=\unitlength,page=3]{repairHomotopyCX.pdf}}%
    \put(-0.11305545,0.25629862){\color[rgb]{1,0,0}\makebox(0,0)[lb]{\smash{$D''_1$}}}%
    \put(-0.11305545,0.34596822){\color[rgb]{0,0,1}\makebox(0,0)[lb]{\smash{$D''_2$}}}%
    \put(-0.11305545,0.43563782){\color[rgb]{0,0.75294118,0}\makebox(0,0)[lb]{\smash{$D''_3$}}}%
    \put(0,0){\includegraphics[width=\unitlength,page=4]{repairHomotopyCX.pdf}}%
    \put(0.10398847,0.73037665){\color[rgb]{0,0,0}\makebox(0,0)[lb]{\smash{$C_4$}}}%
    \put(0.34310741,0.73037665){\color[rgb]{0,0,0}\makebox(0,0)[lb]{\smash{$C_1$}}}%
    \put(0.58222626,0.7299826){\color[rgb]{0,0,0}\makebox(0,0)[lb]{\smash{$C_2$}}}%
    \put(0.8213452,0.7299826){\color[rgb]{0,0,0}\makebox(0,0)[lb]{\smash{$C_3$}}}%
    \put(0,0){\includegraphics[width=\unitlength,page=5]{repairHomotopyCX.pdf}}%
  \end{picture}%
\endgroup%

%% file: polaritycutA.pdf_tex
\begingroup%
  \makeatletter%
  \providecommand\color[2][]{%
    \errmessage{(Inkscape) Color is used for the text in Inkscape, but the package 'color.sty' is not loaded}%
    \renewcommand\color[2][]{}%
  }%
  \providecommand\transparent[1]{%
    \errmessage{(Inkscape) Transparency is used (non-zero) for the text in Inkscape, but the package 'transparent.sty' is not loaded}%
    \renewcommand\transparent[1]{}%
  }%
  \providecommand\rotatebox[2]{#2}%
  \ifx\svgwidth\undefined%
    \setlength{\unitlength}{304.43465586bp}%
    \ifx\svgscale\undefined%
      \relax%
    \else%
      \setlength{\unitlength}{\unitlength * \real{\svgscale}}%
    \fi%
  \else%
    \setlength{\unitlength}{\svgwidth}%
  \fi%
  \global\let\svgwidth\undefined%
  \global\let\svgscale\undefined%
  \makeatother%
  \begin{picture}(1,0.30751311)%
    \put(0,0){\includegraphics[width=\unitlength,page=1]{polaritycutA.pdf}}%
    \put(0.24504922,0.03298504){\color[rgb]{1,0,0}\makebox(0,0)[lb]{\smash{  $\alpha$}}}%
    \put(0.21166763,0.22786614){\color[rgb]{0,0.75294118,0}\makebox(0,0)[lb]{\smash{  $\omega$}}}%
    \put(0.77085479,0.04843845){\color[rgb]{0,0,1}\makebox(0,0)[lb]{\smash{  $\gamma$}}}%
  \end{picture}%
\endgroup%

%% file: polaritycutB.pdf_tex
\begingroup%
  \makeatletter%
  \providecommand\color[2][]{%
    \errmessage{(Inkscape) Color is used for the text in Inkscape, but the package 'color.sty' is not loaded}%
    \renewcommand\color[2][]{}%
  }%
  \providecommand\transparent[1]{%
    \errmessage{(Inkscape) Transparency is used (non-zero) for the text in Inkscape, but the package 'transparent.sty' is not loaded}%
    \renewcommand\transparent[1]{}%
  }%
  \providecommand\rotatebox[2]{#2}%
  \ifx\svgwidth\undefined%
    \setlength{\unitlength}{304.43465586bp}%
    \ifx\svgscale\undefined%
      \relax%
    \else%
      \setlength{\unitlength}{\unitlength * \real{\svgscale}}%
    \fi%
  \else%
    \setlength{\unitlength}{\svgwidth}%
  \fi%
  \global\let\svgwidth\undefined%
  \global\let\svgscale\undefined%
  \makeatother%
  \begin{picture}(1,0.30751315)%
    \put(0,0){\includegraphics[width=\unitlength,page=1]{polaritycutB.pdf}}%
    \put(0.21166764,0.2278662){\color[rgb]{0,0.75294118,0}\makebox(0,0)[lb]{\smash{  $\omega$}}}%
  \end{picture}%
\endgroup%

%% file: X000.pdf_tex
\begingroup%
  \makeatletter%
  \providecommand\color[2][]{%
    \errmessage{(Inkscape) Color is used for the text in Inkscape, but the package 'color.sty' is not loaded}%
    \renewcommand\color[2][]{}%
  }%
  \providecommand\transparent[1]{%
    \errmessage{(Inkscape) Transparency is used (non-zero) for the text in Inkscape, but the package 'transparent.sty' is not loaded}%
    \renewcommand\transparent[1]{}%
  }%
  \providecommand\rotatebox[2]{#2}%
  \ifx\svgwidth\undefined%
    \setlength{\unitlength}{158.69803962bp}%
    \ifx\svgscale\undefined%
      \relax%
    \else%
      \setlength{\unitlength}{\unitlength * \real{\svgscale}}%
    \fi%
  \else%
    \setlength{\unitlength}{\svgwidth}%
  \fi%
  \global\let\svgwidth\undefined%
  \global\let\svgscale\undefined%
  \makeatother%
  \begin{picture}(1,0.48023244)%
    \put(0,0){\includegraphics[width=\unitlength,page=1]{X000.pdf}}%
    \put(0.48794243,0.35481145){\color[rgb]{1,0,0}\makebox(0,0)[lb]{\smash{  $\alpha$}}}%
  \end{picture}%
\endgroup%

%% file: X001.pdf_tex
\begingroup%
  \makeatletter%
  \providecommand\color[2][]{%
    \errmessage{(Inkscape) Color is used for the text in Inkscape, but the package 'color.sty' is not loaded}%
    \renewcommand\color[2][]{}%
  }%
  \providecommand\transparent[1]{%
    \errmessage{(Inkscape) Transparency is used (non-zero) for the text in Inkscape, but the package 'transparent.sty' is not loaded}%
    \renewcommand\transparent[1]{}%
  }%
  \providecommand\rotatebox[2]{#2}%
  \ifx\svgwidth\undefined%
    \setlength{\unitlength}{158.69803962bp}%
    \ifx\svgscale\undefined%
      \relax%
    \else%
      \setlength{\unitlength}{\unitlength * \real{\svgscale}}%
    \fi%
  \else%
    \setlength{\unitlength}{\svgwidth}%
  \fi%
  \global\let\svgwidth\undefined%
  \global\let\svgscale\undefined%
  \makeatother%
  \begin{picture}(1,0.48023244)%
    \put(0,0){\includegraphics[width=\unitlength,page=1]{X001.pdf}}%
    \put(0.52404315,0.23326138){\color[rgb]{1,0,0}\makebox(0,0)[lb]{\smash{  $j_1^*$}}}%
    \put(0.50716147,0.00981387){\color[rgb]{1,0,0}\makebox(0,0)[lb]{\smash{  $f_1^*$}}}%
  \end{picture}%
\endgroup%

%% file: X002.pdf_tex
\begingroup%
  \makeatletter%
  \providecommand\color[2][]{%
    \errmessage{(Inkscape) Color is used for the text in Inkscape, but the package 'color.sty' is not loaded}%
    \renewcommand\color[2][]{}%
  }%
  \providecommand\transparent[1]{%
    \errmessage{(Inkscape) Transparency is used (non-zero) for the text in Inkscape, but the package 'transparent.sty' is not loaded}%
    \renewcommand\transparent[1]{}%
  }%
  \providecommand\rotatebox[2]{#2}%
  \ifx\svgwidth\undefined%
    \setlength{\unitlength}{158.69803962bp}%
    \ifx\svgscale\undefined%
      \relax%
    \else%
      \setlength{\unitlength}{\unitlength * \real{\svgscale}}%
    \fi%
  \else%
    \setlength{\unitlength}{\svgwidth}%
  \fi%
  \global\let\svgwidth\undefined%
  \global\let\svgscale\undefined%
  \makeatother%
  \begin{picture}(1,0.48023244)%
    \put(0,0){\includegraphics[width=\unitlength,page=1]{X002.pdf}}%
    \put(0.45069529,0.22317256){\color[rgb]{1,0,0}\makebox(0,0)[lb]{\smash{  $\ell_2^*$}}}%
    \put(0.51151806,0.13317278){\color[rgb]{1,0,0}\makebox(0,0)[lb]{\smash{  $r_2^*$}}}%
    \put(0,0){\includegraphics[width=\unitlength,page=2]{X002.pdf}}%
  \end{picture}%
\endgroup%